\documentclass[11pt,letterpaper]{amsart}

\usepackage[T1]{fontenc}
\usepackage{lmodern}

\usepackage{amsmath, amsthm, amssymb, amsfonts, mathtools}
\usepackage{mathrsfs}
\usepackage[mathcal]{eucal}
\usepackage{bm}
\usepackage{bbm}
\usepackage{comment}

\usepackage{graphicx,float}
\usepackage{tikz-cd}
\usepackage[all]{xy}
\usepackage{wasysym}

\usepackage{longtable,array}
\usepackage{enumitem}

\usepackage{cite}

\usepackage{tikz}
\usetikzlibrary{tikzmark}

\usepackage{hyperref}
\hypersetup{
    colorlinks=true,
    linkcolor=blue,
    filecolor=blue,
    urlcolor=blue,
    citecolor=red,
}

\numberwithin{equation}{section}

\newtheorem{theorem}{Theorem}[section]

\theoremstyle{definition}

\theoremstyle{remark}
\newtheorem{remark}[theorem]{Remark}

\numberwithin{equation}{section}

\theoremstyle{plain}
\newtheorem{thmx}{Theorem}
\newtheorem{corx}[thmx]{Corollary}

\renewcommand{\arraystretch}{1.1}

\begin{document}

\title{Artin twists of Drinfeld modules and Goss $L$-series}

\author{Jing Ye}
\address{Department of Mathematics, Texas A\&M University,
College Station, Texas, 77843, United States}
\email{yej@tamu.edu}
\date{\today}
\dedicatory{}

\commby{}

\theoremstyle{definition}
\newtheorem{Def}{Definition}[section]
\newtheorem{eg}[Def]{Example}
\newtheorem*{ex}{Exercise}
\newtheorem*{nota}{Notation}
\newtheorem{prob}[Def]{Problem}
\newtheorem*{conv}{Convention}
\newtheorem{asp}[Def]{Assumption}

\theoremstyle{theorem}
\newtheorem{thm}[Def]{Theorem}
\newtheorem{cor}[Def]{Corollary}
\newtheorem{prop}[Def]{Proposition}
\newtheorem{propdef}[Def]{Proposition-Definition}
\newtheorem{lem}[Def]{Lemma}
\newtheorem*{clm}{Claim}
\newtheorem{app}[Def]{Application}
\newtheorem{cjt}[Def]{Conjecture}
\newtheorem*{obs}{Observation}

\newtheorem{rev}{Review}
\newenvironment{pf}{{\noindent\it Proof.}\quad}{\hfill $\qed$\par}

\theoremstyle{remark}
\newtheorem{rmk}[Def]{\normalfont\textit{Remark}}

\newcommand{\rad}[1]{\operatorname{rad}(#1)}
\newcommand{\supp}[1]{\operatorname{supp}(#1)}
\newcommand{\spec}[1]{\operatorname{Spec}(#1)}
\newcommand{\Spec}{\operatorname{Spec}}
\newcommand{\spl}{\textnormal{spec}}
\newcommand{\mspec}[1]{\operatorname{MaxSpec}(#1)}
\newcommand{\ann}[2]{\operatorname{ann}_{#1}(#2)}
\newcommand{\ass}[1]{\operatorname{Ass}(#1)}
\newcommand{\assb}[1]{\operatorname{Ass}\Big(#1\Big)}

\newcommand{\z}[1]{\mathbb{Z}/#1\mathbb{Z}}
\newcommand{\dis}{\displaystyle}
\newcommand{\mor}[2]{\operatorname{Hom}({#1},{#2})}
\newcommand{\obj}[1]{\operatorname{obj}(#1)}
\newcommand{\homm}[3]{\operatorname{Hom}_{#1}({#2},{#3})}

\newcommand{\udl}[1]{\underline{#1}}
\newcommand{\sat}{\textnormal{sat}}

\newcommand{\fra}[1]{\operatorname{Frac}(#1)}
\newcommand{\im}{\operatorname{im}}
\newcommand{\conj}[1]{\overline{#1}}
\newcommand{\cok}{\operatorname{coker}}
\newcommand{\img}{\operatorname{Im}}
\newcommand{\coim}{\operatorname{coim}}

\newcommand{\inlim}{\varprojlim}
\newcommand{\dlim}{\varinjlim}

\newcommand{\eps}{\varepsilon}

\renewcommand{\AA}{\mathbb{A}}

\newcommand{\CC}{\mathbb{C}}
\newcommand{\EE}{\mathbb{E}}
\newcommand{\FF}{\mathbb{F}}
\newcommand{\GG}{\mathbb{G}}
\newcommand{\HH}{\mathbb{H}}
\newcommand{\II}{\mathbb{I}}
\newcommand{\MM}{\mathbb{M}}
\newcommand{\NN}{\mathbb{N}}
\newcommand{\PP}{\mathbb{P}}
\newcommand{\QQ}{\mathbb{Q}}
\newcommand{\RR}{\mathbb{R}}
\renewcommand{\SS}{\mathbb{S}}
\newcommand{\KK}{\mathbb{K}}
\newcommand{\LL}{\mathbb{L}}
\newcommand{\TT}{\mathbb{T}}
\newcommand{\ZZ}{\mathbb{Z}}
\newcommand{\bfm}{\mathbf{m}}
\newcommand{\mcA}{\mathcal{A}}
\newcommand{\mcB}{\mathcal{B}}
\newcommand{\mcC}{\mathcal{C}}
\newcommand{\mcD}{\mathcal{D}}
\newcommand{\mcE}{\mathcal{E}}
\newcommand{\mcF}{\mathcal{F}}
\newcommand{\mcG}{\mathcal{G}}
\newcommand{\mcH}{\mathcal{H}}
\newcommand{\mcL}{\mathcal{L}}
\newcommand{\mcI}{\mathcal{I}}
\newcommand{\mcJ}{\mathcal{J}}
\newcommand{\mcW}{\mathcal{W}}

\newcommand{\Div}{\operatorname{Div}}
\newcommand{\Prin}{\operatorname{Prin}}
\newcommand{\Fitt}{\operatorname{Fitt}}
\newcommand{\Char}{\operatorname{Char}}
\newcommand{\ord}{\operatorname{ord}}
\newcommand{\Reg}{\operatorname{Reg}}

\newcommand{\mcM}{\mathcal{M}}
\newcommand{\mcN}{\mathcal{N}}
\newcommand{\mcO}{\mathcal{O}}
\newcommand{\mcP}{\mathcal{P}}
\newcommand{\mcU}{\mathcal{U}}
\newcommand{\mcR}{\mathcal{R}}
\newcommand{\mcS}{\mathcal{S}}
\newcommand{\mcQ}{\mathcal{Q}}
\newcommand{\mcZ}{\mathcal{Z}}
\newcommand{\mfa}{\mathfrak{a}}
\newcommand{\mfA}{\mathfrak{A}}
\newcommand{\mfb}{\mathfrak{b}}
\newcommand{\mfB}{\mathfrak{B}}
\newcommand{\mfC}{\mathfrak{C}}
\newcommand{\mfD}{\mathfrak{D}}
\newcommand{\mfF}{\mathfrak{F}}
\newcommand{\mfL}{\mathfrak{L}}
\newcommand{\mff}{\mathfrak{f}}
\newcommand{\mfj}{\mathfrak{j}}
\newcommand{\mfl}{\mathfrak{l}}
\newcommand{\mfM}{\mathfrak{M}}
\newcommand{\mfm}{\mathfrak{m}}
\newcommand{\mfN}{\mathfrak{N}}
\newcommand{\mfn}{\mathfrak{n}}
\newcommand{\mfo}{\mathfrak{o}}
\newcommand{\mfO}{\mathfrak{O}}
\newcommand{\mfP}{\mathfrak{P}}
\newcommand{\mfQ}{\mathfrak{Q}}
\newcommand{\mfR}{\mathfrak{R}}
\newcommand{\mfS}{\mathfrak{S}}
\newcommand{\mfT}{\mathfrak{T}}
\newcommand{\mfU}{\mathfrak{U}}
\newcommand{\mfV}{\mathfrak{V}}
\newcommand{\mfW}{\mathfrak{W}}
\newcommand{\mfX}{\mathfrak{X}}
\newcommand{\mfY}{\mathfrak{Y}}
\newcommand{\mfZ}{\mathfrak{Z}}
\newcommand{\mfp}{\mathfrak{p}}
\newcommand{\mfq}{\mathfrak{q}}
\newcommand{\mfz}{\mathfrak{z}}
\newcommand{\bfa}{\mathbf{A}}
\newcommand{\bfp}{\mathbf{P}}
\newcommand{\AGL}{\mathbb{A}\GL}
\newcommand{\Qbar}{\overline{\QQ}}
\newcommand{\dmn}{\trianglerighteq}
\renewcommand{\qedsymbol}{$\blacksquare$}
\newcommand{\abs}[1]{\left|#1\right|}
\newcommand{\pphi}{\varphi}
\newcommand{\upto}[1]{\overset{#1}{\to}}

\newcommand{\op}{\textnormal{op}}
\newcommand{\nat}[2]{\textnormal{Nat}(#1,#2)}
\newcommand{\comr}{\textnormal{\textbf{ComRings}}}
\newcommand{\Mod}{\textnormal{\textbf{Mod}}}
\newcommand{\lmod}{\sideset{_R}{}{\mathop{\Mod}}}
\newcommand{\rmod}{\sideset{}{_R}{\mathop{\Mod}}}
\newcommand{\mmod}[1]{\sideset{_{#1}}{}{\mathop{\Mod}}}
\newcommand{\set}{\textnormal{\textbf{Sets}}}
\newcommand{\grp}{\textnormal{\textbf{Grp}}}
\newcommand{\Ab}{\textnormal{{\textbf{Ab}}}}
\newcommand{\Top}{\textnormal{\textbf{Top}}}
\newcommand{\en}[2]{\textnormal{End}_{#1}(#2)}
\newcommand{\hgt}[1]{\textnormal{ht}(#1)}
\newcommand{\gr}[2]{\textnormal{gr}_{#1}(#2)}
\newcommand{\grd}[3]{\textnormal{gr}_{#1}^{#2}(#3)}
\newcommand{\br}{\operatorname{Br}}
\newcommand{\ab}{\operatorname{ab}}
\newcommand{\gal}{\operatorname{Gal}}
\newcommand{\rig}{\textnormal{rig}}

\newcommand{\N}{\operatorname{N}}
\newcommand{\h}{\operatorname{H}}
\newcommand{\disc}[1]{\textnormal{disc}({#1})}
\newcommand{\norm}[1]{|\!|#1|\!|}
\newcommand{\bignorm}[1]{\bigg|\!\bigg|#1\bigg|\!\bigg|}
\renewcommand{\mod}{\operatorname{mod}}
\renewcommand{\Re}{\operatorname{Re}}
\renewcommand{\Im}{\operatorname{Im}}
\newcommand{\Gal}{\operatorname{Gal}}
\newcommand{\cov}{\operatorname{Cov}}
\newcommand{\cat}{\operatorname{Cat}}
\renewcommand{\op}{\operatorname{op}}
\newcommand{\eff}{\textnormal{eff}}
\newcommand\nn             {\nonumber \\}

\newcommand{\tdiv}{\operatorname{\mid\!\mid}}

\newcommand\be            {\begin{equation}}
\newcommand\ee            {\end{equation}}
\newcommand\bea           {\begin{eqnarray}}
\newcommand\eea         {\end{eqnarray}}
\newcommand\bnu          {\begin{enumerate}}
\newcommand\enu          {\end{enumerate}}
\newcommand{\Ad}{\operatorname{Ad}}

\newcommand\id            {\mathrm{id}}
\newcommand\ob          {\operatorname{Ob}}
\renewcommand\hom         {\operatorname{Hom}}
\newcommand\ev          {\mathrm{ev}}
\newcommand\coev      {\mathrm{coev}}
\newcommand\edo    {\mathrm{End}}
\newcommand\funend {\EuScript{E}\mathrm{nd}}
\newcommand\aut      {\mathrm{Aut}}

\newcommand\out      {\mathrm{Out}}
\newcommand\Aut      {\mathcal{A}\mathrm{ut}}
\newcommand\hilb   {\mathrm{Hilb}}
\newcommand\vect    {\mathrm{Vect}}
\newcommand\Mat  {\EuScript{M}\mathrm{at}}
\renewcommand{\div}{\operatorname{div}}
\newcommand\rep     {\mathrm{Rep}}
\newcommand\fun     {\mathrm{Fun}}
\newcommand\Fun    {\EuScript{F}\mathrm{un}}
\newcommand\LMod  {\mathrm{LMod}}
\newcommand\RMod  {\mathrm{RMod}}
\newcommand\BMod {\mathrm{BMod}}
\newcommand\bk       {\mathbb{k}}
\newcommand\forget  {\mathbf{f}}

\newcommand{\sep}{\textnormal{sep}}

\newcommand\alg     {\EuScript{A}\mathrm{lg}}
\newcommand\cTop {\mathrm{Top}}
\newcommand\mfd    {\EuScript{M}\mathrm{fd}}
\newcommand\Set    {\EuScript{S}\mathrm{et}}
\newcommand\sSet    {\mathbf{sSet}}
\newcommand\Topo    {\EuScript{T}\mathrm{op}}
\newcommand\Ring  {\EuScript{R}\mathrm{ing}}
\newcommand\abel {\EuScript{A}\mathrm{bel}}
\newcommand\cring {\EuScript{C}\mathrm{ring}}

\newcommand\one    {\mathbf{1}}
\newcommand\coker  {\mathrm{Coker}}
\newcommand\image     {\mathrm{Im}}

\newcommand{\auto}[1]{\operatorname{Aut}(#1)}
\newcommand{\inv}[1]{\operatorname{Inv}(#1)}
\newcommand{\ch}[1]{\operatorname{char}(#1)}
\newcommand{\colim} {\varinjlim}
\newcommand{\Lim}  {\varprojlim}
\newcommand{\Lan}  {\mathrm{Lan}}
\newcommand{\Ran}  {\mathrm{Ran}}
\newcommand{\Lie}{\operatorname{Lie}}
\newcommand{\BB}{\mathbb{B}}
\newcommand{\Max}{\operatorname{Max}}
\renewcommand{\sp}{\operatorname{Sp}}
\newcommand{\End}{\operatorname{End}}
\renewcommand{\Mat}{\operatorname{Mat}}

\newcommand{\andm}{\textnormal{-}\textsf{AndMod}}
\newcommand{\abm}{\textnormal{-}\textsf{AbMod}}
\newcommand{\rk}{\operatorname{rank}}
\newcommand{\Frac}{\operatorname{Frac}}

\newcommand\CA           {\EuScript{A}}
\newcommand\CB           {\EuScript{B}}
\newcommand\CCC           {\EuScript{C}}
\newcommand\CDD           {\EuScript{D}}
\newcommand\CE          {\EuScript{E}}
\newcommand\CF          {\EuScript{F}}
\newcommand\CG         {\EuScript{G}}
\newcommand\CH         {\EuScript{H}}
\newcommand\CI           {\EuScript{I}}
\newcommand\CJ           {\EuScript{J}}
\newcommand\CK         {\EuScript{K}}
\newcommand\CL          {\EuScript{L}}
\newcommand\CM          {\EuScript{M}}
\newcommand\CN         {\EuScript{N}}
\newcommand\CO         {\EuScript{O}}
\newcommand\CP         {\EuScript{P}}
\newcommand\CQ         {\EuScript{Q}}
\newcommand\CR         {\EuScript{R}}
\newcommand\CS         {\EuScript{S}}
\newcommand\CT         {\EuScript{T}}
\newcommand\CU        {\EuScript{U}}
\newcommand\CV        {\EuScript{V}}
\newcommand\CW        {\EuScript{W}}
\newcommand\CX         {\EuScript{X}}
\newcommand\CY         {\EuScript{Y}}
\newcommand\CZ         {\EuScript{Z}}
\newcommand{\plim}{\varprojlim}

\newcommand{\fA}{\mathfrak{A}}
\newcommand{\fB}{\mathfrak{B}}

\newcommand{\uparr}[1]{\stackrel{\to}{#1}}

\newcommand{\SL}{\operatorname{SL}}
\newcommand{\GL}{\operatorname{GL}}
\newcommand{\sgn}{\operatorname{sgn}}
\newcommand{\ksym}[2]{\left(\frac{#1}{#2}\right)}

\renewcommand{\bfm}{\mathbf{m}}
\newcommand{\bfn}{\mathbf{n}}
\newcommand{\bfi}{\mathbf{i}}
\newcommand{\bfj}{\mathbf{j}}

\newcommand{\bfu}{\mathbf{u}}
\newcommand{\bfv}{\mathbf{v}}

\newcommand{\fitt}{\operatorname{Fitt}}
\newcommand{\Cl}{\operatorname{Cl}}

\newcommand{\bA}{\mathbf{A}}
\newcommand{\bB}{\mathbf{B}}
\newcommand{\bC}{\mathbf{C}}
\newcommand{\bD}{\mathbf{D}}
\newcommand{\bE}{\mathbf{E}}
\newcommand{\bF}{\mathbf{F}}
\newcommand{\bG}{\mathbf{G}}
\newcommand{\bH}{\mathbf{H}}
\newcommand{\bI}{\mathbf{I}}
\newcommand{\bJ}{\mathbf{J}}
\newcommand{\bK}{\mathbf{K}}
\newcommand{\bL}{\mathbf{L}}
\newcommand{\bM}{\mathbf{M}}
\newcommand{\bN}{\mathbf{N}}
\newcommand{\bO}{\mathbf{O}}
\newcommand{\bP}{\mathbf{P}}
\newcommand{\bQ}{\mathbf{Q}}
\newcommand{\bR}{\mathbf{R}}
\newcommand{\bS}{\mathbf{S}}
\newcommand{\bT}{\mathbf{T}}
\newcommand{\bU}{\mathbf{U}}
\newcommand{\bV}{\mathbf{V}}
\newcommand{\bW}{\mathbf{W}}
\newcommand{\bX}{\mathbf{X}}
\newcommand{\bY}{\mathbf{Y}}
\newcommand{\bZ}{\mathbf{Z}}

\newcommand{\sfA}{\mathsf{A}}
\newcommand{\sfB}{\mathsf{B}}
\newcommand{\sfC}{\mathsf{C}}
\newcommand{\sfD}{\mathsf{D}}
\newcommand{\sfE}{\mathsf{E}}
\newcommand{\sfF}{\mathsf{F}}
\newcommand{\sfG}{\mathsf{G}}
\newcommand{\sfH}{\mathsf{H}}
\newcommand{\sfI}{\mathsf{I}}
\newcommand{\sfJ}{\mathsf{J}}
\newcommand{\sfK}{\mathsf{K}}
\newcommand{\sfL}{\mathsf{L}}
\newcommand{\sfM}{\mathsf{M}}
\newcommand{\sfN}{\mathsf{N}}
\newcommand{\sfO}{\mathsf{O}}
\newcommand{\sfP}{\mathsf{P}}
\newcommand{\sfQ}{\mathsf{Q}}
\newcommand{\sfR}{\mathsf{R}}
\newcommand{\sfS}{\mathsf{S}}
\newcommand{\sfT}{\mathsf{T}}
\newcommand{\sfU}{\mathsf{U}}
\newcommand{\sfV}{\mathsf{V}}
\newcommand{\sfW}{\mathsf{W}}
\newcommand{\sfX}{\mathsf{X}}
\newcommand{\sfY}{\mathsf{Y}}
\newcommand{\sfZ}{\mathsf{Z}}

\newcommand{\sA}{\mathscr{A}}
\newcommand{\sB}{\mathscr{B}}
\newcommand{\sC}{\mathscr{C}}
\newcommand{\sD}{\mathscr{D}}
\newcommand{\sE}{\mathscr{E}}
\newcommand{\sF}{\mathscr{F}}
\newcommand{\sG}{\mathscr{G}}
\newcommand{\sH}{\mathscr{H}}
\newcommand{\sI}{\mathscr{I}}
\newcommand{\sJ}{\mathscr{J}}
\newcommand{\sK}{\mathscr{K}}
\newcommand{\sL}{\mathscr{L}}
\newcommand{\sM}{\mathscr{M}}
\newcommand{\sN}{\mathscr{N}}
\newcommand{\sO}{\mathscr{O}}
\newcommand{\sP}{\mathscr{P}}
\newcommand{\sQ}{\mathscr{Q}}
\newcommand{\sR}{\mathscr{R}}
\newcommand{\sS}{\mathscr{S}}
\newcommand{\sT}{\mathscr{T}}
\newcommand{\sU}{\mathscr{U}}
\newcommand{\sV}{\mathscr{V}}
\newcommand{\sW}{\mathscr{W}}
\newcommand{\sX}{\mathscr{X}}
\newcommand{\sY}{\mathscr{Y}}
\newcommand{\sZ}{\mathscr{Z}}

\newcommand{\St}{\textnormal{St}}
\newcommand{\codim}{\operatorname{codim}}
\newcommand{\Frob}{\operatorname{Frob}}
\newcommand{\Pic}{\operatorname{Pic}}
\newcommand{\res}[3]{\left(\frac{#1}{#2}\right)_{#3}}
\newcommand{\qhom}{\operatorname{Hom^\circ}}
\newcommand{\mot}{\textnormal{-}\textsf{Mot}}
\newcommand{\motI}{\textnormal{-}\textsf{MotI}}
\newcommand{\effmot}{\textnormal{-}\textsf{Mot}^{\eff}}
\newcommand{\fgeffmot}{\textnormal{-}\textsf{Mot}^{\eff}_{<\infty}}
\newcommand{\imot}{\textnormal{-}\textsf{Mot}^{\circ}}
\newcommand{\armot}{\textnormal{-}\textsf{Mot}_{\textnormal{Artin}}^{\circ}}
\newcommand{\reps}{\textnormal{-}\textsf{Reps}}

\newcommand{\bx}{\mathbf{x}}
\newcommand{\by}{\mathbf{y}}
\newcommand{\Span}{\operatorname{Span}}
\newcommand{\et}{\textnormal{\'et}}
\newcommand{\Rep}{\mathbf{Rep}}
\newcommand{\Tr}{\operatorname{Tr}}
\newcommand{\Ind}{\operatorname{Ind}}
\newcommand{\resi}[3]{\left(\frac{#1}{#2}\right)_{#3}}

\newcommand{\bsalpha}{\boldsymbol{\alpha}}
\newcommand{\bsbeta}{\boldsymbol{\beta}}
\newcommand{\bsi}{\boldsymbol{i}}
\newcommand{\bslambda}{\boldsymbol{\lambda}}
\newcommand{\bsmu}{\boldsymbol{\mu}}
\newcommand{\bsnu}{\boldsymbol{\nu}}
\newcommand{\bspi}{\boldsymbol{\pi}}

\newcommand{\oD}{\mkern2.5mu\overline{\mkern-2.5mu D}}
\newcommand{\oE}{\mkern2.5mu\overline{\mkern-2.5mu E}}
\newcommand{\oF}{\mkern2.5mu\overline{\mkern-2.5mu F}}
\newcommand{\og}{\overline{g}}
\newcommand{\oh}{\overline{h}}
\newcommand{\ok}{\overline{k}}
\newcommand{\oK}{\mkern2.5mu\overline{\mkern-2.5mu K}}
\newcommand{\oL}{\overline{L}}
\newcommand{\oM}{\mkern2.5mu\overline{\mkern-2.5mu M}}
\newcommand{\on}{\overline{n}}
\newcommand{\oalpha}{\overline{\alpha}}
\newcommand{\obeta}{\overline{\beta}}
\newcommand{\oGamma}{\overline{\Gamma}}
\newcommand{\ochi}{\overline{\chi}}
\newcommand{\oeta}{\overline{\eta}}
\newcommand{\okappa}{\overline{\kappa}}
\newcommand{\ophi}{\mkern2.5mu\overline{\mkern-2.5mu \phi}}
\newcommand{\opsi}{\mkern2.5mu\overline{\mkern-2.5mu \psi}}
\newcommand{\otheta}{\mkern2.5mu\overline{\mkern-2.5mu \theta}}
\newcommand{\oPhi}{\overline{\Phi}}
\newcommand{\oTheta}{\overline{\Theta}}
\newcommand{\oUpsilon}{\overline{\Upsilon}}
\newcommand{\oP}{\mkern2.5mu\overline{\mkern-2.5mu P}}
\newcommand{\oQ}{\overline{Q}}
\newcommand{\oR}{\mkern2.5mu\overline{\mkern-2.5mu R}}
\newcommand{\oS}{\mkern2.5mu\overline{\mkern-2.5mu S}}
\newcommand{\oT}{\overline{T}}
\newcommand{\oU}{\overline{U}}
\newcommand{\oW}{\overline{W}}
\newcommand{\ox}{\overline{x}}
\newcommand{\oX}{\overline{X}}
\newcommand{\oZ}{\mkern2.5mu\overline{\mkern-2.5mu Z}}
\newcommand{\oxi}{\overline{\xi}}
\newcommand{\oXi}{\overline{\Xi}}
\newcommand{\ozeta}{\overline{\zeta}}
\newcommand{\obn}{\overline{\bn}}
\newcommand{\ui}{\underline{i}}
\newcommand{\ufp}{\underline{\fp}}
\newcommand{\ufq}{\underline{\fq}}
\newcommand{\ut}{\underline{t}}
\newcommand{\htheta}{\hat{\theta}}
\newcommand{\hzeta}{\hat{\zeta}}
\newcommand{\ulM}{\underline{M}}
\newcommand{\inn}{\mathrm{in}}
\newcommand{\si}{\mathrm{si}}
\newcommand{\sr}{\mathrm{sr}}
\newcommand{\perf}{\mathrm{perf}}
\newcommand{\Koh}{\operatorname{H}}
\newcommand{\isoto}{\stackrel{\sim}{\to}}
\newcommand{\tors}{\mathrm{tors}}
\newcommand{\nr}{\mathrm{nr}}
\newcommand{\tB}{\widetilde{B}}
\newcommand{\tC}{\widetilde{C}}
\newcommand{\teps}{\widetilde{\varepsilon}}
\newcommand{\tE}{\widetilde{E}}
\newcommand{\tcE}{\widetilde{\cE}}
\newcommand{\tbe}{\widetilde{\be}}
\newcommand{\tg}{\widetilde{g}}
\newcommand{\tilh}{\widetilde{h}}
\newcommand{\tM}{\widetilde{M}}
\newcommand{\tiota}{\widetilde{\iota}}
\newcommand{\tfp}{\widetilde{\fp}}
\newcommand{\tfq}{\widetilde{\fq}}
\newcommand{\tpi}{\widetilde{\pi}}
\newcommand{\tphi}{\widetilde{\phi}}
\newcommand{\tPhi}{\widetilde{\Phi}}
\newcommand{\tPsi}{\widetilde{\Psi}}
\newcommand{\trho}{\widetilde{\rho}}
\newcommand{\ttheta}{\widetilde{\theta}}
\newcommand{\tP}{\widetilde{P}}
\newcommand{\tQ}{\widetilde{Q}}
\newcommand{\tS}{\widetilde{S}}
\newcommand{\tT}{\widetilde{T}}
\newcommand{\tU}{\widetilde{U}}
\newcommand{\tV}{\widetilde{V}}
\newcommand{\tx}{\widetilde{x}}
\newcommand{\ty}{\widetilde{y}}
\newcommand{\tX}{\widetilde{X}}
\newcommand{\btau}{\bar{\tau}}
\newcommand{\Ga}{\GG_{\mathrm{a}}}
\newcommand{\Gm}{\GG_{\mathrm{m}}}
\newcommand{\LLhat}{\widehat{\LL}}
\newcommand{\C}{\CC_{\infty}}

\newcommand{\oFq}{\overline{\FF}_q}
\newcommand{\oFqt}{\overline{\FF_q(t)}}
\newcommand{\oEE}{\overline{\EE}}
\newcommand{\soEE}{\mkern1mu\overline{\mkern-1mu \EE}}
\newcommand{\obE}{\overline{\bE}}
\newcommand{\sobE}{\mkern1mu\overline{\mkern-1mu \bE}}
\newcommand{\oFF}{\overline{\FF}}
\newcommand{\oinfty}{\overline{\infty}}
\newcommand{\Fqts}{\FF_q[\ut_s]}

\newcommand{\tauid}{{\tau=\mathrm{id}}}
\newcommand{\sigmaid}{{\sigma=\mathrm{id}}}
\newcommand{\iso}{\stackrel{\sim}{\longrightarrow}}
\newcommand{\mayeq}{\stackrel{?}{=}}
\newcommand{\power}[2]{{#1 [[ #2 ]]}}
\newcommand{\laurent}[2]{{#1 (( #2 ))}}
\newcommand{\brac}[2]{\genfrac{\{}{\}}{0pt}{}{#1}{#2}}

\newcommand{\dnorm}[1]{\lVert #1 \rVert}
\newcommand{\cdnorm}[1]{\lVert #1 \rVert}
\newcommand{\inorm}[1]{{\lvert #1 \rvert}_{\infty}}
\newcommand{\idnorm}[1]{{\lVert #1 \rVert}_{\infty}}
\newcommand{\diam}[1]{\langle #1 \rangle}
\newcommand{\smod}[1]{{\, (\mathrm{mod}\, #1)}}
\newcommand{\Aord}[2]{{[ #1 ]}_{#2}}
\newcommand{\bigAord}[2]{{\left[ #1 \right]}_{#2}}
\newcommand{\pd}{\partial}
\newcommand{\bdot}{\mathbin{.}}
\newcommand{\tr}{{\mathsf{T}}}
\newcommand{\assign}{\mathrel{\vcenter{\baselineskip0.5ex \lineskiplimit0pt
                     \hbox{\scriptsize.}\hbox{\scriptsize.}}}%
                     =}
\newcommand{\rassign}{=%
                     \mathrel{\vcenter{\baselineskip0.5ex \lineskiplimit0pt
                     \hbox{\scriptsize.}\hbox{\scriptsize.}}}%
                     }
\newcommand{\Res}{\operatorname{Res}}
\newcommand{\ind}{\textnormal{ind}}

\begin{abstract}
Motivated by the classical theory of twisted $L$-functions, this article introduces a motivic framework for studying twisted Goss $L$-series in the function field setting. For a Drinfeld module and an Artin representation, we construct an associated Anderson motive and prove that it arises from a uniformizable abelian Anderson $\bA$-module. Furthermore, we show that the $L$-series of these motives recover the norm of the twisted Goss $L$-values. Ultimately, applying Taelman’s class number formula, our framework provides a structural interpretation of twisted Goss $L$-values in terms of regulators of Anderson modules and gives an application in determining the transcendence of twisted special $L$-values.
\end{abstract}
\maketitle

\section{Introduction}
\subsection{Background and motivation}
    Let $E: y^2 = x^3 + Ax +B$ be an elliptic curve over $\QQ$. Its Hasse-Weil $L$-function admits an infinite series expansion
    $$L(E,s) = \sum_{n\geq 1} \frac{a_n}{n^s}.$$
    Let $D$ be a square-free integer and let $\chi_D = \left(\frac{D}{\cdot}\right)$ be a quadratic character. One may then form the twisted $L$-function
    $$L(E,\chi_D,s) = \sum_{n\geq 1}\frac{a_n\chi_D(n)}{n^s}.$$
    On the other hand, consider the quadratic twist
    $$E^D: Dy^2 = x^3 + Ax + B.$$
    Then
    $$L(E^D,s) = L(E,\chi_D,s).$$
    In particular, this twisted $L$-function arises from another elliptic curve. This leads to the question: given an arbitrary Dirichlet character $\chi$, does $L(E,\chi,s)$ also arise as the Hasse-Weil $L$-function of an elliptic curve?

    The answer is negative, since the twists of an elliptic curve $E$ are controlled by its automorphism group $\Aut(E)$. See, for example, \cite{silverman2009arithmetic}. 

    If a twist by a character $\chi$ does not come from an elliptic curve, one may still ask whether it comes from another geometric object. For instance, can one find an abelian variety $A$ such that $L(A,s)=L(E,\chi,s)$? More generally, what can be said when $\chi$ is replaced by an Artin representation $\rho: \Gal(\QQ^\sep/\QQ)\to \GL_n(\conj{\QQ})$? Following Dokchitser \cite{dokchitser2005l}, twists of elliptic curves by Artin representations should be viewed as motives in their own right. Such twisted $L$-functions also arise naturally in non-commutative Iwasawa theory, providing a family of special values at $s=1$ used to $p$-adically interpolate the (conjectural) analytic $p$-adic $L$-function of the main conjecture; see \cite{dokchitser2007,venjakob2005} for more details.

    Motivated by this perspective, we study Artin twists in the function field setting by investigating twists of Drinfeld modules. Using the theory of Anderson motives, we construct a motivic framework that incorporates higher-dimensional Artin representations, including genuinely non-abelian situations. This provides a function field analogue of Artin twists of elliptic curves, in which abelian varieties are replaced by abelian Anderson modules and Hasse-Weil $L$-functions are replaced by Goss $L$-series.

    This lies within a central theme of modern number theory: relating special values of $L$-functions to arithmetic geometry. In positive characteristic, Drinfeld modules and (abelian) Anderson motives play the role of elliptic curves and motives of abelian type.

    Our primary object is an Anderson motive $M = \MM(\varphi,\rho)$, constructed from a Drinfeld module $\varphi$ and an Artin representation $\rho$ of the absolute Galois group. A key outcome of the construction is the identity:
\begin{equation}
    L(M,s) = N_{K_\infty(\rho)/K_\infty}\big(L(\varphi^\vee,\rho,s)\big),
\end{equation}
which provides a precise motivic interpretation of twisted $L$-values, see Proposition \ref{motivic_L_series} and Theorem \ref{thmC}. To illustrate the consistency of this framework, we show that it recovers well-known objects in special cases; for instance, when $\varphi$ is the \emph{Carlitz module} $C$ and $\chi$ is a quadratic character, the identity yields $L(M,s-1) = L(\chi,s)$. Ultimately, this approach not only elucidates the structure of twisted $L$-functions in positive characteristic but also opens new avenues for exploring the arithmetic of special $L$-values over function fields.

\subsection{Main results.} 	Let $X/\FF_q$ be a smooth, geometrically connected projective curve over $\FF_q$ and $\infty\in X$ a fixed closed point of degree $d_\infty$, i.e. $d_\infty = [\FF_\infty: \FF_q]$, where $\FF_\infty:=\mcO_{X,\infty}/\mfm_\infty$ is the residue field at $\infty$. Let $\bA = \mcO_X(X-\infty,\mcO_X)$ be the ring of rational functions regular outside $\infty$. Let $\bK = \Frac(\bA)$ and $\bK_\infty$ the completion of $\bK$ with respect to the normalized valuation $v_\infty$ at $\infty$.

  Let $A, K, \text{and } K_\infty$ denote fixed isomorphic copies of $\bA, \bK, \text{and } \bK_\infty$, respectively. The natural inclusion $\iota:\bA\stackrel{\sim}{\to}A\hookrightarrow K$ makes $K$ into an $\bA$-field. Throughout this paper, symbols in boldface will represent operators, while elements of $A, K, \text{and } K_\infty$ will serve as constants. For example, if $X = \PP^1_{\FF_q}$, $\infty = [0:1]\in \PP^1(\FF_q)$. Then we usually have $\iota:\bA = \FF_q[t]\stackrel{\sim}{\longrightarrow}\FF_q[\theta] = A$, $t\mapsto \theta$. Let $E/K$ be a finite extension.
  
  Our main results focus on Artin representations $\rho: G_E\to \GL_n(\conj{\FF}_{q})$, where $G_E = \Gal(E^\sep/E)$. Let $\FF_q(\rho) = \FF_q(\rho(g)_{ij}: g\in G)$ and put $d=d_\rho = [\FF_q(\rho): \FF_q]$. Denote by $\rho^{(i)}$ the $i$-th Frobenius twist of $\rho$, i.e. $\rho^{(i)}(g) = (a_{ij}^{q^i})$ if $\rho(g)= (a_{ij})$.
\begin{thmx}[Theorem \ref{existsAnderson}]
	Let $\varphi/E$ be a Drinfeld $\bA$-module of rank $r$ and $\rho: G_E\to \GL_n(\conj{\FF}_{q})$ an $\conj{\FF}_{q}$-representation of dimension $n$. Then, there exists an abelian $\bA$-module $\mcE$ of dimension $N = nd$ and rank $rN$ over $E$ such that $$\MM(\varphi,\rho)\cong_E \MM(\mcE).$$
\end{thmx}

In the case that $\bA = \FF_q[t]$, $A = \FF_q[\theta]$, we investigate further and construct an explicit model of the Artin twist $\mcE$ over $E$. This makes it possible to calculate its associated Goss $L$-series and its special $L$-values, especially using Taelman's class number theorem and extensions due to Angl\'es, Fang, Ngo Dac and Ribeiro. See \cite{angles2022class,fang2015special,taelman2012special} for details.

Let $\alpha_1,\cdots,\alpha_d$ be a fixed $\FF_{q}$-basis for $\FF_q(\rho)$ and $\boldsymbol{\vec{\alpha}} = (\alpha_1,\cdots,\alpha_d)$. By a fundamental solution, we mean a solution $\vec{\mathbf{u}}\in \Mat_{n\times d}(E^\sep)$  of the system of equations $\rho^{(\ell)}(g)g(\vec{\mathbf{u}})\boldsymbol{\vec{\alpha}}^{(\ell)} = \vec{\mathbf{u}}\boldsymbol{\vec{\alpha}}^{(\ell)}$, whose entries are linearly independent over $\FF_q$, for all $g\in G_E$ and $0\leq \ell\leq d-1$. See \S \ref{construction} for details.
\begin{thmx}[Theorem \ref{thmB}]
	Let $\varphi: \bA\to E[\tau]$ be a Drinfeld module over $E$ with $$\varphi_t = \theta + a_1\tau +\cdots + a_r\tau^r$$ and $\rho:G_E\to \GL_n(\conj{\FF}_q)$ be an Artin representation. Let $N = nd$ with $d = d_\rho = [\FF_q(\rho): \FF_q]$. Then, $\mcE$ has a model over $E$ given by
    $$\mcE_t = \theta I_N + a_1\Psi \tau +a_2 \Psi\Psi^{(1)}\tau^2+\cdots+ a_r\Psi\Psi^{(1)}\cdots \Psi^{(r-1)}\tau^r,$$ where $\Psi = (\Phi(\vec{\mathbf{u}})^T)^{-1}$ and $\vec{\mathbf{u}}\in \operatorname{Sol}_E(\rho)$ is a fundamental solution. See $(\ref{sol_space})$ and $(\ref{Phimatrix})$ for the definition of $\operatorname{Sol}_E(\rho)$ and $\Phi(\bm{\vec{u}})$ respectively.
\end{thmx}
As a result, we have the following corollary.
\begin{corx}[Corollary \ref{twsit_is_uniformizable}]
    The Anderson $\bA$-module $\mcE$ over $E$ is uniformizable.
\end{corx}
Furthermore, the following result establishes a precise connection between the $L$-functions of these objects:
\begin{thmx}[Theorem \ref{thmC}]
Let $\varphi:\bA\to E[\tau]$ be a Drinfeld $\bA$-module over $E$ and $\rho: G_E\to \GL_n(\conj{\FF}_q)$ an Artin representation. Let $\mcE = \mcE(\varphi,\rho)$ be an Artin twist of $\varphi$ by $\rho$. Then, there exists a finite set $S$ of places such that $$L_S(\mcE^\vee,s) = \prod_{i=0}^{d-1}L_S(\varphi^\vee,\rho^{(i)},s).$$ 
In the case $\bA = \FF_q[t]$, we have $$L_S(\mcE^\vee,s) = N_{K_\infty(\rho)/K_\infty}(L_S(\varphi^\vee,\rho,s))$$
\end{thmx}

By Taelman's class number formula, we obtain the following two corollaries. Note that although $\mcE$ in the theorem is defined over the field $E$, we may replace it by an integral model over $\mcO_E$ via clearing denominators. This modification only affects finitely many primes and introduces a rational factor in $K^*$.

\begin{corx}[Corollary \ref{corD}]
	Let $\varphi$ be a Drinfeld module over $E$ of rank $r$ and $\rho: G\to \GL_n(\conj{\FF}_q)$ be an Artin representation with $d = d_\rho\geq 1$. Let $\mcE$ be a model over $\mcO_E$ of the Artin twist $\mcE(\varphi,\rho)$. Then, there exists a constant $C = C(\varphi,\rho)\in K^*$ such that $$N_{K_\infty(\rho)/K_\infty}(L_S(\varphi^\vee, \rho, 0)) = C\cdot \Reg(\mcE/\mcO_E)\cdot h(\mcE/\mcO_E).$$
\end{corx}
Here, $\Reg(\mcE/\mcO_E)$ is Taelman's regulator and $h(\mcE/\mcO_E)\in A_+$ is a generator of the Fitting ideal of the class module. See \S \ref{Taelmans_class_number} for details.

We also have the following corollary.
\begin{corx}[Corollary \ref{corE}]
Let $\varphi: \bA\to \mcO_E[\tau]$ be a Drinfeld module over $\mcO_E$. Let $L/E$ be a Galois extension of degree $m$ with $p\nmid m$. Then there exists an element $c\in K^*$ such that
    $$\Reg(\varphi/\mcO_L) = c\cdot \prod_{[\rho]}\Reg(\mcE(\varphi,\rho)/\mcO_E)^{\dim\rho},$$ where $\mcE(\varphi,\rho)$ denotes an integral model over $\mcO_E$ and the product runs through all Frobenius orbits $[\rho]$ of irreducible representations of $\Gal(L/E)$.
\end{corx}

As an application, we have the following result.
\begin{thmx}[Theorem \ref{thm:transcendence_criterion}]
	Let $\varphi$ be a Drinfeld $\bA$-module over $K$ and let $\rho: G_K\to \GL_n(\conj \FF_q)$ be an Artin representation.  Then the special $L$-value $L(\varphi^\vee, \rho, 0)$ is transcendental over $K$.
\end{thmx}
The paper is organized as follows. In section~\ref{sec:prelim}, we review background material on Anderson modules, Anderson motives, and their Tate modules. In section~\ref{sec:GossLseries}, we define Goss $L$-series attached to several arithmetic objects. In section~\ref{sec:twist}, we give an explicit construction of Artin twists of Drinfeld modules and prove the main theorems. Then, in section~\ref{sec:application}, we apply our main results to show that the special value $L(\varphi^\vee,\rho,0)$ is transcendental over $K$. Finally, in section \ref{sec:example} we discuss some examples of Artin twists of Drinfeld modules.
\subsection*{Acknowledgements}
The author would like to thank Matthew Papanikolas for many fruitful discussions and for his helpful comments on an earlier draft.
\section{Preliminaries}\label{sec:prelim}
\subsection{Notation} \label{SS:Notation}
Throughout the paper, we will use the following notation.
\begin{longtable}{l @{\hspace{1em}=\hspace{1em}} p{0.85\textwidth}}
$p$ & a prime number in $\ZZ$.\\
$\FF_q$ & a finite field with $q$ elements, where $q$ is a power of $p$.\\
$X$ & a smooth, geometrically irreducible projective curve over $\FF_q$, e.g. $X = \PP^1_{\FF_q}$.\\
$\infty$ & a fixed closed point of degree $d_\infty$ on $X$, e.g. $\infty = [1:0]\in \PP^1_{\FF_q}$.\\
$\bA$ & $\Gamma(X-\infty,\mcO_X)$, e.g. $\bA = \FF_q[t]$.\\
$\bK$ & $\Frac(\bA)$, e.g. $\bK = \FF_q(t)$.\\
$A$ & a fixed isomorphic copy of $\bA$, e.g. $A = \FF_q[\theta]$.\\
$K$ & a fixed isomorphic copy of $\bK$, e.g. $K = \FF_q(\theta)$.\\
$v_\infty$ & the normalized valuation at $\infty$.\\
$\inorm{\,\cdot\,}$ & the normalized $\infty$-adic norm on $\CC_\infty$ given by $\inorm{a} = q^{-d_\infty v_\infty(a)}$. \\
$K_{\infty}$ & the completion of $K$ with respect to $\inorm{\,\cdot\,}$, e.g. $\laurent{\FF_q}{1/\theta}$. \\
$\conj{K}_\infty$ & a fixed algebraic closure of $K_\infty$.\\
$\CC_\infty$ & the completion of $\conj{K}_{\infty}$ with respect to $\inorm{\cdot}$, which is also algebraically closed. \\
$\deg$ & $-d_\infty v_{\infty}$. \\
$\conj{K}$ & the algebraic closure of $K$ in $\CC_\infty$. \\
\end{longtable}

\subsubsection{Frobenius twist and twisted polynomial ring}
Let $R$ be an $\FF_q$-algebra. We denote by $\tau: R\to R$, $r\mapsto r^q$ the Frobenius map.
Denote by $\bA_R:= \bA\otimes_{\FF_q} R$. Then the Frobenius map $\tau$ extends $\bA$-linearly to $\bA_R$ in a unique way, i.e. $\tau_{\bA_R}: \bA_R\to \bA_R$, $$\sum_{i=0}^n a_i\otimes r_i\mapsto \sum_{i=0}^n a_i\otimes r_i^q.$$

For any positive integer $i$ and any element $x$ in $R$ or $A_R$, we will use $x^{(i)}$ to denote its image under $\tau^i$. We call it the $i$-th twist of $x$.

Let $R[\tau] = \left\{\dis \sum_{i=0}^n a_i\tau^i: a_i\in R\right\}$, which is closed under addition and scalar multiplication. We define multiplication by $\bigg(\sum\limits_{i = 0}^na_i\tau^i\bigg)\bigg(\sum\limits_{j = 0}^mb_j\tau^j\bigg) = \sum\limits_{k = 0}^{m+n}\sum\limits_{i+j = k} a_ib_j^{(i)}\tau^{k}$. We call $R[\tau]$ the twisted polynomial ring in $\tau$ with coefficients in $R$. For an element $f\in R[\tau]$, denote by $\partial f$ the constant term of $f$, i.e. $\partial f = a_0$ if $\dis f = \sum_i a_i\tau^i$.

Let $\GG_{a,R}$ be the additive group scheme over $R$. It is well known that $R[\tau] \cong \End_{\FF_q}(\GG_{a,R})$, where the latter is the ring of all $\FF_q$-linear endomorphisms of $\GG_{a,R}$.

\subsubsection{Twists of matrices}
Let $R$ be an $\FF_q$-algebra as above. Denote by $\Mat_{m\times n}(R)$ the set of $m$ by $n$ matrices with entries in $R$. Let $M = (m_{ij})\in \Mat_{m\times n}(R)$, then denote by $M^{(k)} = (m_{ij}^{(k)})$ the $k$-th twist of $M$. 

We will identify $\Mat_{m\times n}(R[\tau])$ with the twisted polynomial ring $\Mat_{m\times n}(R)[\tau]$ with matrix coefficients.
The multiplication is defined in a natural way, i.e. $$\left(\sum_{i=0}^mM_i\tau^i\right)\left(\sum_{j=0}^nN_j\tau^j\right) = \sum_{k=0}^{m+n}\sum_{i+j=k} M_iN_j^{(i)}\tau^k.$$

With some work, one can show that $$\hom_{\FF_q}(\GG_{a,R}^{\oplus n}, \GG_{a,R}^{\oplus m})\cong \Mat_{m\times n}(R)[\tau],$$ where the left-hand side denotes the set of all $\FF_q$-linear morphisms $\GG_{a,R}^{\oplus n}\to \GG_{a,R}^{\oplus m}$ of group schemes. See \cite[Lemma 5.1]{hartl2020pinks} for a proof. Let $M =\dis  \sum_{i=0}^n M_i\tau^i\in \Mat_{m\times n}(R)[\tau]$, we denote by $\partial M := M_0$.
\subsection{$\bA$-motives}
In this section, we briefly introduce the theory of $\bA$-motives, originally proposed by G.~Anderson in \cite{anderson1986t} for $\bA = \FF_q[t]$. Our exposition follows \cite[\S 2]{Gazda21motcoh} and \cite[\S 3]{hartl2020pinks}. No originality in this section is claimed.
\subsubsection{Definition of $\bA$-motives}
Let $R$ be an $\bA$-algebra containing $\FF_q$ with the structure map $\iota:\bA\to R$. Let $\mfj :=\mfj_\iota= (a\otimes1-1\otimes\iota(a): a\in \bA)$ be the ideal of $\bA_R$ generated by all elements of the form $a\otimes 1-1\otimes\iota(a)$. Then, $\mfj$ is the kernel of the surjective map $\bA_R\to R,$ $a\otimes r\mapsto \iota(a)r.$ 
It is easy to see that the above map induces an isomorphism of $R$-algebras $\bA_R/\mfj\cong R.$ Moreover, $\mfj$ is a maximal ideal of $\bA_R$ if and only if $R$ is a field; $\mfj$ is a prime ideal if and only if $R$ is an integral domain.

By \cite[Lemma 2.1]{Gazda21motcoh}, the ideal $\mfj$, as an $\bA_R$-module, is invertible. In particular, the ideal $\mfj^n$ is well-defined for any integer $n$ and we have inclusions $\mfj^{-n}\hookrightarrow \mfj^{-(n+1)}$ for each integer $n$.

For an $\bA_R$-module $M$, we define $$M[\mfj^{-1}] = \dlim M\otimes_{\bA_R}\mfj^{-n}.$$

We will also use the following notation throughout this article.
\begin{enumerate}
    \item Let $\tau^*M$ denote the pull-back of $M$ along $\tau$, i.e. the $\bA_R$-module $\bA_R\otimes_{\tau,\bA_R}M,$ where by $\otimes_\tau$ we mean the relation $a\otimes bm = a\tau(b)\otimes m = ab^{(1)}\otimes m$ holds for all $a,b\in \bA_R$ and the $\bA_R$-module structure on $\tau^*M$ is given by $b\cdot (a\otimes m) := ba\otimes m$.
	
	\item For $m\in M$, we denote by $\tau^*m = 1\otimes m$.
	
	\item Let $$\tau^*M[\mfj^{-1}]: = \tau^*M\otimes_{\bA_R}\bA_R[\mfj^{-1}]$$ and $$M[\mfj^{-1}] = M\otimes_{\bA_R}\bA_R[\mfj^{-1}].$$
\end{enumerate}

\begin{Def}
    $(1)$ An \textit{Anderson $\bA$-motive} or simply \textit{$\bA$-motive} over $R$ is a pair $(M,\tau_M)$, where $M$ is a projective $\bA_R$-module and $$\tau_M:\tau^*M[\mfj^{-1}]\stackrel{\sim}{\to}M[\mfj^{-1}]$$ is an isomorphism of $\bA_R[\mfj^{-1}]$-modules. We will simply write $M$ if there is no confusion from the context.
    
    $(2)$ We say an $\bA$-motive $(M,\tau_M)$ is \textit{abelian} if $M$ is finitely generated of constant rank over $\bA_R$.
    
    $(3)$ The \textit{rank} of an $\bA$-motive $(M,\tau_M)$ is defined to be the rank of the underlying $\bA_R$-module $M$.

    $(4)$ A \textit{morphism} $(M, \tau_M) \to (N, \tau_N)$ of $\bA$-motives over $R$ is a homomorphism $f: M \to N$ of $\bA_R$-modules such that the following diagram commutes $$\xymatrix{
	\tau^*M[\mfj^{-1}] \ar[r]^{(\tau^*f)_\mfj}\ar[d]_{\tau_M}& \tau^*N[\mfj^{-1}]\ar[d]^{\tau_N}\\
	M[\mfj^{-1}]\ar[r]_{f_\mfj} & N[\mfj^{-1}]
}$$ 

$(5)$ We say an $\bA$-motive $(M,\tau_M)$ is \textit{effective} if $\tau_M$ restricts to $\tau_M: \tau^*M\to M$.

\end{Def}
For an effective abelian $\bA$-motive over $R$, $\tau_M$ is an isomorphism if and only if $\operatorname{coker}(\tau_M)$ is a locally free $R$-module and annihilated by $\mfj^e$ for some integer $e$. See \cite[\S 2]{hartl2017isogenies} for details.

\begin{rmk}
    If $\bA = \FF_q[t]$, we shall use the term Anderson $t$-motive or simply $t$-motive instead of $\bA$-motive.
\end{rmk}
\begin{rmk}\label{rmk_motives}
    Note that we can define a $\tau$-action on $M$ via $m\mapsto \tau_M(\tau^\ast m)$. This is a semi-linear action in the sense that $\tau(\lambda m) = \lambda^{(1)}\tau(m)$ for all $\lambda\in \bA_R$ and $m\in M$. Let $\bA\mot(R)$ denote the category of $\bA$-motives over $R$ and $\bA\fgeffmot(R)$ denote the full subcategory of effective $\bA$-motives $(M,\tau_M)$ over $R$ such that there is a faithfully flat ring homomorphism $R\to R'$ and $M\otimes_RR'$ is a finite free left $R'[\tau]$-module.
\end{rmk}
If $M$ is free over $\bA_R$ and $\{m_1,\cdots,m_r\}$ is a basis over $\bA_R$. Let $\Phi = (a_{ij})\in \Mat_{r\times r}(\bA_R)$ be the matrix such that $$\tau (m_j) = \sum_{i=1}^r a_{ij}m_i,$$ or equivalently, $$\tau\mathbf{m} = \Phi^\top\mathbf{m},$$ where $\mathbf{m} = (m_1,\cdots,m_r)^\top\in \Mat_{r\times 1}(M)$. In this case, we say $\Phi$ is the matrix of the $\tau$-action with respect to $m_1,\cdots,m_r$.
\subsubsection{$\mfl$-adic cohomology of $\bA$-motives}
In this subsection, we let $L$ be a field containing $K$. Let $L^\sep$ be a separable closure of $L$ fixed once for all and denote by $G_L=\Gal(L^\sep/L)$ the absolute Galois group of $L$ equipped with the profinite topology. Note that we have an inclusion $\iota:\bA\to L$ as $L$ contains $K$ as a subfield.

Let $\mfl$ be a \emph{finite place} of $X$, i.e. $\mfl\in X-\infty$ is a closed point. Let $\bA_\mfl$ be the $\mfl$-adic completion of $\bA$, and let $\bK_\mfl$ be the fraction field of $\bA_\mfl$. Now, set $\bA_{L^\sep} = \bA\otimes_{\FF_q} L^\sep$. Then we can define an action of $G_L$ on $\bA_{L^\sep}$ by $$g\cdot (a\otimes x) = a\otimes g(x).$$

Then this action is $\bA$-linear. Furthermore, we can extend this action to an $\bA_\mfl$-linear action of $G_L$ on
$$\bA_{\mfl}(L^\sep):=\varprojlim_n~\bA_{L^\sep}/\mfl^n\bA_{L^\sep}$$ continuously. Also set $\bA_{\mfl}(L):=\varprojlim_n~\bA_{L}/\mfl^n\bA_{L}$.

Let $M$ be an $\bA$-motive over $L$ of rank $r$. Let $M_{L^\sep} : = M\otimes_L L^\sep$ be the $\bA$-motive over $L^\sep$ obtained from $M$ by base-change to $L^\sep$. For $R = L$ or $L^\sep$, we can further define the $R$-value points of $M$ at $\mfl$ by $$M_{\mfl}(R):=\varprojlim_{n}M_{R}/\mfl^n M_{R}.$$

Note that we have $M_\mfl(R) \cong M\otimes_{\bA_R}\bA_{\mfl}(R)$ canonically. Then, one immediately sees that $G_L$ acts $\bA_\mfl$-linearly on $M_\mfl(L^\sep)$ and the $G_L$-action on the submodule $M_{\mfl}(L)=M\otimes_{\bA_L}\bA_{\mfl}(L)$ is invariant. 

Now, define a $\tau$-action on $M_{L^\sep}$ by $$\tau\colon m\mapsto\tau_M(\tau^\ast m).$$ We then have an $L^\sep$-semilinear map $\tau: M_\mfl(L^\sep)\to M_\mfl(L^\sep)$ that is also $\bA_\mfl$-linear. Here, by $L^\sep$-linear, we mean $\tau(\lambda m) = \lambda^q\tau(m)$ for all $\lambda\in L^\sep$ and $m\in M_\mfl(L^\sep)$. We call the $\bA_\mfl$-module 
$$M_\mfl^\tau :=\{m\in M_\mfl(L^\sep)\;|\;\tau(m)=m\}$$ \emph{the submodule of $\tau$-invariants} of $M_\mfl(L^\sep)$.

Following \cite{anderson1986t}, \cite[\S 2.3]{Gazda21motcoh} and \cite[\S 3.5]{hartl2020pinks}, we define:
\begin{Def}\label{def:m-adic realization functor}
Let $M$ be an $\bA$-motive over $L$. The \emph{$\mfl$-adic cohomology realization} of $M$, denoted by $H^1_\mfl(M)$, is the $\bA_\mfl$-submodule of $\tau$-invariants of $M_\mfl(L^\sep)$. In other words, $$\operatorname{H}^1_{\mfl}(M): = H^1_\mfl(M,\bA_\mfl): = M_\mfl^\tau.$$ 

We also define the \emph{$\mfl$-adic cohomology realization} of $M$ with coefficients in $\bK_\mfl$ to be the $\bK_\mfl$-vector space $$H_\mfl^1(M,\bK_\mfl):= M_\mfl^\tau\otimes_{\bA_\mfl}\bK_\mfl.$$ 
\end{Def}

Note that $H^1_\mfl(M)$ and $H^1_\mfl(M,\bK_\mfl)$ admits a natural $G_L$-action in the following way. First, every element in $\bA_\mfl(L^\sep)\cong \bA_\mfl \otimes_{\FF_q} L^\sep$ has the form $\dis x = \sum_i a_i\otimes\lambda_i$ with $a_{i}\in \bA_\mfl$ and $\lambda_{i}\in L^\sep$. Thus, the $G_L$-action on $L^\sep$ extends to a $G_L$-action on $\bA_\mfl(L^\sep)$ via $$\left(g, \sum_i a_i\otimes \lambda_i\right) = \sum_i a_{i}\otimes g(\lambda_{i}).$$ Note that the $\tau$-action commutes with the $G_L$-action. This $G_L$-action can be further extended to $M_\mfl(L^\sep) = M\otimes_{\bA_L} \bA_\mfl(L^\sep)$ via $g\cdot (m\otimes \lambda) = m\otimes g(\lambda)$ for all $m\in M$ and $\lambda\in \bA_\mfl(L^\sep)$.

Since the $\tau$-action on $M_\mfl(L^\sep)$ commutes with the $G_L$-action, i.e. $g(\tau(m)) = \tau(g(m))$, if $m\in M_\mfl^\tau$, we then have $g(m) = g(\tau(m)) = \tau(g(m))$. In other words, $g(m)\in M_\mfl^\tau$. Thus, the $G_L$-action restricts to the $\bA_\mfl$-module $M_\mfl^\tau$.

Hence, we obtain an $\bA_\mfl$-representation of $G_L$: $$\rho_{M,\mfl}: \Gal(L^\sep/L)\to \aut_{\bA_\mfl}(H^1_\mfl(M))$$ and similarly a $\bK_\mfl$-representation of $G_L$: $$\rho_{M,\mfl}: \Gal(L^\sep/L)\to \GL(H^1_\mfl(M,\bK_\mfl)).$$
\begin{Def}
We define the \emph{$\mfl$-adic homology realizations} of $M$ with coefficients in $\bA_\mfl$ (resp. $\bK_\mfl$) by 
$$H_{1,\mfl}(M): = H_{1,\mfl}(M,\bA_\mfl) := \hom_{\bA_\mfl}(H^1_\mfl(M),\bA_\mfl) $$ and resp. $$H_{1,\mfl}(M,\bK_\mfl)  :=  \hom_{\bK_\mfl}(H^1_\mfl(M,\bK_\mfl),\bK_\mfl).$$
\end{Def}

We hence have dual Galois representations of $G_L$: $$\rho^\vee_{M,\mfl}: \Gal(L^\sep/L)\to \aut_{\bA_\mfl}(H_{1,\mfl}(M))$$ and $$\rho^\vee_{M,\mfl}: \Gal(L^\sep/L)\to \GL(H_{1,\mfl}(M,\bK_\mfl)).$$

We record the following result for future use. See \cite[Lemma 2.26]{Gazda21motcoh} and \cite[\S 3.5]{hartl2020pinks} for details.
\begin{prop} Let $M$ be an $\bA$-motive of rank $r$. Denote by $\Rep_{\bA_\mfl}(G_L)$ the category of continuous $\bA_{\mfl}$-linear $G_L$-representations. Then, 
\begin{enumerate}
    \item We have an isomorphism $$H^1_\mfl(M)\otimes_{\bA_\mfl}\bA_\mfl(L^\sep)\to M_\mfl(L^\sep),$$ $$v\otimes a\mapsto a\cdot v,$$ of $\bA_\mfl(L^\sep)$-modules.
    \item The $\bA_\mfl$-module $H^1_\mfl(M)$ is free of rank $r$.
    \item Moreover, the $G_L$-action on $H^1_\mfl(M,\bA_\mfl)$ is continuous.
    \item We have an exact sequence 
\begin{equation}
0\to H^1_\mfl(M)\longrightarrow M_{\mfl}(L^\sep)\xrightarrow{\id-\tau_M} M_{\mfl}(L^\sep) \longrightarrow 0. \nonumber
\end{equation}
of $\bA_{\mfl}[G_L]$-modules.
\item The $\mfl$-adic cohomology functor $$H^1_\mfl: \bA\mot(L)\to \Rep_{\bA_\mfl}(G_L),$$ given by $$M\mapsto \Koh^1_\mfl(M),$$is an exact faithful functor.
    \end{enumerate}
\end{prop}

\subsubsection{Base change and Weil restriction of $\bA$-motives}
Let $S$ be an $R$-algebra. Then we have a canonical isomorphism $$\begin{aligned}
    \operatorname{can}: \tau^*(\bA_S[\mfj^{-1}]\otimes_{\bA_R[\mfj^{-1}]}M[\mfj^{-1}])&\stackrel{\sim}{\to}\bA_S[\mfj^{-1}]\otimes_{\bA_R[\mfj^{-1}]}\tau^*M[\mfj^{-1}],\\
    1\otimes_\tau (s\otimes m)&\mapsto \tau(s)\otimes (1\otimes_\tau m)
\end{aligned}$$

We now consider the following two functors:

$(1)$ The base-change: $$S\otimes_R-: \bA\mot(R)\to \bA\mot(S),$$ $$M\mapsto M_S: = (\bA_S\otimes_{\bA_R}M,\tau_{M_S}:=(\id_{\bA_S}\otimes \tau_M)\circ \operatorname{can})$$ and 

$(2)$ The Weil restriction of scalars: $$\Res_{S/R}: \bA\mot(S)\to \bA\mot(R),$$ $$M\mapsto (\Res_{S/R}(M),\tau_M)$$ where $\Res_{S/R}(M)$ is considered as an $\bA_R$-module.

We have $$\hom_{\bA\mot(R)}(M,\Res_{S/R}(N)) = \hom_{\bA\mot(S)}(M_S,N).$$

Note that the $\tau$-action on $M_S$ is diagonal.

\subsection{Anderson $\bA$-modules} \label{SectDefAModules}
In this section, we review the concepts of abelian Anderson $\bA$-modules and their associated $\bA$-motives. These objects were originally introduced by Anderson in  \cite{anderson1986t} under the assumption that the base field is perfect and $\bA = \FF_q[t]$. Subsequently, Hartl \cite{hartl2017isogenies} generalized these notions by replacing the base field with a general commutative ring $R$ and considering a general $\bA$. We will mainly follow \cite{Gazda21motcoh}, \cite{hartl2017isogenies} and \cite{hartl2020pinks}. An interested reader can also check \cite{anderson1986t} and \cite{Goss} for more details.

Let $R$ be a commutative $\bA$-algebra with a structure map $\iota: \bA\to R$. Let $\mcE$ be a smooth commutative group scheme over $R$ and $\Lie_\mcE$ be its Lie algebra, i.e. the tangent space at the zero element $e:\Spec R\to \mcE$. See \cite[Chapter 6]{GortzWedhorn20} or \cite[\S 33.16]{stacks-project} for the definition of tangent space of a general scheme.

For each group $R$-scheme homomorphism $f: \mcE_1\to \mcE_2$, we have an induced homomorphism $\partial f:\Lie_{\mcE_1}\to\Lie_{\mcE_2}$. See \cite[Remark 6.3(3)]{GortzWedhorn20} or \cite[Lemma 33.16.6]{stacks-project} for details.

Denote by $\GG_{a,R}=\Spec R[X]$ the additive group scheme over $R$, which is an $R$-module scheme. By \cite[Proposition 2.3.]{Har}, for any element $r\in R$, the map $R[X]\to R[X], X\mapsto rX$ on global sections corresponds to the left multiplication on $\GG_{a,R}$, denoted by $r:\GG_{a,R}\to \GG_{a,R}$. Similarly, the arithmetic Frobenius map $R[X]\to R[X], X\mapsto X^q$ on global sections defines a map $\Frob_q: \GG_{a,R}\to \GG_{a,R}$, call the geometric Frobenius of $\GG_{a,R}$. See \cite[Remark/Definition 4.24.]{GortzWedhorn20} for details.

In order to introduce the definition of Anderson $\bA$-modules, we need the following result, whose proof can be found in \cite[Lemma 5.1]{hartl2020pinks} or \cite{vanderHeiden2003}.
\begin{lem}\label{LemmaLucas}
Define a map $$F:\Mat_{m\times n}(R[\tau])\to \hom_{\FF_q,R}(\GG_{a,R}^{\oplus n},\GG_{a,R}^{\oplus m})$$ of $R$-modules, given by $$(f_{ij})_{ij}\mapsto (f:\GG_{a,R}^{\oplus n}\to \GG_{a,R}^{\oplus m}),$$
where $f$ is defined by $\dis f^{\#}: R[Y_1,\cdots,Y_m]\to R[X_1,\cdots,X_n]$, $\dis Y_i\mapsto \sum_{j=0}^n f_{ij}(X_j)$ via the correspondence in \cite[Proposition 3.4]{GortzWedhorn20}. 

Then, $F$ is an isomorphism of $R$-modules.
\end{lem}

For each $f\in \hom_{\FF_q,R}(\GG_{a,R}^{\oplus n},\GG_{a,R}^{\oplus m})$, it corresponds to an element of the form $\dis \sum_{i = 0}^N A_i\tau^i\in \Mat_{m\times n}(R[\tau])$. We define $\partial f\in \hom_{\FF_q,R}(\GG_{a,R}^{\oplus n},\GG_{a,R}^{\oplus m})$ to be the map corresponding to $A_0$.

\begin{Def}
Let $d$ be a positive integer.
\begin{enumerate}
\item 
An \emph{Anderson $\bA$-module of dimension $d$} over $R$ is a pair  $(\mcE,\phi)$ consisting of a smooth affine commutative group $R$-scheme $\mcE$ of relative dimension $d$ together with an $\FF_q$-algebra homomorphism $$\begin{aligned}
    \phi:\bA&\to \End_{\FF_q,R}(\mcE), \\a&\mapsto\phi_a
\end{aligned}$$ such that 
\begin{enumerate}
    \item there exists a faithfully flat ring homomorphism $R\to R'$ with $\mcE\times_{\Spec R}\Spec R'\cong \GG_{a,R'}^{\oplus d}$ as $\FF_q$-module schemes,
    \item on the Lie algebra $\Lie_\mcE$, we have 
$$
\left(\partial\phi_a-\iota(a)\cdot \operatorname{id}_{\Lie_\mcE}\right)^d=0.
$$
\end{enumerate}
For our own purpose, we will mainly restrict to the special case that $R' = R$ in this article.
\item 
A \emph{morphism} of Anderson $\bA$-modules $f\colon(\mcE,\phi)\to(\mcF,\psi)$ is a homomorphism of group $R$-schemes $f: \mcE\to \mcF$ such that $\psi_a\circ f = f\circ\phi_a$ for all $a\in \bA$.

\item  By a \emph{Drinfeld $\bA$-module} over $R$, we mean an Anderson $\bA$-module of dimension $1$ over $R$.
\item \label{DefAndersonAModule_D}
For $a\in \bA$, define $\mcE[a]:=\ker(\phi_a\colon \mcE\to \mcE)$. This is a closed subgroup scheme.
\end{enumerate}
\end{Def}
\begin{rmk}
    If $\bA = \FF_q[t]$, we say it is an Anderson $t$-module instead of Anderson $\bA$-module.
\end{rmk}

As usual, $\mcE$ defines an $\bA$-module structure on $\mcE(R) = R^\ell$ through $$a\cdot r: = \mcE_a(r), \quad\quad \forall a\in \mcA,\ r\in \mcE(R).$$ Correspondingly, the Lie algebra $\Lie_\mcE(R) = R^\ell$ has $\bA$-module structure via $\partial \mcE_a$, i.e. $$a\cdot r: = \partial \mcE_a(r), \quad\quad \forall a\in \bA,\ r\in \Lie_\mcE(R).$$

If we write $\dis \mcE_a = \sum_{k=0}^{\ell}E_k(a)\tau^k$ with $E_k(a)\in \Mat_n(R)$, then one has $\partial \mcE_a = E_0(a)$ and $(E_0(a)-\iota(a)\cdot \id)^n = 0$. One may write $\partial \mcE_a = E_0(a) = \iota(a)\cdot\id +N$, where $N:\Lie_\mcE\to \Lie_\mcE$ is a nilpotent homomorphism of Lie algebras over $R$.

As pointed out in \cite[\S 3]{hartl2017isogenies}, when $R$ is a perfect field and $\bA = \FF_q[t]$, this definition is equivalent to Anderson's original definition of $t$-modules.

Let $L$ be a perfect $\bA$-field containing $\FF_q$ with the structure map $\iota: \bA\to L$. By \cite[\S 5]{hartl2020pinks}, every Anderson $\bA$-module $(\mcE,\phi)$ of dimension $d$ over $L$ admits a unique exponential function (up to a constant multiple) $$\exp_\mcE\colon\Lie_\mcE(L)\to \mcE(L)$$ such that $\phi_a\circ \exp_\mcE = \exp_\mcE\circ \partial\phi_a$  for all $a\in \bA$.

Choose a coordinate system $v:\mcE\stackrel{\sim}{\to}\GG_{a,L}^d$ of $\FF_q$-module schemes and consider the induced isomorphism $\partial v:\Lie_\mcE(L)\stackrel{\sim}{\to} L^d$ on tangent spaces. Then the map $\exp_{\mcE,v}$ defined by the following commutative diagram $$\xymatrix{
\Lie_\mcE\ar[r]^{\exp_\mcE}\ar[d]_{\partial v} & \mcE(L)\ar[d]^v\\
L^d\ar[r]_{\exp_{\mcE,v}} & \GG_{a,L}^{\oplus d}
}$$
is of the form $$\exp_{\mcE,v}(\bm{x}) = \sum_{i=0}^\infty E_i \bm{x}^{(i)}$$ for some matrices $E_i\in \Mat_{d\times d}(L)$  and it converges for all $x\in L^d$. We will normalize $\exp_\mcE$ by setting $E_0 = \operatorname{I}_d$ throughout this article.

This entire function was constructed by Anderson in \cite[Theorem~3]{anderson1986t} for $\bA=\FF_q[t]$. See \cite[\S 5]{hartl2020pinks} for the passage to general $\bA$. 

We define the \textit{period lattice} of $\mcE$ to be $\Lambda_{\mcE}:=\ker(\exp_\mcE)\subseteq \Lie_\mcE$. It has an $\bA$-module structure via $\partial\phi_a$, i.e. $a\cdot \bm{x}: = \partial\phi_a(\bm{x})$ for all $a\in \bA$ and $\bm{x}\in \Lie_\mcE$. Moreover, as pointed out in \cite[\S 5]{hartl2020pinks}, the exponential map $\exp_\mcE$ and the period lattice $\Lambda_\mcE$ are covariant functorial in $\mcE$. 

For any Anderson $\bA$-module $(\mcE,\phi)$, Anderson attached the following $\bA_R$-module to it in \cite{anderson1986t}. See also \cite[Definition 3.4]{hartl2017isogenies}.

Let $M:=\MM(\mcE):=\hom_{\FF_q, R}(\mcE,\GG_{a, R})$ be the set of $\FF_q$-linear homomorphisms of group $R$-schemes. We can equip $M$ an $\bA_R$-action via
\[
(r\otimes a)\cdot m\mapsto r\cdot m\circ\phi_a 
\] for all $a\in \bA$ and $r \in R$.

Define a $\tau$-semilinear endomorphism $$\begin{aligned}
    \tau:M&\to M\\
m&\mapsto\Frob_{q,\GG_a}\circ m.
\end{aligned}$$ This map corresponds an $\bA_{R}$-linear homomorphism $\tau_M: \tau^* M\to M$. 

If $M = \MM(\mcE)$ is finite projective of rank $r$ as an $\bA_R$-module, then we say $\mcE$ is \emph{abelian}. The rank of $\MM(\mcE)$ is called the \emph{rank of $\mcE$} and is denoted $\rk(\mcE)$. One can show that (See \cite[Proof of Theorem 3.5]{hartl2017isogenies}) in this case $(M,\tau_M)$ is an abelian $\bA$-motive over $R$ and will be called the \emph{$\bA$-motive associated to $\mcE$}. Denote by $\bA\abm(R)$ the category of abelian Anderson $\bA$-modules over $R$.
\begin{eg}
    Let $X=\PP^1_{\FF_q}$, $\bA=\FF_q[t]$ and $K = \FF_q(\theta)$ with the structure map $\iota:\bA\to K$, $t\mapsto \theta$. Consider the Drinfeld module $(\GG_{a, K},C)$ with $C_t=\theta+\tau$. It is called the \emph{Carlitz module}, and usually denoted by $C$. Since the \emph{Carlitz $t$-motive} $\MM(C)=(K[t],\tau_M=t-\theta)$, we see that $C$ is abelian of rank $1$.
\end{eg}

We have the following result proved by Anderson in \cite{anderson1986t} for the case $R$ is a perfect field and $\bA = \FF_q[t]$. See \cite[Theorem 5.7]{hartl2020pinks} and \cite[Theorem 3.5]{hartl2017isogenies} for the proof of general case.

\begin{thm}[cf. {\cite[Theorem 3.5]{hartl2017isogenies}}]\label{ThmAnderson}
The contravariant functor $$\begin{aligned}
    \MM: \bA\abm(R) &\to \bA\fgeffmot(R)\\
    \mcE &\mapsto \MM(\mcE)
\end{aligned}$$
 is an anti-equivalence of categories, where $\bA\fgeffmot(R)$ is defined in Remark \ref{rmk_motives}.
\end{thm}

\subsubsection{Tate modules and $\mfl$-adic cohomology}
Now, we briefly review the theory of $\mfl$-adic Tate modules and $\mfl$-adic cohomology for Anderson modules, following \cite[\S 5.7]{hartl2020pinks} and \cite[Chapter 5]{Goss}.

Let $\mcE$ be an abelian Anderson $\bA$-module over $L$ with exponential function $\exp_\mcE\colon\Lie_\mcE\to \mcE(L)$ and let $\Lambda_\mcE:=\ker(\exp_\mcE)$.

Let $\mfl$ be a finite place of $X$, that is a closed point $\mfl\in X - \infty$ and let $\bA_\mfl$ be the $\mfl$-adic completion of $\bA$, and $\bK_\mfl$ the fraction field of $\bA_\mfl$. Let $$T_\mfl(\mcE):=\hom_\bA\bigl(\bK_\mfl/\bA_\mfl,\,\mcE(L)\bigr)$$ be the \emph{$\mfl$-adic Tate module} of $\mcE$. 
\begin{Def}
    The \emph{$\mfl$-adic homology realization} of $\mcE$ with coefficients in $\bA_\mfl$ (resp. $\bK_\mfl$) is defined to be $H_{1,\mfl}(\mcE,\bA_\mfl):=  T_\mfl(\mcE)$ (resp. $H_{1,\mfl}(\mcE,\bK_\mfl):= T_\mfl(\mcE)\otimes_{\bA_\mfl}\bK_\mfl$).

    The \emph{$\mfl$-adic cohomology realization} of $\mcE$ with coefficients in $\bA_\mfl$ (resp. $\bK_\mfl$) is defined to be $H^1_\mfl(\mcE,\bA_\mfl)  :=  \hom_{\bA_\mfl}(T_\mfl(\mcE),\bA_\mfl)$ (resp. $H^1_\mfl(\mcE,\bK_\mfl)  :=  \hom_{\bA_\mfl}(T_\mfl(\mcE),\bK_\mfl)$).
\end{Def}
We collect the following result in \cite{Goss} and \cite{hartl2020pinks} for future use.
\begin{prop}[cf. {\cite[Theorem 5.6.8]{Goss}}]\label{coh_and_mod}
Let $\mcE$ be an Anderson $\bA$-module over $L$ of rank $r$ and dimension $d$. Then,
    \begin{enumerate}
        \item $H_{1,\mfl}(\mcE,\bA_\mfl)$ and $H_\mfl^1(\mcE,\bA_\mfl)$ are free $\bA_\mfl$-modules of rank $r$; $H_{1,\mfl}(\mcE,\bK_\mfl)$ and $H_\mfl^1(\mcE,\bK_\mfl)$ are $\bK_\mfl$-vector spaces of dimension $r$.
        \item  Choose an integer $m$ such that $\mfl^m = (a)\subset \bA$ is principal, then we have isomorphisms $\bA[\tfrac{1}{a}]/\bA\isoto \bK_\mfl/\bA_\mfl$ and $$
T_\mfl(\mcE) \isoto \inlim_n \mcE[a^n](L).$$
\item We have an isomorphism of $\bA_\mfl[G_L]$-modules $$H_{1,\mfl}(\mcE,\bA_\mfl)\cong \hom_{\bA_\mfl}(H^1_\mfl(\MM(\mcE)),\bA_\mfl)\cong H_{1,\mfl}(\MM(\mcE),\bA_\mfl).$$
    \end{enumerate}
\end{prop}

\subsection{Taelman class number formula}\label{Taelmans_class_number}c
The exploration of a function field analogue to the classical class number formula began with Carlitz in 1935 \cite{carlitzclass}, who related special values of the Goss zeta value $\zeta_\bA(1)$ to the Carlitz logarithm at $1$. However, a general formula for Drinfeld modules remained elusive until 2009, when Taelman proposed and subsequently proved a ground-breaking formula for Drinfeld $\FF_q[t]$-modules in 2012. See \cite{taelman2009special, taelman2012special}. His work expressed the special $L$-value as a product of a regulator and the class number, directly paralleling the Dedekind zeta function's behavior. 

Following Taelman's work, the theory was rapidly generalized to higher dimensions by Fang and Demeslay for Anderson $\FF_q[t]$-modules \cite{fang2015special, demeslay2015class}. For the significantly more complex case of general coefficient rings $\bA$, Mornev utilized shtuka cohomology to treat Drinfeld modules with good reduction everywhere \cite{mornev2018}, while Anglès, Ngo Dac, Pellarin and Tavares Ribeiro established the formula in full generality for admissible Anderson modules using the theory of Stark units in a series of papers \cite{angles2014functional,angles2017arithmetic,angles2017stark,angles2022class}. To state the formula in this setting, where the regulator and class number are involved, we require the notion of the Fitting ideals and ratio of co-volumes.

\subsubsection{Fitting ideals}
In this subsection, we recall the definition and basic properties of Fitting ideals. We follow the standard exposition found in the literature, for example \cite{BourbakiAlgComCh5to7Y06,EisenbudCA,stacks-project} and \cite{LangAlg}.

Let $R$ be a commutative ring and let $M$ be a finitely presented $R$-module. Note that if $R$ is Noetherian, being finitely presented is equivalent to being finitely generated. By definition, there exists a presentation of the form:
$$R^a \stackrel{T}{\longrightarrow} R^b \longrightarrow M \longrightarrow 0,$$
where the map $R^a \to R^b$ is represented by a matrix $T$ with entries in $R$.

\begin{Def}
The \emph{Fitting ideal} of $M$, denoted by $\Fitt_R(M)$, is defined as the ideal of $R$ generated by the minors of size $b \times b$ of the matrix $T$.
\begin{itemize}
    \item If $b > a$, we set $\Fitt_R(M) = 0$.
    \item If $b \le a$, it is the ideal generated by all $b \times b$ minors.
\end{itemize}
It is a standard fact that this ideal is independent of the choice of the finite presentation of $M$. In the literature, this is sometimes referred to as the initial or zero-th Fitting ideal.
\end{Def}

The Fitting ideal serves as a refined invariant of the module structure, carrying more information than the annihilator ideal. See \cite{angles2022class} for further properties of Fitting ideals.

\subsubsection{The symbol $\left[\frac{\mcO_E}{\cdot}\right]_A$}
	For a Dedekind domain $R$, denote by $\mcI(R)$ the abelian group of fractional ideals of $R$ and $\mcP(R)$ the subgroup of principal fractional ideals of $R$. Denote by $\Cl(R) = \mcI(R)/\mcP(R)$ the class group of $R$ and $h=h_R=|\Cl(R)|$ the \textit{class number} of $R$ whenever it is finite. Also, we denote by $\mcI(R)_+$ the semigroup of integral ideals of $R$, i.e. $\mcI(R)_+ = \{\mfa\in \mcI(R): \mfa\subset R\}$.

	Following \cite{angles2017stark}, we will define a completely multiplicative map $$\left[\frac{\mcO_E}{\cdot}\right]_A: \mcI(\mcO_E)_+\to \conj{K}_\infty^*$$ such that $\dis \left[\frac{\mcO_E}{\wp}\right]_A\in A$ for all $\wp\in \Spec(\mcO_E)$.
	
	For $I\in\mcI(A)$ and $I\subset A$, define $\deg I:= \dim_{\FF_q}A/I$ and extend it to a group homomorphism $\deg: \mcI(A)\to \ZZ$. Clearly, for $x\in K^\times$, $\deg(x) = \deg(xA) = -d_\infty v_\infty(x)$. Since $\Cl(A)$ is a finite group, for any $I\in \mcI(A)$, there exists an integer $h\geq 1$ such that $I^h = xA$ for some $x\in K^\times$. Set $\langle I\rangle = \langle x\rangle^{1/h}\in U_\infty.$ Then we have a group homomorphism $[\cdot]: \mcI(A)\to \conj{K}_\infty^\times$ via $I\mapsto \langle I\rangle \pi^{-\frac{\deg I}{d_\infty}}$ such that $[xA] = \frac{x}{\sgn(x)}$ for all $x\in K^\times$. By the construction, one clearly sees that $\sgn([I]) = 1$ for any $I\in \mcI(A)$. If $M$ is a finitely generated $A$-module, we set $[M]_A:=[\fitt_A(M)]$, where $\fitt_A(M)$ denotes the Fitting ideal of $M$.
	
	It is well-known that $\fitt_A(\mcO_E/\mfa) = N_{E/K}(\mfa)$ for all $\mfa\in \mcI(\mcO_E)_+$ and $N_{E/K}: \mcI(\mcO_E)\to \mcI(A)$ is a group homomorphism such that $$N_{E/K}(\wp) = \mfp^{f(\wp|\mfp)},$$ where $\mfp = \wp\cap A$ and $f(\wp|\mfp) = [\mcO_E/\wp: A/\mfp]$ is the inertia degree of $\wp|\mfp$. Thus, we have a completely multiplicative map $$\left[\frac{\mcO_E}{\cdot}\right]_A: \mcI(\mcO_E)_+\to \conj{K}_\infty^\times,$$ $$\wp\mapsto [N_{E/K}(\wp)].$$
	\begin{lem}\label{normprin}
		Let $H$ be the Hilbert class field of $A$ and suppose $E\supseteq H$. For any $\wp\in \Spec(\mcO_E)$, $N_{E/K}(\wp)$ is a principal ideal of $A$.
	\end{lem}
	\begin{proof}
		Let $P = \wp\cap \mcO_H$ and $f(P|\mfp)$ denote the inertia degree of $P|\mfp$. Then $P$ is lying above $\mfp$, i.e. $P\cap A = \mfp$. Consider the Artin reciprocity map $$\Phi_{H/K}: \mcI(A)\to \Gal(H/K),$$ whose kernel is $\mcP(A)$. By class field theory, we know that $N_{H/K}(P) = \mfp^{f(P|\mfp)}\in \ker(\Phi_{H/K}) = \mcP(A)$. Since $E\supseteq H$, we see $N_{E/K}(\wp) = \mfp^{f(\wp|\mfp)} = \mfp^{f(\wp|P)f(P|\mfp)}$ is a principal ideal of $A$.
	\end{proof}

	\begin{cor}
		Let $E\supseteq H$ as above. For any integral ideal $I\subseteq \mcO_E$, we have $\dis \left[\frac{\mcO_E}{I}\right]_A\in A_+$.
	\end{cor}

\subsubsection{Ratio of covolumes}
Let $V$ be a finite-dimensional vector space over $K_\infty$. An $A$-submodule $M \subset V$ is termed an $A$-lattice if it is discrete and generates $V$ over $K_\infty$. 

Let $n = \dim_{K_\infty} V$ and $M, N$ be two $A$-lattices in $V$. Since $A$ is a Dedekind domain, there exist two $K_\infty$-bases $e_1,\cdots,e_n$ and $f_1,\cdots,f_n$ of $V$ such that $$M = \bigoplus_{i=1}^{n-1}Ae_i\oplus Ie_n$$ and $$N = \bigoplus_{i=1}^{n-1}Af_i\oplus Jf_n,$$ where $I$ and $J$ are two non-zero fractional ideals of $A$. Let $\gamma:V\to V$ be the $K_\infty$-linear map given by $\gamma: e_i\mapsto f_i$. We set $$[M:N]_A = \det(\gamma) JI^{-1}.$$

We have the following properties:
\begin{enumerate}
    \item $[M:N]_A = [N:M]_A^{-1}$;
    \item $[M_1:M
_3]_A = [M_1:M_2]_A[M_2:M_3]_A$ for three $A$-lattices $M_1,M_2,M_3$ in $V$;
\item if $N \subset M$, then $[M: N]_A = \Fitt_A(M/N)$.
\end{enumerate}

This invariant allows for a unified formulation of the class number formula. For the rest of this section, we assume $\bA = \FF_q[t]$ for simplicity. Note that in the setting where $\bA = \mcO_X(X-\infty)$ for a general curve $X$, this has been worked out in \cite{angles2022class}. Let $\mcE$ be an admissible Anderson module defined over $\mcO_E$.

We define $E_\infty = E\otimes_KK_\infty$ and view $E$ as contained in $E_\infty$ via the diagonal embedding $x\mapsto x\otimes 1$. Furthermore, the $\exp_\mcE$ induces a homomorphism of $\bA$-modules $$\exp_\mcE: \Lie_\mcE(E_\infty)\to \mcE(E_\infty).$$
We define
\begin{enumerate}
    \item the unit module $$U(\mcE/\mcO_E) = \exp_\mcE^{-1}(\mcE(\mcO_E)) = \{x\in \Lie_\mcE(E_\infty):\exp_\mcE(x)\in \mcE(\mcO_E)\},$$ which is an $A$-lattice in $\Lie_\mcE(E_\infty)$
    \item the class module $$H(\mcE/\mcO_E):= \frac{\mcE(E_\infty)}{\mcE(\mcO_E)+\exp_\mcE(\Lie_\mcE(E_\infty))},$$ which fits into the exact sequence
\begin{equation} \label{eq:exact_sequence}
    0 \longrightarrow U(\mcE/\mcO_E) \longrightarrow \Lie_\mcE(E_\infty) \xrightarrow{\exp_\mcE} \frac{\mcE(E_\infty)}{\mcE(\mcO_E)} \longrightarrow H(\mcE/\mcO_E) \longrightarrow 0.
\end{equation}
This is a finite $\bA$-module
\end{enumerate}

The Taelman class number formula asserts that the special $L$-value $$\mcL(\mcE/\mcO_E): = \prod_{\wp\subseteq \mcO_E}\frac{[\Lie_{\conj{\mcE}}(\mcO_E/\wp)]_A}{[\conj{\mcE}(\mcO_E/\wp)]_A} \in K_\infty$$ satisfies
\begin{equation} \label{eq:taelman_formula}
    \mcL(\mcE/\mcO_E) = \Reg(\mcE/\mcO_E) \cdot h(\mcE/\mcO_E),
\end{equation} where $\Reg(\mcE/\mcO_E):= [\Lie_\mcE(\mcO_E) : U(\mcE/\mcO_E)]_A$ is called the regulator of $\mcE/\mcO_E$ and $h(\mcE/\mcO_E)$ is the monic generator of $\Fitt_A \left( H(\mcE/\mcO_E) \right)\subseteq A$.

\section{Goss $L$-series}\label{sec:GossLseries}

\subsection{Goss $L$-series of $\bA$-motives and Anderson $\bA$-modules}
Let $E$ be a finite extension of $K$ and $F/\FF_q$ an extension. We call a place of $E$ finite if it corresponds to a prime of $\mcO_E$ and infinite otherwise.
Let $(\mcE,\phi)$ be an Anderson $\bA$-module defined over $E$. For simplicity, we will use the notation $\phi/E$ throughout this paper.

By an \emph{Artin representation} $(\rho,V)$ over $E$ with coefficients in $F$, we mean a continuous Galois representation $\rho: \Gal(E^\sep/E)\to \aut_{F}(V)$ with $V$ an $F$-vector space. Here, continuity means $\rho$ factors through a finite quotient. In other words, there is a finite Galois extension $L/E$ such that $\rho|_{\Gal(E^\sep/L)}$ is trivial and hence can be identified with a representation of $\Gal(L/E)$.

In this paper, we shall always consider the case that $(\rho,V)$ has coefficients in $\conj{\FF}_q$, i.e. after choosing an $\conj{\FF}_q$-basis of $V$, $$\rho: \Gal(E^\sep/E)\to \GL_n(\conj{\FF}_q).$$ 

Let $G_E = \Gal(E^\sep/E)$. We denote the primes of $K$, $E$, and $L$ by $\mfp$, $\wp$, and $\mfP$, respectively. Generally, we do not require that $p \nmid [L:E]$; however, in certain applications (e.g., Corollary~\ref{corE}), we will assume this condition to avoid technical complications.
\begin{Def}
	By a system of $\mfl$-adic representations $V = \{V_\mfl\}_\mfl$ over $E$, we mean a collection $\rho = (\rho_\mfl)$ of continuous homomorphisms $\rho_\mfl: G_E\to \GL(V_\mfl)$ one for every prime $\mfl$ of $K$, where the $V_\mfl$ are finite dimensional vector spaces over $K_\mfl$.
	
	We say such a system is \emph{strictly compatible} if it satisfies the condition that there exists a finite set $S$ of places of $E$ such that 
	\begin{enumerate}
		\item[(i)] for every finite place $\wp\notin S$ and all $\mfl$ not under $\wp$, the representation $\rho_\mfl$ is unramified at $\mfl$.
        \item[(ii)] for these $\mfl$ and $\wp$ the characteristic polynomial of the Frobenius $\Frob_\wp$ at $\wp$ has coefficients in $K$ and is independent of $\mfl$.
\end{enumerate}

\end{Def}
\begin{eg}\label{egDrin}
    Suppose $\iota: \bA\to E$ is an $\bA$-field that has generic characteristic, i.e. $\ker(\iota) = 0$. Let $\varphi/E$ be a Drinfeld $\bA$-module defined over $E$ of rank $r$. Suppose $\varphi$ has good reduction everywhere.  For a prime $\mfl$ of $K$, we write $T_\mfl(\varphi)$ for the $\mfl$-adic Tate module of $\varphi$, $$T_\mfl(\varphi) = \inlim_n \varphi[\mfl^n].$$ We write $V_\mfl(\varphi) = T_\mfl(\varphi)\otimes_{A_\mfl} K_\mfl$, which is an $K_\mfl$-vector space of dimension $r$. It is well-known that $\{V_\mfl(\varphi)\}_\mfl$ forms a strictly compatible system of $\mfl$-adic Galois representations of $G_E$. (See e.g. \cite[\S 4.12, \S 8.6]{Goss}.)
\end{eg}

To a compatible system of algebraic $\mfl$-adic Galois representations $\rho = \{\rho_\mfl\}_\mfl$ over $E$, we can associate an $L$-series $L(\rho/E,s)$ as follows. 

For a prime $\wp$ of $E$, we choose a prime $\mfl$ of $K$ that is not under $\wp$ and consider the characteristic polynomial $$P_\wp(\rho,X) = \det(X-\rho_\mfl(\Frob_\wp) |V_\mfl^{I_{\wp}})\in K[X],$$ which does not depend on $\mfl$ and $\conj{\wp}$ over $\wp$. Let $$\begin{aligned}
	Q_\wp(\rho,X) &= X^{\deg P_\wp}P_\wp(V/E, X^{-1}) \\
	&= \det(1-X\rho_\mfl(\Frob_{\conj{\wp}})|V_\mfl^{I_{\wp}})\in K[X]
\end{aligned}$$ be the reciprocal polynomial of $P_\wp(\rho,X)$. 

We adopt Goss's framework for $L$-series, which extends the domain of definition from $\ZZ$ to the topological space $\SS_\infty$ \cite{goss1992lseries, Goss}. Also see \cite{kramermiller2023} for additional perspective. While this theory relies on the exponentiation of ideals $I^s$ for $s \in \SS_\infty$, we restrict our attention primarily to integer values $s = n \in \ZZ$. 

Let $\pi_\infty$ a uniformizer for $v_\infty$. Let $\sgn:K_\infty^*\to \FF_\infty^*$ be the sign function given by $\sgn(x) = \zeta$ if $x = \zeta\cdot \pi_\infty^m\cdot \langle x\rangle$ with $m\in \ZZ$, $\zeta\in \FF_\infty^*$ and $\langle x\rangle$ the 1-unit part of $x$. We say $x\in K_\infty^*$ is positive if $x$ has sign 1, i.e. $\sgn(x) = 1$. Let $\mcP$ be the subgroup of principal ideals of $\mcI(\bA)$, the group of fractional ideals of $\bA$.  Let $\mcP^+\subseteq \mcP$ be the subgroup of principal fractional ideals generated by positive elements. Let $\hat{U}_1\subseteq \CC_\infty^*$ be the group of 1-units. Let $\pi_*\in \CC_\infty^*$ be a fixed $d_\infty$-root of $\pi_\infty$. Then by \cite[Corollary 8.2.4]{Goss}, the map $\mcP\to \hat{U}_1$, $(\alpha)\mapsto \langle\alpha\rangle$ extends to a map $\langle\cdot \rangle: \mcI(\bA)\to \hat{U}_1$ uniquely. For any fractional ideal $I\subseteq \bA$, we define $I^n: = \pi_*^{-n\deg I}\langle I\rangle^n\in \hat{U}_1\subseteq \CC_\infty^*$.

Let $\wp$ be a prime ideal of $\mcO_E$ lying above the prime $\mfp$ of $A$. We denote the norm ideal by $n\wp = \mfp^{f(\wp|\mfp)}$. If $E$ contains the Hilbert class field of $A$, which holds trivially when $\bA = \FF_q[t]$, then the ideal $n\wp$ is principal. In this case, its unique monic generator is given by the symbol
\[
    \left[\frac{\mcO_E}{\wp}\right]_A \in A.
\]
and if $s\in \SS_\infty$ is an integer, then $(n\wp)^s=\left[\frac{\mcO_E}{\wp}\right]_A^s$ in the sense of Goss's exponentiation of ideals.
\begin{Def} Let $S$ be a finite set of places of $E$ containing all infinite places.
	Let $\wp\notin S$ be a prime of $E$ over $\mfp$ of $K$.
	For any integer $s$, we define the local factor at $\wp$ by $$L_\wp(\rho,s) = Q_\wp(\rho,n\wp^{-s})^{-1}.$$ 
	
	The \emph{Goss $L$-series} attached to $\rho$ is given by $$L_S(\rho, s) = \prod_{\wp\notin S}L_\wp(\rho,s).$$
\end{Def}

This converges to an element of $\CC_\infty$ for all sufficiently large integers $n$. In the case of $\bA = \FF_q[t]$, it takes values in the subfield $K_\infty$. See \cite{taelman2009special} for example.

By the theory of $\bA$-motives, for an $\bA$-motive $M$ over $E$, the system $\{\rho_{M,\mfl}\}$ forms a strictly compatible system, where $\rho_{M,\mfl}:G_E\to \GL(H^1_\mfl(M,K_\mfl))$, see \cite[Proposition 3.36]{hartl2020pinks}. We can thus define a motivic $L$-series $L_S(M,s)$.

Let $\mfl$ and $\wp$ be as above and $\Frob_\wp$ the arithmetic Frobenius at $\wp$. Let $I_\wp$ be the inertia subgroup at $\wp$. Let $M$ be an $\bA$-motive over $E$. We define $$P_\wp(M,X) = \det(X - \rho_{M,\mfl}(\Frob_\wp)|H_\mfl^1(M,K_\mfl)^{I_{\wp}})$$ and $Q_{\wp}(M,X)$ the reciprocal polynomial of $P_\wp(M,X)$. Similarly, we define $$P^\vee_\wp(M,X) = \det(X - \rho_{M,\mfl}(\Frob^{-1}_\wp)|H_\mfl^1(M,K_\mfl)^{I_{\wp}})$$ and $Q^\vee_{\wp}(M,X)$ the reciprocal polynomial of $P^\vee_\wp(M,X)$. See \cite[\S 4.12]{Goss} or \cite[\S 3.6]{papikian2023drinfeld} for additional perspective.
\begin{Def}
    Let $S$ be a finite set of finite places. For any integer $s$, we define the following Goss $L$-series $$L_S(M,s) = \prod_{\wp\notin S}Q_\wp(M,n\wp^{-s})^{-1}$$ and $$L_S(M^\vee,s) = \prod_{\wp\notin S}Q^\vee_\wp(M,n\wp^{-s})^{-1}.$$ 
\end{Def}

Similarly, let $(\mcE,\phi)$ be an Anderson $\bA$-module over $E$. Then, we have a system of strictly compatible $\mfl$-adic Galois representations
$$\rho_{\mcE,\mfl}: G_E\to \GL(V_\mfl(\mcE)).$$
We define $$P_\wp(\mcE,X) = \det(X - \rho_{\mcE,\mfl}(\Frob_\wp)|V_\mfl(\mcE)^{I_{\wp}})$$ and $Q_{\wp}(\mcE,X)$ the reciprocal polynomial of $P_\wp(\mcE,X)$. We further define $$\begin{aligned}
    P^\vee_\wp(\mcE,X) &= \det(X -\rho_{\mcE,\mfl}(\Frob^{-1}_\wp)|V_\mfl(\mcE)^{I_{\wp}})\\
    & = \det(X -\rho^\vee_{\mcE,\mfl}(\Frob_\wp)|V^\vee_\mfl(\mcE)^{I_{\wp}})
\end{aligned}$$ and $Q^\vee_{\wp}(\mcE,X)$ the reciprocal polynomial of $P^\vee_\wp(\mcE,X)$.
\begin{Def}
    Let $S$ be a finite set of finite places. For any integer $s$, we define the following Goss $L$-series $$L_S(\mcE,s) = \prod_{\wp\notin S}Q_\wp(\mcE,n\wp^{-s})^{-1}$$ and $$L_S(\mcE^\vee,s) = \prod_{\wp\notin S}Q^\vee_\wp(\mcE,n\wp^{-s})^{-1}.$$ 
\end{Def}

For the question of convergence, see \cite[\S 2.9]{taelman2009special}.

Let $M = \MM(\mcE)$ be its associated $\bA$-motive. By Proposition \ref{coh_and_mod}(3), we have $$H^1_{\mfl}(M,K_\mfl)\cong H_\mfl^1(\mcE,K_\mfl)\cong V^\vee_\mfl(\mcE).$$ Thus, we have \begin{equation}\label{L-series_motivic_module}
    L_S(\MM(\mcE),s) = L_S(\mcE^\vee,s)
\end{equation} and similarly \begin{equation}
    L_S(\MM(\mcE)^\vee,s) = L_S(\mcE,s).
\end{equation} Compare with \cite[\S Example 8.6.6.2]{Goss} or \cite{taelman2009special}.
\subsection{Computation of Goss $L$-series}
We first review Gazda's theory on maximal integral models, which serves as an analogue of N\'eron models for abelian varieties. We restrict our attention to the global situation following \cite[\S 4]{Gazda21motcoh}.
	
	Let $R$ be a Dedekind domain which is an $\bA$-algebra via $\iota: \bA\to R$. Let $F = \Frac(R)$. We assume that $\iota$ is injective and compatible with the prime ideals in the sense that $\iota^{-1}(\mfp)=(0)$ if and only if $\mfp=(0)$. A typical example is when $F$ is a local or global function field and $R$ is the ring of integers or its localization at a prime.
	
	\subsubsection{Maximal integral models and reductions of $\bA$-motives}
	
	\begin{Def}[Gazda, see {\cite[\S 4]{Gazda21motcoh}}]
    Let $(M,\tau_M)$ be an abelian $\bA$-motive over $F$. An \emph{$R$-integral model} for $M$ is a finitely generated $\bA_R$-submodule $\Lambda \subset M$ such that:
		\begin{enumerate}
			\item $\Lambda$ generates $M$ over $F$;
			\item $\tau_M(\tau^*\Lambda) \subset \Lambda[\mfj^{-1}]$.
		\end{enumerate}
		We say $\Lambda$ is \emph{maximal} if it is not strictly contained in any other $R$-integral model.
	\end{Def}
	
	For a maximal ideal $\mfp \subset R$, let $R_\mfp$ denote the completion of $R$ at $\mfp$, and $F_\mfp$ its fraction field. The existence and local-global compatibility of maximal models are established by Gazda:
	
	\begin{thm}[Gazda, see {\cite[\S 4]{Gazda21motcoh}}] \label{thm:maximal-model-existence}
    Let $M$ be an abelian $\bA$-motive over $F$.
		\begin{enumerate}
			\item There exists a unique maximal $R$-integral model for $M$, denoted by $M_R$.
			\item Let $M_{R_\mfp}$ be the unique maximal $R_\mfp$-integral model of the base change $M_{F_\mfp}$. Then we have
			\[ M_R \otimes_R R_\mfp \cong M_{R_\mfp}. \]
			\item $M_R$ is a projective $\bA_R$-module.
		\end{enumerate}
	\end{thm}
	
	We now define good reduction. Let $\mfp$ be a maximal ideal of $R$ and $\FF_\mfp = R/\mfp$ be the residue field.
	\begin{Def}
		We say that $M$ has \emph{good reduction at $\mfp$} if the maximal model $M_{R_\mfp}$ satisfies
		\[ \tau_M(\tau^* M_{R_\mfp})[\mfj^{-1}] = M_{R_\mfp}[\mfj^{-1}]. \]
		In this case, we say $\mfp$ is a \emph{good prime} for $M$ and define the \emph{reduction} of $M$ at $\mfp$ as the $\bA$-motive $(M_\mfp, \tau_\mfp)$ over $\FF_\mfp$, where $M_\mfp := M_{R_\mfp} \otimes_{R_\mfp} \FF_\mfp$, and $\tau_\mfp$ is induced by the diagram:
		$$
		\xymatrix{
			\tau^*M_{R_{\mfp}}[\mfj^{-1}] \ar[r]^{\sim}\ar@{->>}[d] & M_{R_{\mfp}}[\mfj^{-1}] \ar@{->>}[d]\\
			\tau^*M_{\mfp}[\mfj^{-1}] \ar[r]^{\tau_{\mfp}} & M_{\mfp}[\mfj^{-1}]
		}
		$$
	\end{Def}
	Consequently, $\tau_\mfp$ is an isomorphism, hence $M_\mfp$ is an $\bA$-motive over $\FF_\mfp$.
	\subsubsection{Base change of Anderson modules}
	Let $R$ and $S$ be two $\bA$-algebras with structure maps $\iota_R: \bA \to R$ and $\iota_S: \bA \to S$. Let $f: R \to S$ be a ring homomorphism compatible with the $\bA$-algebra structures, i.e., $f \circ \iota_R = \iota_S$.
	
	Let $(\mcE, \phi)$ be an abelian Anderson $\bA$-module over $R$ of dimension $d$ such that $\mcE\cong \GG_{a,R}^{\oplus d}$, i.e. we assume that $R' = R$ in the definition. We define the \emph{base change} of $\mcE$ to $S$, denoted by $\mcE \otimes_R S$, or simply $\mcE_S$, as follows:
	\begin{enumerate}
		\item The underlying group scheme is $\mcE_S \cong \mcE \times_{\Spec R} \Spec S\cong \GG_{a,S}^{\oplus d}$.
		\item The action $\phi_S: \bA \to \End_{\FF_q, S}(\mcE_S)$ is given by base change of morphisms in the sense that if
		\[ \phi_a = \sum_{i=0}^N A_i \tau^i \in \Mat_{d\times d}(R[\tau]), \] for $a \in \bA$, then the action of $a$ on $\mcE_S$, via $\phi_{S}$, is given by applying $f$ to the coefficients:
		\[ \phi_{S,a} := \sum_{i=0}^N f(A_i) \tau^i \in \Mat_{d\times d}(S[\tau]), \]
		where $f(A_i)$ is the matrix obtained by applying $f$ to each entry of $A_i$.
	\end{enumerate}
	
	\begin{prop}
		The pair $\mcE_S = (\GG_{a,S}^d, \phi_S)$ defined above is an Anderson $\bA$-module over $S$.
	\end{prop}
	
	\begin{pf}
		We only need to verify the nilpotent condition on the Lie algebra.
		Recall that for $\mcE$, the condition is $(\partial \phi_a - \iota_R(a)\cdot I_d)^d = 0$ in $\Mat_d(R)$.
		Note that $\partial \phi_{S,a} =f(A_0) =  f(\partial \phi_a)$. Apply the homomorphism $f$ to the equation over $R$:
		\[
		f\left( (\partial \phi_a - \iota_R(a)\cdot I_d)^d \right) = f(0) = 0.
		\]
		Since $f$ is a ring homomorphism, it commutes with matrix addition, multiplication. Thus:
		\[
		\left( f(\partial \phi_a) - f(\iota_R(a))\cdot f(I_d) \right)^d = 0.
		\]
		Using $f \circ \iota_R = \iota_S$ and $\partial \phi_{S,a} = f(\partial \phi_a)$, we obtain:
		\[
		(\partial \phi_{S,a} - \iota_S(a)\cdot I_d)^d = 0 \quad \text{in } \Mat_d(S).
		\]
		Thus, $\mcE_S$ satisfies the nilpotent condition on its Lie algebra and is an Anderson $\bA$-module over $S$.
	\end{pf}
\begin{prop}\label{basechangeM}
	Let $f: R \to S$ be a homomorphism of $\bA$-algebras. Let $\mcE$ be an Anderson $\bA$-module over $R$. Let $\mcE_S := \mcE \otimes_R S$ be its base change.
	Then there is a natural isomorphism of $\bA$-motives over $S$:
	\[
	\MM(\mcE_S) \cong S\otimes_R\MM(\mcE)\cong \bA_S\otimes_{\bA_R}\MM(\mcE).
	\]
\end{prop}

\begin{pf}
	By definition, the underlying group scheme of $\mcE$ is isomorphic to $\GG_{a,R}^{\oplus d}$.
	Recall from Lemma \ref{LemmaLucas} that we have an isomorphism of $R$-modules
	\[
	\MM(\mcE) = \hom_{\FF_q, R}(\mcE, \GG_{a,R}) \cong \Mat_{1\times d}(R[\tau]).
	\]
	
	By definition of $\mcE_S$, its underlying scheme is $\GG_{a,S}^{\oplus d}$. Applying Lemma \ref{LemmaLucas} over $S$, we have:
	\[
	\MM(\mcE_S) = \hom_{\FF_q, S}(\mcE_S, \GG_{a,S}) \cong \Mat_{1\times d}(S[\tau]).
	\]
	
	On the other hand, consider the base change of the motive $\MM(\mcE) \otimes_R S$. Using the identification above:
	\[
	S\otimes_R\MM(\mcE)   \cong S\otimes_R\Mat_{1\times d}(R[\tau]).
	\]
	We have a natural map
	\[
	\begin{aligned}
		\lambda_f: S\otimes_R\Mat_{1\times d}(R[\tau]) &\longrightarrow \Mat_{1\times d}(S[\tau]) \\
		s\otimes \left(\sum \bm{v}_i \tau^i\right)  &\longmapsto \sum (s \cdot f(\bm{v}_i)) \tau^i
	\end{aligned}
	\]
	where $f: R \to S$ is the structure map applied to the coefficients.

    Let $\mu: \Mat_{1\times d}(S[\tau]) \to \Mat_{1\times d}(R[\tau])\otimes_R S$ be the homomorphism of $S$-modules given by $$\sum_{k}\bm{v}_k\tau^k\mapsto \sum_k\sum_{j = 1}^dv_{kj} \otimes (\bm{e}_j\tau^k),$$ where $\bm{v}_k = (v_{k1},\cdots,v_{kd})$ and $\bm{e}_1,\cdots,\bm{e}_d\in \Mat_{1\times d}(R)$ is the standard $R$-basis, i.e., $\bm{e}_j= (0,\cdots,1,\cdots,0)$ has $1$ in the $j$-th entry.
    Then, one immediately sees that $\mu\circ \lambda_f = \id$ and $\lambda_f\circ \mu = \id$. Thus, $\lambda_f$ is an isomorphism of $S$-modules.
	
	It remains to verify that this isomorphism respects the $\bA$-module structure and the $\tau$-action.
	The $\tau$-action on $\MM(\mcE)$ corresponds to left multiplication by $\tau$ in the matrix representation. Thus, $$\begin{aligned}
	    \lambda_f\left(\tau\left(s\otimes \sum_i \bm{v}_i\tau^i\right)\right) & = \lambda_f\left(\left(\tau(s)\otimes \sum_i \bm{v}_i^{(1)}\tau^{i+1}\right)\right)\\
        & = \sum_i s^{(1)}f(v_i^{(1)})\tau^{i+1}\\
        & = \tau\left(\lambda_f\left(\sum_i\bm{v}_i\tau^i\right)\right).
	\end{aligned}$$ Hence, $\lambda_f$ is $\tau$-equivariant.
	
    The action of $a \in \bA$ on $\MM(\mcE)$ is given by composition with $\phi_a$ on the right. Thus, for all $a\in\bA$, we have $$\begin{aligned}
        \lambda_f\left(a\cdot \left(s\otimes \sum_i \bm{v}_i\tau^i\right)\right)& = \lambda_f\left( s\otimes \left(\sum_i \bm{v}_i\tau^i\right)\circ \phi_a\right)\\
        & = \left(\sum_i sf(\bm{v}_i)\tau^i\right)\circ \phi_{S,a}\\
        & = a\cdot \lambda_f\left(s\otimes \sum_i \bm{v}_i\tau^i\right)
    \end{aligned}$$

	Therefore, $\lambda_f:S\otimes_R\MM(\mcE)\to \MM(\mcE_S)$ is an isomorphism of Anderson $\bA$-motives.
\end{pf}
	\subsubsection{Integral models and reduction of Anderson modules}
	In this subsection, we develop the notion of integral models for Anderson modules and establish their relation to the reduction theory of $\bA$-motives.
    
	Let $R$ be a Dedekind domain which is an $\bA$-algebra via $\iota: \bA\to R$. Let $F=\Frac(R)$. Let $\mfp \subset R$ be a maximal ideal and let $\FF_\mfp = R/\mfp$ be the residue field.
	
	\begin{Def}\label{def:module-integral-model}
		Let $(\mcE, \phi)$ be an Anderson $\bA$-module of dimension $d$ over $F$. An $R$-\emph{integral model} for $\mcE$ is a pair $(\mcE_R, \alpha)$, where:
		\begin{enumerate}
			\item $\mcE_R = (\GG_{a, R}^{\oplus d}, \psi)$ is an Anderson $\bA$-module over $R$. This means $\psi_a \in \Mat_d(R[\tau])$ for all $a \in \bA$ and satisfies the nilpotent condition over $R$: $(\partial \psi_a - \iota(a) \cdot I_d)^d = 0$.
			\item $\alpha: \mcE_R \otimes_R F \stackrel{\sim}{\longrightarrow} \mcE$ is an isomorphism of Anderson $\bA$-modules over $F$.
		\end{enumerate}
	\end{Def}
	
	\begin{rmk}
		Since the underlying group schemes are isomorphic to a power of additive group schemes, the isomorphism $\alpha$ is given by a matrix $P \in \GL_d(F)$. The condition that $\alpha$ is a morphism of Anderson modules translates to:
		\[ \psi_a = P^{-1} \phi_a P \quad \text{for all } a \in \bA. \]
		Thus, finding an integral model is equivalent to finding a change of coordinates $P \in \GL_d(F)$ such that all conjugated matrices $\psi_a: = P^{-1} \phi_a P$ have coefficients in $R$ and satisfy the nilpotent condition $(\partial \psi_a - \iota(a) \cdot I_d)^d = 0$.
	\end{rmk}
	Let $\mfp$ be a prime of $R$. Denote by $R_{(\mfp)}$ its localization at $\mfp$ of $R$, which is still a (local) Dedekind domain. Let $\FF_\mfp = R/\mfp\cong R_{(\mfp)}/\mfp R_{(\mfp)}$ be the residue field at $\mfp$.
	\begin{Def}\label{def:module-good-reduction}
		We say that $\mcE$ has \emph{good reduction} at $\mfp$ if there exists an $R_{(\mfp)}$-integral model $(\mcE_{R_{(\mfp)}}, \psi)$ such that $\MM(\mcE_{R_{(\mfp)}})$ is finite projective over $\bA_{R_{(\mfp)}}$. We call such $\mfp$ a \emph{good prime} for $\mcE$.
        
        In this case, $\MM(\mcE_{R_{(\mfp)}})$ is an abelian $\bA$-motive over $R_{(\mfp)}$ by \cite[Theorem 3.5]{hartl2017isogenies} and the reduction modulo $\mfp$
		\[ \bar{\mcE} := \mcE_{R_{(\mfp)}} \otimes_{R_{(\mfp)}} \FF_\mfp = (\GG_{a, \FF_\mfp}^d, \bar{\psi}) \]
		is an abelian Anderson $\bA$-module of dimension $d$ over $\FF_\mfp$ and $\rk(\bar{\mcE}) = \rk(\mcE)$.
	\end{Def}
	\begin{rmk}
        The terms \emph{bad reduction} and \emph{bad prime} make sense in the obvious way.
    \end{rmk}
	This definition is compatible with the motivic reduction defined by Gazda.
	
	\begin{lem}\label{lem:reduction-compatibility}
		Let $F$ be a global function field and $R$ its ring of integers. Let $\mfp$ be a maximal ideal of $R$.
		Let $\mcE$ be an abelian Anderson $\bA$-module over $F$.
		
        If $\mcE$ has good reduction at $\mfp$, then its associated motive $\MM(\mcE)$ has good reduction at $\mfp$.
			In this case, let $(\mcE_{R_{(\mfp)}}, \alpha)$ be the integral model of $\mcE$ witnessing the good reduction. Then we have a canonical isomorphism of $\bA$-motives over the residue field $\FF_\mfp$:
			\[ \MM(\bar{\mcE}) \cong \MM(\mcE)_\mfp, \]
			where $\MM(\mcE)_\mfp$ is the reduction of the $\bA$-motive $\MM(\mcE)$ at $\mfp$.
	
	\end{lem}
	\begin{proof}
	    Let $S = R_{(\mfp)}$ be the localization at $\mfp$ and $\FF_\mfp = S/\mfp S = R/\mfp$ be the residue field.
		Assume $\mcE$ has good reduction at $\mfp$. By Definition \ref{def:module-good-reduction}, there exists an integral model $(\mcE_S, \alpha)$ over $S$, where $\mcE_S \cong \GG_{a,S}^{\oplus d}$ as group schemes, such that its reduction $\bar{\mcE} := \mcE_S \otimes_S \FF_\mfp$ is an abelian Anderson $\bA$-module over $\FF_\mfp$ and $\rk(\bar{\mcE}) = \rk(\mcE)$.

        Define
		\[ \Lambda := \MM(\mcE_S) = \hom_{\FF_q, S}(\mcE_S, \GG_{a,S}). \]
        The isomorphism $\alpha: \mcE_S \otimes_S F \stackrel{\sim}{\to} \mcE$ induces an isomorphism of left $F[\tau]$-modules:
		\[ \alpha^*: \MM(\mcE) \xrightarrow{\sim} \MM(\mcE_S \otimes_S F) \cong \Lambda \otimes_S F. \]
		Through $(\alpha^*)^{-1}$, we identify $\Lambda$ as an $S$-submodule of $\MM(\mcE)$ which generates $\MM(\mcE)$ over $F$. Furthermore, write $M = \MM(\mcE)$, then for any $m \in \Lambda$, by definition the action $\tau_M(\tau^*m) = \tau_\Lambda(\tau^*m) \subseteq \Lambda$ because $\mcE_S$ is defined over $S$.
		Thus, $\Lambda$ is an $S$-integral model for the motive $\MM(\mcE)$ in the sense of Gazda. By \cite[Corollary 4.44]{Gazda21motcoh}, $\Lambda$ is the unique maximal integral model for $\MM(\mcE)$. 

        Since $\Lambda$ is an $\bA$-motive over $R_{(\mfp)}$, we have an isomorphism $\tau_\Lambda:\tau^*\Lambda[\mfj^{-1}] \stackrel{\sim}{\to} \Lambda[\mfj^{-1}]$. This means that $\MM(\mcE)$ has a good reduction at $\mfp$. The second statement follows immediately from the definition and Proposition \ref{basechangeM}.
	\end{proof}
\begin{rmk}
    Asking whether the converse to Lemma \ref{lem:reduction-compatibility} holds is equivalent to seeking a Néron-Ogg-Shafarevich criterion for $\mcE$. Due to the technical subtleties involved for general Anderson modules, and because the converse is not needed for our applications, we do not address this question in the present paper.
\end{rmk}
\begin{prop}\label{finite_bad}
    Let $E$ be a global function field containing $\FF_q$, and let $\mcO_E$ be its ring of integers. Let $(\mcE,\phi)$ be an abelian Anderson $\bA$-module of rank $r$ defined over $E$.
    
    Then, the set of maximal ideals $\mfp \subset \mcO_E$ where $\mcE$ has bad reduction is finite.
\end{prop}
\begin{proof}
    Consider the motive $\MM(\mcE)$ over  $E$, identified with $\Mat_{1\times d}(E[\tau])$. Since $\mcE$ is abelian of rank $r$, $M := \MM(\mcE)$ is a finitely generated projective $\bA_E$-module of rank $r$. Let $a_1,\cdots,a_m$ be a set of generators of the $\FF_q$-algebra $\bA$. Then, the Anderson module $\mcE$ is determined by $$\phi_{a_j} = \sum_{i = 0}^{N_j}A_{ij}\tau^i\textnormal{ with }A_{ij}\in \Mat_{d\times d}(E)[\tau], 1\leq j\leq m.$$ Each entry of $A_{ij}$ can be written as a fraction of elements in $R = \mcO_E$ and we fix a choice of these numerators and denominators once for all. 
    
    Let $S_{0}$ be the finite set of primes $\mfp \subset R$ dividing the denominators of these entries. Then, for any $\mfp \notin S_{0}$,  $\mcE$ is defined over the local ring $R_{(\mfp)}$, hence itself is an $R_{(\mfp)}$-integral model of $\mcE$. We denote this integral model by $\mcE_{R_{(\mfp)}}$. Let $S_1$ be the finite set of primes $\mfp$ dividing the determinant of highest leading matrices of $\phi_{a_i}$. Then we see that these matrices are invertible in $\Mat_{d\times d}(R_{(\mfp)})$ and hence $\tau^{m}e_1,\cdots,\tau^me_d\in \MM(\mcE_{R_{(\mfp)}})$ can be written in lower degree terms, where $e_i = (0,\cdots,1,\cdots,0)$ is a standard $E[\tau]$-basis for $\Mat_{d\times d}(E[\tau])$. Thus, for $\mfp\notin S_0\cup S_1$, $\MM(\mcE_{R_{(\mfp)}})$ is finitely generated over $\bA_{R_{(\mfp)}} $. 

    Let $f$ be the product of numerators, denominators appearing in entries of $A_{ij}$ and determinants of leading matrices. Consider $R_f$, then a similar argument yields an $R_f$-integral model for $\mcE$, denoted by $\mcE_{R_f}$ and let $N = \MM(\mcE_{R_f})$. By our construction, as a submodule of $M= \MM(\mcE)$, its image inside $M$ is exactly $ \Mat_{1\times d}(R_f[\tau])$. Note that $N$ is finitely generated over $\bA_{R_f}$ for the same reason but not necessarily projective over $\bA_{R_f}$. However, if we consider the saturation module $\widetilde{N} = \{x\in M: a\cdot x \in N \textnormal{ for all }a\in\bA\}$, then $N\subseteq \widetilde{N}\subseteq M$ and $\widetilde{N}$ is also finitely generated over $\bA_{R_{f}}$. Note that $N\otimes_{R_f} E = M$, we must have $N\otimes_{R_f} E = \widetilde{N}\otimes_{R_f} E = M$. Then, $(\widetilde{N}/N)\otimes_{R_f}E = 0$.

    Let $Q = \widetilde{N}/N$. Since $\widetilde{N}$ is finitely generated over $\bA_{R_f}$, the quotient module $Q$ is also finitely generated over $\bA_{R_f}$. Let $\{q_1, \dots, q_k\}$ be a set of generators of $Q$.

    The condition $Q \otimes_{R_f} E = 0$ implies that for each generator $q_i$, its image in the tensor product is zero. Since $E = \Frac(R_f)$, this implies that each $q_i$ is a torsion element. Specifically, for each $i \in \{1, \dots, k\}$, there exists a non-zero element $s_i \in R_f$ such that $s_i \cdot q_i = 0$ in $Q$. Let $s = \prod_{i=1}^k s_i$. Then $s$ is a non-zero element of $R_f$ that annihilates every generator $q_i$. Consequently, $s$ annihilates the entire module $Q$, i.e., $s Q = 0$. Since $s \in R_f$, we can write $s = r/f^m$ for some $r \in R \setminus \{0\}$ and integer $m \ge 0$. Let $S_2$ be the finite set of prime ideals of $R$ consisting of primes dividing $r$ or $f$. For any prime $\mfp \notin S_0\cup S_1\cup S_2$ the localization map $R_f \to R_{(\mfp)}$ is well-defined and flat as $\mfp$ does not divide $f$. Since $\mfp$ does not divide $r$, the image of $r$ is a unit in $R_\mfp$. Thus, the image of $s = r/f^m$ is a unit in $R_{(\mfp)}$. Thus, $Q \otimes_{R_f} R_\mfp= 0$. Thus, $\MM(\mcE_{R_{(\mfp)}}) = N\otimes_{R_f} R_{(\mfp)}= \widetilde{N}\otimes_{R_f}R_{(\mfp)}$ and it verifies the condition $\MM(\mcE_{R_{(\mfp)}})\cap aM = a\MM(\mcE_{R_{(\mfp)}})$ for all $a\in \bA$. Thus it is projective over $\bA_{R_{(\mfp)}}$ by \cite[Proposition 4.33]{Gazda21motcoh}.

    Therefore, for all $\mfp\notin S_0\cup S_1\cup S_2$, $\mcE$ has a good reduction at $\mfp$.
\end{proof}
Now, we can compute these local $L$-factors by considering the Frobenius action on reductions, thanks to the following result.
\begin{prop}\label{reduction_computation_lemma}
    Let $M$ be an $\bA$-motive over $E$ and $M_{\mcO_E}$ its maximal integral model over $\mcO_E$. Let $\wp\in \Spec \mcO_E$ be a good prime for $M$ and $\FF_\wp$ the residue field at $\wp$. Then, for all but a finite number of places $\mfl$ of $K$, we have $$P^\vee_\wp(M,X) = \det(X- \tau^{\deg \wp}|M_{\wp}),$$ where $M_\wp$ denotes the reduction of $M$ at $\wp$.
\end{prop}
\begin{proof}
     This is \cite[Proposition 7]{taelman2009special}. Also the proof is similar to that of \cite[Theorem 7.3]{gardeyngalois}.
\end{proof}

\begin{cor}\label{reduction_at_good} Assume $\bA = \FF_q[t]$.
    Let $\mcE$ be an abelian Anderson $\bA$-module of dimension $d$ over $E$. Let $\wp\in \Spec \mcO_E$ be a good prime. Then, $$P_\wp(\mcE,X) = \det(X- \tau^{\deg \wp}|\MM(\mcE_{\wp})) = P_\wp(\mcE_\wp,X),$$ where $\mcE_\wp$ denotes the reduction of $\mcE$ modulo $\wp$ and $P_\wp(\mcE_\wp,X)$ denotes the characteristic polynomial of the Frobenius element acting on the Tate module $T_\mfl(\mcE_\wp)$ for some $\mfl$ not lying below $\wp$.
\end{cor}
\begin{proof}
    The first equality follows from Proposition \ref{reduction_computation_lemma} by taking $M = \MM(\mcE)$ and the second equality follows from \cite[Corollary 3.7.3]{huang2022convolutions}. 
\end{proof}
\subsection{Twisted $L$-series}
Let $\varphi: \bA\to E[\tau]$ be a Drinfeld $\bA$-module, which will be denoted by $\varphi/E$ for simplicity. Let $\rho: G_E\to \GL(V)$ be an Artin representation with coefficients in $\conj{\FF}_q$. In this section, we will define the twisted $L$-series $L(\varphi^\vee/E,\rho,s)$.

Fix a prime $\wp$ of $\mcO_E$ of degree $d = [\FF_\wp:\FF_q]$. Let $D_{\wp}$, $I_{\wp}$ be the decomposition group and inertia group respectively. Let $\FF_\wp$ be the corresponding residue field at $\wp$. Write $G_\wp = D_{\wp}/I_{\wp}\cong \Gal(\conj{\FF}_{\wp}/\FF_\wp)$. The quotient group $G_\wp$ is therefore generated by the Frobenius automorphism $\Frob_{\wp}$, i.e. the $q^d$-th power map $x\mapsto x^{q^d}$. $\Frob_{\wp}$ is an endomorphism of the module $V^{I_\wp}$ of invariants and the characteristic polynomial $$P_\wp(\rho,X):= \det(X - \rho(\Frob_{\wp})|V^{I_\wp})\in \conj{\FF}_q[X]$$ depends only on the prime ideal $\wp$ and not on the choice of the prime ideals above $\wp$.

Let $\mfl$ be a prime in $A$ such that $\wp$ does not lie above $\mfl$. Recall that we have the $\mfl$-adic Galois representation $\rho_{\varphi,\mfl}$ and its dual representation $\rho_{\varphi,\mfl}^\vee$:
$$\rho_{\varphi,\mfl}: G_E\to \GL(V_\mfl(\varphi))\quad\quad \textnormal{and}\quad\quad \rho^\vee_{\varphi,\mfl}: G_E\to \GL(V^\vee_\mfl(\varphi)),$$ where $T_\mfl(\varphi)$ is the $\mfl$-adic Tate module of $\varphi$ and $V_\mfl(\varphi) = T_\mfl(\conj{\varphi})\otimes_{A_\mfl}K_\mfl$. 

Let $\conj{K}_\mfl$ denote a fixed algebraic closure of the local field $K_\mfl$ and fix an embedding $$\iota: \conj{\FF}_q\hookrightarrow \conj{K}_\mfl.$$ We can extend the scalar to $\conj{K}_\mfl$ by defining $$\conj{V}_\mfl(\varphi): = V_\mfl(\varphi) \otimes_{K_\mfl}\conj{K}_\mfl\quad\textnormal{  and  } \quad \conj{V}_\mfl(\rho): = V\otimes_{\conj{\FF}_q,\iota}\conj{K}_\mfl.$$ This gives us representations over $\conj{K}_\mfl$. Now, define the tensor product representation $$W_\mfl:= \conj{V}_\mfl(\varphi)\otimes_{\conj{K}_\mfl} \conj{V}_\mfl(\rho)\quad \textnormal{ and }\quad W^\vee_\mfl:= \conj{V}^\vee_\mfl(\varphi)\otimes_{\conj{K}_\mfl} \conj{V}_\mfl(\rho)$$ and hence we have $\mfl$-adic Galois representations $$\Theta_{\varphi,\rho,\mfl}: G_E\to \GL(W_\mfl)\quad \textnormal{ and }\quad \Theta^\vee_{\varphi,\rho,\mfl}: G_E\to \GL(W^\vee_\mfl).$$ Note that here $W^\vee_\mfl$ is not the dual representation of $W_\mfl$.
\begin{Def}
    We define the following characteristic polynomials at $\wp$:
    \begin{enumerate}
        \item[(1)] $\bP_{\wp}(\varphi,\rho,X) = \det(X- \Theta_{\varphi,\rho,\mfl}(\Frob_\wp)|W_\mfl^{I_\wp})$;
        \item[(2)] $\bP^\vee_{\wp}(\varphi,\rho,X) = \det(X- \Theta_{\varphi,\rho,\mfl}^\vee(\Frob_\wp)|(W_\mfl^\vee)^{I_\wp})$.   
    \end{enumerate}
\end{Def}

One can compare our construction with \cite[\S 1]{dokchitser2005l}.

Let $\bQ_{\wp}:=\bQ_{\wp}(\varphi,\rho,X)$ and $\bQ_{\wp}^\vee:=\bQ^\vee_{\wp}(\varphi,\rho,X)$ be the reciprocal polynomial of $\bP_{\wp}(\varphi,\rho,X)$ and $\bP^\vee_{\wp}(\varphi,\rho,X)$ respectively.
\begin{Def}
    Let $\varphi:A\to E[\tau]$ be a Drinfeld $A$-module and $\rho: G_E\to \GL(V)$ be an Artin representation with coefficients in $\conj{\FF}_q$. Let $S$ be a finite set of finite places containing the bad primes of $E$. We define $$L_S(\varphi,\rho,s) = \prod_{\wp\notin S} \bQ_{\wp}\left(n\wp^{-s}\right)^{-1}$$ and $$L_S(\varphi^\vee,\rho,s) = \prod_{\wp\notin S} \bQ^\vee_{\wp}\left(n\wp^{-s}\right)^{-1}$$ for each integer $s$.
\end{Def}

	\begin{prop} \label{prop:good_primes}
		Suppose that the prime $\wp$ satisfies the following conditions:
		\begin{enumerate}
			\item The Drinfeld module $\varphi$ has good reduction at $\wp$.
			\item The Artin representation $\rho$ is unramified at $\wp$, i.e., $\rho(I_{\wp}) = \{1\}$.
		\end{enumerate}
		Then the tensor product representation $W_{\mfl}$ is unramified at $\wp$. Consequently, $$\bP_\wp(\varphi,\rho,X) = P_\wp(\varphi,X)\otimes P_\wp(\rho,X)\quad \textnormal{ and }\quad \bP^\vee_\wp(\varphi,\rho,X) = P^\vee_\wp(\varphi,X)\otimes P_\wp(\rho,X)$$
		
		In other words, if $\{\alpha_1, \dots, \alpha_r\}$ are the eigenvalues of $\rho_{\varphi, \mfl}(\Frob_{\wp})$ acting on $V_\mfl(\varphi)$, and $\{\beta_1, \dots, \beta_n\}$ are the eigenvalues of $\rho(\Frob_{\wp})$ acting on $V$, then:
		\begin{equation} \label{eq:product_formula}
			\bP_{\wp}(\varphi, \rho, X) = \prod_{i=1}^r \prod_{j=1}^n (X - \alpha_i \beta_j ).
		\end{equation}
	\end{prop}
	
	\begin{proof}
		Since $\wp \nmid \mfl\mcO_E$ and $\varphi$ has good reduction at $\wp$, the N\'eron-Ogg-Shafarevich criterion for Drinfeld modules (see \cite[Corollary 4.10.10]{Goss}) implies that the action of $I_{\wp}$ on $V_\mfl(\varphi)$ is trivial. Similarly, since $\rho$ is unramified at $\wp$, $I_{\wp}$ acts trivially on $V$. Thus, the diagonal action of $I_{\wp}$ on $W_{\mfl}$ is therefore trivial, which implies $W_{\mfl}^{I_{\wp}} = W_{\mfl}$. 
        
        Recall that the eigenvalues of the tensor product of two linear operators are precisely the pairwise products of their individual eigenvalues. Thus, the eigenvalues of $\Theta_{\varphi,\rho,\mfl}(\Frob_\wp)$ acting on $W_{\mfl}$ are $\{\alpha_i \beta_j\}_{i,j}$. The result now follows from the definition.
	\end{proof}

\begin{remark}
    Let $S_{\mathrm{bad}}$ be the set of primes of $E$ consisting of:
    \begin{enumerate}
        \item The primes where $\varphi$ has bad reduction;
        \item The primes at which $\rho$ is ramified.
    \end{enumerate}
    The set $S_{\mathrm{bad}}$ is finite. For all primes $\wp \notin S_{\mathrm{bad}}$ (i.e., for all but finitely many primes), the explicit formula \eqref{eq:product_formula} holds. 
    Consequently, the Goss $L$-series, defined as the Euler product over all primes
    \[
        L(\varphi, \rho, s) := \prod_{\wp} \bQ_{\wp}(\varphi, \rho, n\wp^{-s})^{-1},
    \]
    is well-defined as a formal series or convergent product in a suitable half-plane, and its factors are explicitly computable for almost all $\wp$.
\end{remark}
\section{Artin twists of Drinfeld modules}\label{sec:twist}
\subsection{Construction of $\MM(\varphi,\rho)$}\label{construction}
We let $\varphi:\bA\to E[\tau]$ be a Drinfeld $\bA$-module over $E$ and $\rho: G_E\to \GL_n(\conj{\FF}_q)$ be an Artin representation.

Consider the field $\FF_q(\rho): = \FF_q(\rho(g)_{ij}: \forall g\in G_E, 1\leq i,j\leq n)$. Since $\rho$ factors through a finite quotient, $\FF_q(\rho)/\FF_q$ must be a finite extension. We may assume that $\FF_q(\rho) = \FF_{q^d}$ for some integer $d\geq 1$, i.e. $d = d_\rho: = [\FF_q(\rho): \FF_q]$.

Let $V_\rho$ be the $\FF_q(\rho)$-representation of $G_E$, i.e. $V_\rho : = \bigoplus_{i=1}^n \FF_q(\rho) v_i$ for an $\FF_q(\rho)$-basis and $$g\left(v_j\right) := \sum_{i=1}^n \rho(g)_{ij}v_i\textnormal{ for all }1\leq j\leq n.$$

We first review the theory of \'etale $\tau$-modules over a field $E\supset \FF_q$. Fix $E^\sep$ and $G_E = \Gal(E^\sep/E)$. We follow \cite{FontaineOuyang2022} as an exposition.
\begin{Def}
	A $\tau$-module over $E$ is a pair $(V,\tau)$ where $V$ is a finite-dimensional $E$-vector space and $\tau:V\to V$ is a semi-linear map with respect to the Frobenius endomorphism $\tau: x\mapsto x^q$, i.e. 
	\begin{equation}
		\tau(v+w) = \tau(v)+\tau(w), \forall v,w\in V;
	\end{equation}
\begin{equation}
	\tau(\lambda v) = \lambda^q\tau(v), \forall \lambda\in E, v\in V.
\end{equation}
\end{Def}

For an $E$-vector space $V$, we denote by $\tau^*V = E^{\tau}\otimes_{E}V$, where $E^{\tau} = E$ denotes the $E$-module structure via the Frobenius endomorphism $\tau: E\to E$. In other words, we have \begin{equation}
	\lambda(\mu\otimes v) = \lambda\mu\otimes v, \lambda\otimes \mu v = \mu^q\lambda\otimes v
\end{equation} for all $\lambda,\mu\in E$ and $v\in V$.
Clearly, $\tau^*V$ is also an $E$-vector space. 

We observe that if $\{v_1,\cdots,v_n\}$ is an $E$-basis for $V$, then $\{1\otimes v_1,\cdots,1\otimes v_n\}$ is an $E$-basis for $\tau^*V$. Furthermore, we have the following lemma:
\begin{lem}[cf. {\cite[Lemma 3.16]{FontaineOuyang2022}}]
	If $V$ is an $E$-vector space, giving a semi-linear map $\tau: V\to V$ is the same as giving a linear map \begin{equation}
		\begin{aligned}
			\tau_V: \tau^*V&\to V,\\
			\lambda\otimes v&\mapsto \lambda\tau(v)
		\end{aligned}
	\end{equation}
\end{lem}
\begin{Def}
	We say a $\tau$-module $V$ over $E$ is \emph{\'etale} if $\tau_V:\tau^*V\to V$ is an isomorphism.
\end{Def}
Let $V$ be a $\tau$-module of dimension $n$ and $\{v_1,\cdots,v_n\}$ an $E$-basis of $V$. Then, we can write
\begin{equation}
	\tau v_j = \sum_{i=1}^{n}a_{ij}v_i,
\end{equation} for some unique $a_{ij}\in E$. 
So, $$\tau_V(1\otimes v_j) = \sum_{i=1}^{n} a_{ij}v_i.$$ Since $\tau_V: \tau^*V \to V$ is an $E$-linear map between $E$-vector spaces with the same finite dimension, then we have the following result
\begin{prop}[cf. {\cite[Proposition 3.17]{FontaineOuyang2022}}]
	Let $V$ be a $\tau$-module over $E$ of dimension $n$ and keep the notation as above. The following are equivalent:
	\begin{enumerate}
		\item[(1)] the $\tau$-module $(V,\tau)$ over $E$ is \'etale;
		\item[(2)] $\tau_V$ is injective;
		\item[(3)] $\tau_V$ is surjective;
		\item[(4)] the matrix $(a_{ij})\in \GL_n(E)$;
		\item[(5)] $V = \Span_E (\tau(V))$.
	\end{enumerate}
\end{prop}

\begin{Def}
	Let $\mcM^{\et}_\tau(E)$ be the category of \'etale $\tau$-modules over $E$ consisting of 
	\begin{enumerate}
		\item[$\bullet$] Objects: \'etale $\tau$-modules $(V,\tau)$ over $E$.
		\item[$\bullet$] Morphisms: $E$-linear maps $f: V\to W$ such that $$\xymatrix{
			\tau^*V \ar[r]^{\tau^*f} \ar[d]_{\tau_V}& \tau^*W\ar[d]^{\tau_W}\\
			V \ar[r]^{f} & W
		}$$ commutes.
	\end{enumerate}
\end{Def}
\begin{Def}
	Let $G = G_E = \Gal(E^\sep/E)$. An $\conj{\FF}_q$-representation of $G$ is a finite dimensional $\conj{\FF}_q$-vector space $W$ together with a linear and continuous action of $G$.

Denote by $\Rep_{\conj{\FF}_q}(G)$ the category of all Artin $\conj{\FF}_q$-representations of $G$.
\end{Def}

Let $\rho$ be an Artin $\conj{\FF}_q$-representation of $G$ of dimension $n$. Define the $G$-action on $E^\sep\otimes_{\FF_q}V_\rho$ diagonally, i.e. $$g\cdot(\lambda\otimes v) = g(\lambda)\otimes g(v).$$ Let $$\bD(V_\rho):= (E^\sep\otimes_{\FF_q}V_\rho)^G.$$ 

First we observe that the Frobenius $\tau:x\mapsto  x^q$ on $E^\sep$ commutes with the $G$-action $\tau(g(x)) = g(\tau(x)),\forall g\in G, x\in E^\sep.$ We extend the Frobenius on $E^\sep\otimes_{\FF_q}V_\rho$ via
$\tau(\lambda\otimes v) = \tau(\lambda)\otimes v= \lambda^q\otimes v .$ Then for all $x\in E^\sep\otimes_{\FF_q}V_\rho$, we have $\tau(g(x)) = g(\tau(x)), \forall g\in G.$ One immediately sees that $x\in \bD(V_\rho)$ implies $\tau(x)\in \bD(V_\rho)$, hence by restricting to $\bD(V_\rho)$ we get $$\tau:\bD(V_\rho)\to \bD(V_\rho).$$
\begin{lem}[cf. {\cite[Proposition 3.20]{FontaineOuyang2022}}]\label{G-inv_bc_to_Esep}
	If $\rho$ is an Artin $\conj{\FF}_q$-representation of $G$ of dimension $d$, then
	the map
	$$\alpha_\rho: E^\sep \otimes_E\bD(V_\rho)\to E^\sep\otimes _{\FF_q}V_\rho,$$ $$\lambda\otimes x\mapsto \lambda\cdot x$$
	is a $G_E$-equivariant isomorphism, $\bD(V_\rho)$ is an \'etale $\tau$-module over $E$ and $\dim_E\bD(V_\rho) = nd_\rho$.

    Furthermore, if we equip the $\tau$-action on left hand side diagonally, this isomorphism is also $\tau$-equivariant.
\end{lem}
\begin{proof}
    We can regard $V_\rho$ as a finite dimensional $\FF_q$-representation of $G$. Then, the statement that $\alpha_\rho$ is an isomorphism follows from \cite[Proposition 3.20]{FontaineOuyang2022}. Hence, $$\dim_E \bD(V_\rho) = \dim_{\FF_q} V_\rho = [\FF_q(\rho):\FF_q]\cdot \dim_{\FF_q(\rho)}V_\rho = nd_\rho.$$

     We now verify the this isomorphism is $G_E$-equivariant. Indeed, since $g(x) = x$,
    \[
        \alpha_\rho(g(\lambda\otimes x)) = \alpha_\rho(g(\lambda)\otimes x)
        = g(\lambda)\cdot x
        = g(\lambda\cdot x)
        = g(\alpha_\rho(\lambda\otimes x)).
    \]

    It remains to verify that this isomorphism is $\tau$-equivariant. Indeed, \[
        \alpha_\rho(\tau (\lambda\otimes x))  = \alpha_\rho(\lambda^q\otimes \tau(x))
         = \lambda^q\cdot \tau(x)
         = \tau(\lambda\cdot x)
         = \tau(\alpha_\rho(\lambda\otimes x)).
    \]
\end{proof}
Thus, we get an additive functor $$\bD: \Rep_{\conj{\FF}_q}(G_L)\to \mcM^\et_\tau(L).$$

Also, since the $\FF_q(\rho)$-vector space $V_\rho$ can be viewed as an $\FF_q$-vector space, denoted by $\Res_{\FF_q(\rho)/\FF_q}V_\rho$, we see that $\bD(\Res_{\FF_q(\rho)/\FF_q} V_\rho)$ is an \'etale $E$-module. We can therefore form an $\bA$-motive over $E$ by $$\MM(\rho) = (\bD(\Res_{\FF_q(\rho)/\FF_q} V_\rho)\otimes_{E}\bA_{E},\tau\otimes \tau_{\bA_{E}}),$$ where $\tau_{\bA_{E}}: \bA_{E}\to \bA_{E}$ is the Frobenius twist $a\mapsto a^{(1)}$.

\begin{remark}\label{rem:taelman_inverse}
    It is instructive to compare our construction with the theory of Artin $t$-motives developed by Taelman in \cite{taelman2009artin}. Taelman's construction is based on an equivalence of categories between the category of \'etale $\tau$-modules over $E$ and the category of finite dimensional $\FF_q$-linear representations of the absolute Galois group, see \cite[Theorem 4.1.1]{taelman2009artin}. The functor $\bD$ is the quasi-inverse of the fiber functor in \cite[Theorem 4.1.1]{taelman2009artin} in the case $\bA = \FF_q[t]$ and $d = [\FF_q(\rho):\FF_q] = 1$. Our construction of $\MM(\rho)$ is motivated by Taelman's construction of $M(V)$ in \cite[\S 4]{taelman2009artin}.
\end{remark}
\begin{Def}
    Let $\varphi$ be a Drinfeld $\bA$-module over $E$ and $\rho$ an Artin $\conj{\FF}_q$-representation of $G_E$. We define the \textit{Artin twisted $\bA$-motive} of $\varphi$ by $\rho$ to be $$\MM(\varphi,\rho) = \MM(\varphi)\otimes \MM(\rho).$$
\end{Def}

Our goal in this section is to study this $\bA$-motive and its associated motivic $L$-series.

Recall $G = G_E = \Gal(E^\sep/E)$ where $E$ is a field containing $\FF_q$.  Let $W$ be an $\FF_{q^d}[G]$-module via $\rho$, i.e. $\rho: G\to \GL(W)$. For each $0\leq j\leq d-1$, we define $$W^{(j)} = E^\sep\otimes_{\FF_{q^d},\tau^j}W,$$ i.e. for all $\lambda\in \FF_{q^d}$, we have $$ a\otimes \lambda w = \tau^j(\lambda)  a\otimes w=\lambda^{q^j} a\otimes w$$
\begin{lem}\label{split_base_change}
    We have an isomorphism of $E^\sep$-vector spaces $$\Psi: E^\sep\otimes_{\FF_q}W\to \bigoplus_{j=0}^{d-1} W^{(j)},$$
$$x\otimes w\mapsto (x\otimes w, x\otimes w, \cdots, x\otimes w).$$
\end{lem}
\begin{proof}
    First, we check $\Psi$ is well-defined. Indeed, for any $c\in \FF_q$, $$\Psi(cx\otimes w) = (cx\otimes w,\cdots, cx\otimes w) $$ and $$\Psi(x\otimes cw) = (x\otimes cw,\cdots, x\otimes cw)= (cx\otimes w, c^qx\otimes w,\cdots, c^{q^{d-1}}x\otimes w).$$ Since $c^{q^j} = c$, we see that $\Psi(cx\otimes w) = \Psi(x\otimes cw)$ for all $c\in\FF_q$.

    Injectivity: We take an $\FF_{q^d}$-basis $\{w_1,\cdots,w_n\}$ for $W$ and an $\FF_q$-basis $\{\alpha_1,\cdots,\alpha_d\}$ for $\FF_{q^d}$. Then, $\{\alpha_iw_j: 1\leq i\leq d,1\leq j\leq n\}$ is an $\FF_q$-basis for $W$. Thus, any element of $E^\sep\otimes_{\FF_q}W$ has the form of $$y = \sum_{i=1}^d\sum_{j=1}^n x_{ij}\otimes \alpha_iw_j$$ with $x_{ij}\in E^\sep.$

    If $\Psi(y) = 0$, then $$\Psi(y) = \sum_{i=1}^d\sum_{j=1}^n(\alpha_ix_{ij}\otimes w_j,\cdots ,\alpha_{i}^{q^{d-1}}x_{ij}\otimes w_j) = 0.$$ Thus, for each $0\leq k\leq d-1$, we must have $$\sum_{i=1}^d\sum_{j=1}^n \alpha_i^{q^k}x_{ij}\otimes w_j = 0 \textnormal{ in }W^{(k)}.$$ Since $\{1\otimes w_1,\cdots , 1\otimes w_n\}$ is an $\FF_{q^d}$-basis for $W^{(k)}$, we see that $$\sum_{i=1}^d \alpha_i^{q^k} x_{ij} = 0$$ for each $1\leq j\leq n$ and $0\leq k\leq d-1$.

    Thus, if we fix $j$, then we have $$\begin{pmatrix}
        \alpha_1 & \alpha_2 & \cdots & \alpha_d\\
        \alpha_1^q & \alpha_2^q & \cdots &\alpha_d^q\\
        \vdots & \vdots & \ddots & \vdots \\
        \alpha_1^{q^{d-1}} & \alpha_2^{q^{d-1}}& \cdots & \alpha_d^{q^{d-1}}
    \end{pmatrix}\begin{pmatrix}
        x_{1j}\\
        x_{2j}\\
        \vdots \\
        x_{dj}
    \end{pmatrix} = \begin{pmatrix}
        0\\
        0\\
        \vdots\\
        0
    \end{pmatrix}.$$

    By the result of the Moore determinant, we have $$\det \begin{pmatrix}
        \alpha_1 & \alpha_2 & \cdots & \alpha_d\\
        \alpha_1^q & \alpha_2^q & \cdots &\alpha_d^q\\
        \vdots & \vdots & \ddots & \vdots \\
        \alpha_1^{q^{d-1}} & \alpha_2^{q^{d-1}}& \cdots & \alpha_d^{q^{d-1}}
    \end{pmatrix} = \prod_{(c_1,\cdots, c_d)\in \FF_q^d\backslash (0,\cdots,0)}(c_1\alpha_1+ \cdots c_d\alpha_d).$$ 

    Thus, $\det M\neq 0$ if and only if $\alpha_1,\cdots,\alpha_d$ are linearly independent over $\FF_q$. The latter is true as it is an $\FF_q$-basis of $\FF_{q^d}$.

    Thus, $x_{ij} = 0$ for all $i,j$, i.e. $y=0$. Thus, we see that $\Psi$ is injective.

    Now, by counting the dimension over $E^\sep$ of both sides, we see $\Psi$ is an $E^\sep$-linear isomorphism.
\end{proof}

Now, notice that if $x\in \FF_{q^d}\subseteq E^\sep$, then $x^{q^d}-x = 0$ and for any $g\in G = \Gal(E^\sep/E)$ we must have $g(x)^{q^d} - g(x) = 0$, hence $g(x)\in \FF_{q^d}$. Hence, $$g|_{\FF_{q^d}}\in \Gal(\FF_{q^d}/\FF_q) \cong \langle x\mapsto x^q\rangle$$
\begin{Def}
    For $g\in G$, we define $m_g\in \ZZ/d\ZZ$ to be the unique integer such that $$g(x) = x^{q^{m_g}}$$ for all $x\in \FF_{q^d}$.
\end{Def}
Now, we are ready to define a $G$-action on $\dis \bigoplus_{j=0}^{d-1} W^{(j)}$.

Let $\{w_1,\cdots,w_n\}$ be an $\FF_{q^d}$-basis for $W$ as above. Then each $y_j\in W^{(j)} = E^\sep\otimes_{\FF_{q^d},\tau^j}W$ has the form $$y_j = \sum_{i=1}^n x_{ij}\otimes w_i$$ for $x_{ij}\in E^\sep$.

Write $\dis g(y_j) = \sum_{i=1}^n g(x_{ij})\otimes \rho(g)w$.

Let $\dis Y = (y_0,y_1,\cdots,y_{d-1})\in \bigoplus_{j=0}^{d-1} W^{(j)}$, then for each $0\leq j\leq d-1$, we define $$(g\cdot Y)_j = \sum_{i=1}^n g(x_{i,[j-m_g]})\otimes\rho(g)w_i,$$ where $0\leq [j-m_g]\leq d-1$ is the unique integer such that $$[j-m_g]\equiv j-m_g \mod d.$$

For example, if $m_g= 1$, then 
$$g: (y_0,y_1,\cdots, y_{d-1}) \mapsto (g(y_{d-1}), g(y_0), g(y_1),\cdots, g(y_{d-2}))$$

Then, we have the following Lemma.
\begin{lem}
    The isomorphism $$\Psi: E^\sep\otimes_{\FF_q}W\to \bigoplus_{j=0}^{d-1} W^{(j)},$$
$$x\otimes w\mapsto (x\otimes w, x\otimes w, \cdots, x\otimes w)$$ is $G$-equivariant.
\end{lem}
\begin{proof}
    Indeed, for any $g\in G$, let $m = m_g$, then $$\Psi(g(x\otimes cw)) = \Psi(g(x)\otimes c\rho(g)w) = (cg(x)\otimes \rho(g)w,\cdots, c^{q^{d-1}}g(x)\otimes \rho(g)w)$$ and $$\begin{aligned}
        g\Psi(x\otimes cw) &= g(cx\otimes w,\cdots, c^{q^{d-1}}x\otimes w) \\
        &= (cg(x)\otimes \rho(g)w,\cdots, c^{q^{d-1}}g(x)\otimes \rho(g)w) \\
        &= \Psi(g(x\otimes cw)).
    \end{aligned}$$
\end{proof}

For $$y_j = \sum_{i=1}^n x_{ij}\otimes w_i$$ with $x_{ij}\in E^\sep$. We write $$\Frob(y_j) = \sum_{i=1}^n x_{ij}^q\otimes w_i.$$

Now, we can define the $\tau$-action on $\dis \bigoplus_{j=0}^{d-1}W^{(j)}$ via $$\tau \cdot \begin{pmatrix}
    y_0\\
    \vdots\\
    y_{d-2}\\
    y_{d-1}
\end{pmatrix} = \begin{pmatrix}
    0 &  \cdots & 0& \Frob\\
    \Frob  &\cdots &0& 0\\
    \vdots  & \ddots & 0& \vdots\\
    0 &\cdots & \Frob & 0
\end{pmatrix}\begin{pmatrix}
    y_0\\
    \vdots\\
    y_{d-2}\\
    y_{d-1}
\end{pmatrix}.$$
Then, we have the following result.
\begin{prop}
     The isomorphism $$\Psi: E^\sep\otimes_{\FF_q}W\to \bigoplus_{j=0}^{d-1} W^{(j)},$$
$$x\otimes w\mapsto (x\otimes w, x\otimes w, \cdots, x\otimes w)$$ is $\tau$-equivariant.
\end{prop}
\begin{proof}
    Indeed, for any $c\in\FF_{q^d}$, we have $$\begin{aligned}
        \Psi(\tau(x\otimes cw))&= \Psi(x^q\otimes cw) = (cx^q\otimes w, c^qx^q\otimes w,\cdots, c^{q^{d-1}}x^q\otimes w)
    \end{aligned}$$ and $$\begin{aligned}
        \tau(\Psi(x\otimes cw)) &= \tau(cx\otimes w, c^qx\otimes w,\cdots, c^{q^{d-1}}x\otimes w)\\
        & = (cx^q\otimes w, c^qx^q\otimes w,\cdots, c^{q^{d-1}}x^q\otimes w).
    \end{aligned}$$

    Thus, $\Psi\circ \tau  = \tau\circ \Psi$ as desired.
\end{proof}
\begin{cor}
    $\dis \widehat{W}:=\left(\bigoplus_{j=0}^{d-1}W^{(j)}\right)^G$ is an \'etale $\tau$-module over $E$ of dimension $nd$.
\end{cor}

Consider $Y = (y_0,y_1,\cdots,y_{d-1})\in \widehat{W}$. Each $y_\ell\in W^{(\ell)}$ and $\{w_1,\cdots,w_n\}$ is an $\FF_{q^d}$-basis for $W$. By the bijectivity of $\Psi$, we can write 
$$
\begin{aligned}
    Y &= \Psi\left(\sum_{j=1}^{d}\sum_{i=1}^n x_{ij}\otimes \alpha_j w_i\right) \\
    & = \left(\sum_{j=1}^{d}\sum_{i=1}^n x_{ij}\otimes \alpha_j w_i,\sum_{j=1}^{d}\sum_{i=1}^n x_{ij}\otimes \alpha_j w_i,\cdots,\sum_{j=1}^{d}\sum_{i=1}^n x_{ij}\otimes \alpha_j w_i\right)\\
    & = \left(\sum_{k=1}^n \left(\sum_{j=1}^{d} x_{kj}\alpha_j\right)\otimes w_k,\; \cdots,\; \sum_{k=1}^n \left(\sum_{j=1}^{d} x_{kj}\alpha_j^{q^{d-1}}\right)\otimes w_k\right)\\
\end{aligned}
$$ 
with $x_{ij}\in E^{\text{sep}}$ unique for $1\le i \le n$ and $1\le j \le d$.

Also notice that 
$$
\begin{aligned}
    g(Y) & = \Psi\left(g\left(\sum_{j=1}^{d}\sum_{i=1}^n x_{ij}\otimes \alpha_j w_i\right)\right)\\
    & = \Psi\left(\sum_{j=1}^{d}\sum_{i=1}^n g(x_{ij})\otimes \alpha_j g(w_i)\right)\\
    & = \Psi\left(\sum_{j=1}^{d}\sum_{i=1}^n g(x_{ij})\otimes \alpha_j\left(\sum_{k=1}^n \rho(g)_{ki}w_k\right)\right)\\
     & =\sum_{j=1}^{d}\sum_{i=1}^n\sum_{k=1}^n\left( \alpha_j\rho(g)_{ki} g(x_{ij})\otimes  w_k, \dots, \alpha_j^{q^{d-1}}\rho(g)_{ki}^{q^{d-1}}g(x_{ij})\otimes w_k\right)\\
\end{aligned}
$$
Here we use the standard convention for matrix action on basis vectors: $g(w_i) = \sum_{k=1}^n \rho(g)_{ki} w_k$.

Thus, $g(Y) = Y$ if and only if for each $0\leq \ell\leq d-1$, we have 
$$
\sum_{k=1}^n \left( \sum_{i=1}^n \rho(g)_{ki}^{q^{\ell}} \sum_{j=1}^{d} g(x_{ij}) \alpha_j^{q^\ell} \right) \otimes w_k = \sum_{k=1}^n \left( \sum_{j=1}^{d} x_{kj} \alpha_j^{q^\ell} \right) \otimes w_k
$$ 
Comparing the coefficients of the basis vectors $w_k$, this holds if and only if for each $0\leq \ell\leq d-1$ and $1\leq k\leq n$:

\begin{equation}\label{G_invariant_y_transposed}
    \sum_{i=1}^n \rho(g)_{ki}^{q^{\ell}} \left( \sum_{j=1}^{d} g(x_{ij}) \alpha_j^{q^\ell} \right) = \sum_{j=1}^{d} x_{kj} \alpha_j^{q^\ell}.
\end{equation}

Now define:
\begin{itemize}
    \item $\vec{\boldsymbol{\alpha}} = (\alpha_1, \alpha_2, \cdots, \alpha_d)^\top \in \Mat_{d\times 1}(\FF_{q^d})$,
    \item $\vec{\mathbf{x}} = (x_{ij}) = \begin{pmatrix}
    x_{11} & x_{12} & \cdots & x_{1d}\\
    x_{21} & x_{22} & \cdots & x_{2d}\\
    \vdots & \vdots & \ddots & \vdots\\
    x_{n1} & x_{n2} & \cdots & x_{nd}
    \end{pmatrix}\in \Mat_{n\times d}(E^{\text{sep}})$,
    \item $\rho^{(\ell)}(g) = (\rho(g)_{ij}^{q^\ell}) \in \GL_n(\FF_{q^d})$.
\end{itemize}

Then the system of equations (\ref{G_invariant_y_transposed}) can be written in matrix form as:
$$
\rho^{(\ell)}(g) g(\vec{\mathbf{x}}) \vec{\boldsymbol{\alpha}}^{(\ell)} = \vec{\mathbf{x}} \vec{\boldsymbol{\alpha}}^{(\ell)}.
$$

Write $\rho_\ell = \rho^{(\ell)}$ the $\ell$-th Frobenius twist of $\rho$. Then, we have \begin{equation}
    g(\vec{\mathbf{x}})\vec{\boldsymbol{\alpha}}^{(\ell)} = \rho_\ell(g)^{-1}\vec{\mathbf{x}}\vec{\boldsymbol{\alpha}}^{(\ell)}
\end{equation} for all $0\leq \ell \leq d-1$.

Denote $\bm{w} = \{w_1,w_2,\cdots,w_n\}$ the given $\FF_{q^d}$-basis of $W$ aforementioned, and define \begin{equation}\label{sol_space}
\operatorname{Sol}_E(\rho,\vec{\boldsymbol{\alpha}},\bm{w}) = \left\{\vec{\mathbf{x}}\in \Mat_{n\times d}(E^\sep)\middle|\begin{aligned}
   & \ \ \ g(\vec{\mathbf{x}})\vec{\boldsymbol{\alpha}}^{(\ell)} = \rho_\ell(g)^{-1}\vec{\mathbf{x}}\vec{\boldsymbol{\alpha}}^{(\ell)}\\
 &   \forall g\in G\textnormal{ and }0\leq \ell\leq d-1
\end{aligned}  \right\}.
\end{equation}
We may simply write $\operatorname{Sol}_E(\rho)$ if $\vec{\boldsymbol{\alpha}}$ and $\bm{w}$ is clear from the context.

The following results are immediate from the construction.
\begin{prop}
    The semilinear map $\tau: \vec{\mathbf{x}}\mapsto \vec{\mathbf{x}}^{(1)}$ makes $\operatorname{Sol}_E(\rho,\vec{\boldsymbol{\alpha}},\bm{w})$ into an \'etale $\tau$-module over $E$ and we have an isomorphism of \'etale $\tau$-modules over $E$ $$\begin{aligned}
    \widehat{\Psi}: (E^\sep\otimes_{\FF_q}W)^G & \stackrel{\sim}{\to} \operatorname{Sol}_E(\rho,\vec{\boldsymbol{\alpha}},\bm{w}), \\
    \bm{x}&\mapsto \vec{\mathbf{x}} = (x_{ij}) ,
\end{aligned}$$ where $\dis \bm{x}:=\sum_{j=1}^{d}\sum_{i=1}^n x_{ij}\otimes \alpha_j w_i$.

Furthermore, $\dim_E \operatorname{Sol}_E(\rho,\vec{\boldsymbol{\alpha}},\bm{w}) = nd$.
\end{prop}
\begin{cor}
    We have an isomorphism of $\bA$-motives over $E$, $$\MM(\rho) \cong \operatorname{Sol}_E(\rho,\vec{\boldsymbol{\alpha}},\bm{w})\otimes_E \bA_E.$$
\end{cor}
\begin{lem}
    Let $\vec{\mathbf{x}}\in \operatorname{Sol}_E(\rho,\vec{\boldsymbol{\alpha}},\bm{w})$ and $N = nd$. Then $\vec{\mathbf{x}}, \vec{\mathbf{x}}^{(1)}, \cdots, \vec{\mathbf{x}}^{(N-1)}$ are linearly independent over $E$ if and only if all entries of $\vec{\mathbf{x}}$ are linearly independent over $\FF_q$.
\end{lem}
\begin{proof}
    We treat $\vec{\mathbf{x}}$ as a column vector of length $N$ in $(E^\sep)^N$. Then, $\vec{\mathbf{x}}, \vec{\mathbf{x}}^{(1)}, \cdots, \vec{\mathbf{x}}^{(N-1)}$ are linearly independent over $E$ if and only if the Moore matrix $M(\vec{\mathbf{x}}) := [\vec{\mathbf{x}}, \vec{\mathbf{x}}^{(1)}, \cdots, \vec{\mathbf{x}}^{(N-1)}]$ has non-zero determinant. By \cite[Lemma 1.3.3]{Goss}, the latter holds if and only if $\{x_{11}, \cdots, x_{nd}\}$ are linearly independent over $\FF_q$.
\end{proof}
\begin{Def}
    We call an element $\vec{\mathbf{u}}\in \operatorname{Sol}_E(\rho,\vec{\boldsymbol{\alpha}},\bm{w})$ a \textit{fundamental solution} if all its entries are linearly independent over $\FF_q$. We denote by $\operatorname{FS}_E(\rho,\vec{\boldsymbol{\alpha}},\bm{w})$ the subset of all fundamental solutions.
\end{Def}
\begin{lem}
    There exists a fundamental solution $\vec{\mathbf{u}}\in \operatorname{Sol}_E(\rho,\vec{\boldsymbol{\alpha}},\bm{w})$.
\end{lem}
\begin{proof}
    For a non-zero $\bm{c} = (c_{ij})\in \Mat_{n\times d}(\FF_q)$, we define $$H_{\bm{c}} = \left\{\vec{\mathbf{x}}\in \operatorname{Sol}_E(\rho,\vec{\boldsymbol{\alpha}},\bm{w}): \sum_{i,j} c_{ij}x_{ij} = 0\right\}.$$

    Clearly, this is an $E$-subspace of $\operatorname{Sol}_E(\rho,\vec{\boldsymbol{\alpha}},\bm{w})$. We first claim that $H_{\bm{c}}\neq \operatorname{Sol}_E(\rho,\vec{\boldsymbol{\alpha}},\bm{w})$.  Assume the contrary. By Lemma \ref{G-inv_bc_to_Esep}, we have  \[
        E^\sep\otimes_E H_{\bm{c}} 
        \cong E^\sep\otimes_E (E^\sep\otimes_{\FF_q}W)^G
        \cong E^\sep\otimes_{\FF_q}W 
        \cong \Mat_{n\times d}(E^\sep)
    \] and thus there exists an $E^\sep$-basis for $\Mat_{n\times d}(E^\sep)$ whose elements lie in $H_{\bm{c}}$. Call them $\vec{\mathbf{v}}_1, \cdots, \vec{\mathbf{v}}_N$. Write $L_{\bm{c}}(\vec{\mathbf{x}}) = \sum_{i,j} c_{ij}x_{ij}$. Then, we have $L_{\bm{c}}(\vec{\mathbf{v}}_i) = 0$ for all $1\leq i\leq N$. Since this is an $E^\sep$-basis for $\Mat_{n\times d}(E^\sep)$, we see that $$L_{\bm{c}}(\vec{\mathbf{x}}) = 0\quad \textnormal{ for all }\vec{\mathbf{x}}\in \Mat_{n\times d}(E^\sep).$$ This is absurd as the entire space $\Mat_{n\times d}(E^\sep)$ cannot be covered by a hyperplane. Hence, $H_{\bm{c}}\subsetneq \operatorname{Sol}_E(\rho,\vec{\boldsymbol{\alpha}},\bm{w})$ is a proper subspace for any non-zero $\bm{c}\in \Mat_{n\times d}(\FF_q)$.

    Notice that $$\operatorname{FS}_E(\rho,\vec{\boldsymbol{\alpha}},\bm{w}) = \operatorname{Sol}_E(\rho,\vec{\boldsymbol{\alpha}},\bm{w})\backslash \left(\bigcup_{\bm{c}\in \Mat_{n\times d}(\FF_q)\backslash\{0\}}H_{\bm{c}}\right).$$ Since a vector space over an infinite field is not a finite union of proper subspaces, we see that $\operatorname{FS}_E(\rho,\vec{\boldsymbol{\alpha}},\bm{w})\neq \varnothing$. Thus, there exists a fundamental solution $\vec{\mathbf{u}}\in \operatorname{Sol}_E(\rho,\vec{\boldsymbol{\alpha}},\bm{w})$.
\end{proof}
\begin{cor}
    Let $\vec{\mathbf{u}}\in \operatorname{Sol}_E(\rho,\vec{\boldsymbol{\alpha}},\bm{w})$ be a fundamental solution. Then, $\vec{\mathbf{u}}, \vec{\mathbf{u}}^{(1)}, \cdots, \vec{\mathbf{u}}^{(N-1)}$ is an $E$-basis for $\operatorname{Sol}_E(\rho,\vec{\boldsymbol{\alpha}},\bm{w})$.
\end{cor}
\begin{Def}
    We say an $E$-basis for $\operatorname{Sol}_E(\rho,\vec{\boldsymbol{\alpha}},\bm{w})$ is \textit{fundamental} if it is of the form $\vec{\mathbf{u}}, \vec{\mathbf{u}}^{(1)}, \cdots, \vec{\mathbf{u}}^{(N-1)}$ for a fundamental solution $\vec{\mathbf{u}}$.
\end{Def}
Now, we are ready to write down the matrix of the $\tau$-action with respect to a fundamental basis.

Since $\operatorname{Sol}_E(\rho,\vec{\boldsymbol{\alpha}},\bm{w})$ has dimension $N$ over $E$, there exists $f_0,f_1,\cdots,f_{N-1}\in E$ such that $$\vec{\mathbf{u}}^{(N)} = f_0\vec{\mathbf{u}} + f_1 \vec{\mathbf{u}}^{(1)} + \cdots + f_{N-1} \vec{\mathbf{u}}^{(N-1)} = [\vec{\mathbf{u}}, \vec{\mathbf{u}}^{(1)}, \cdots, \vec{\mathbf{u}}^{(N-1)}] \begin{pmatrix}
    f_0\\
    f_1\\
    \vdots\\
    f_{N-1}
\end{pmatrix}.$$

Write $M(\vec{\mathbf{u}}) = [\vec{\mathbf{u}}, \vec{\mathbf{u}}^{(1)}, \cdots, \vec{\mathbf{u}}^{(N-1)}]$, where we consider $\vec{\mathbf{u}}$ as a column vector of length $N$ in $(E^\sep)^N$. Then, \begin{equation}\label{f_i}
    \begin{pmatrix}
    f_0\\
    f_1\\
    \vdots\\
    f_{N-1}
\end{pmatrix} = M(\vec{\mathbf{u}})^{-1} \vec{\mathbf{u}}^{(N)}.
\end{equation}

So, under the basis $\vec{\mathbf{u}}, \vec{\mathbf{u}}^{(1)}, \cdots, \vec{\mathbf{u}}^{(N-1)}$, the $\tau$-action has a matrix \begin{equation}\label{Phimatrix}
    \Phi = \Phi({\vec{\mathbf{u}}}) = \begin{pmatrix}
    0 &  \cdots & 0 & f_0\\
    1 &  \cdots & 0 & f_1\\
    \vdots &  \ddots & \vdots & \vdots\\
    0 &  \cdots & 1 & f_{N-1}
\end{pmatrix}.
\end{equation}
\begin{prop} Keep the notation as above, we have the following.\label{Phi_in_terms_of_Mu}
\begin{enumerate}
    \item The matrix $\Phi = M(\vec{\mathbf{u}})^{-1}M(\vec{\mathbf{u}})^{(1)}$ and in particular, it is invertible. 
    \item $f_0 = (-1)^{N-1}(\det M(\vec{\mathbf{u}}))^{q-1}\neq 0.$
\end{enumerate}
\end{prop}
\begin{proof}
    Notice that $$M(\vec{\mathbf{u}})^{-1}M(\vec{\mathbf{u}})^{(1)} = M(\vec{\mathbf{u}})^{-1} [\vec{\mathbf{u}}^{(1)}, \cdots, \vec{\mathbf{u}}^{(N)}] =  [M(\vec{\mathbf{u}})^{-1}\vec{\mathbf{u}}^{(1)}, \cdots, M(\vec{\mathbf{u}})^{-1}\vec{\mathbf{u}}^{(N)}].$$ Furthermore, for each $1\leq i\leq N-1$, $M(\vec{\mathbf{u}})^{-1}\vec{\mathbf{u}}^{(i)}$ is the $(i+1)$-th column of the identity $I_N = M(\vec{\mathbf{u}})^{-1}M(\vec{\mathbf{u}})$. Together with (\ref{f_i}), we see that $$M(\vec{\mathbf{u}})^{-1}M(\vec{\mathbf{u}})^{(1)} = \Phi.$$

    Now, taking the determinant both sides, we get $$(\det M(\vec{\mathbf{u}}))^{-1} (\det M(\vec{\mathbf{u}}))^q = (-1)^{N-1}f_0.$$ The result now follows.
\end{proof}
Summarizing what we have discussed above, we get the following result.
\begin{prop}
    Let $\vec{\mathbf{u}}\in \operatorname{Sol}_E(\rho,\vec{\boldsymbol{\alpha}},\bm{w})$ be a fundamental solution and $N = nd$. Then, we have an isomorphism of $\bA$-motives $$\MM(\rho)\stackrel{\sim}{\to} \bigoplus_{i=0}^{N-1}\bA_E \vec{\mathbf{u}}^{(i)}$$ over $E$. The matrix of $\tau$-action with respect to  $\vec{\mathbf{u}}, \vec{\mathbf{u}}^{(1)}, \cdots, \vec{\mathbf{u}}^{(N-1)}$ is $$\Phi = \Phi({\vec{\mathbf{u}}}) = \begin{pmatrix}
    0 &  \cdots & 0 & f_0\\
    1 &  \cdots & 0 & f_1\\
    \vdots &  \ddots & \vdots & \vdots\\
    0 &  \cdots & 1 & f_{N-1}
\end{pmatrix},$$ where $(
    f_0,
    f_1,
    \cdots,
    f_{N-1}
)^\top = M(\vec{\mathbf{u}})^{-1} \vec{\mathbf{u}}^{(N)}\in \Mat_{N\times 1}(E)$ and $\vec{\mathbf{u}}$ is viewed as a column vector.
\end{prop}
\subsection{The Anderson $\bA$-module $\mcE(\varphi,\rho)$}
Our goal in this section is to prove that there exists an Anderson $\bA$-module $\mcE = \mcE(\varphi,\rho)$ over $E$ such that $$\MM(\varphi,\rho)\cong_E \MM(\mcE).$$
\subsubsection{The existence of $\mcE$}
\begin{lem}\label{reduce_to_t-module}
	Let $H$ be an $\bA$-field with structure map $\gamma:\bA\to H$. Let $M$ and $N$ be two $\bA$-motives over $H$ and $B\hookrightarrow \bA$ a subring of $\bA$ such that $M$ and $N$ are finite free over $B_H$. Let $\mcM = (M\otimes_{A_H} N,\tau_M\otimes \tau_N)$ be the tensor product in the category of $\bA$-motives and $\mcM_B = (M\otimes_{B_H} N,\tau_M\otimes \tau_N)$ be the tensor product in the category of $B$-motives. If $a\in B$ is an element such that $(a-\gamma(a))^k\mcM_B\subseteq \tau(\mcM_B)$ for some positive integer $k$, then $(a-\gamma(a))^k \mcM\subseteq \tau(\mcM)$.
\end{lem}
\begin{proof}
	Let $\{e_1,\cdots,e_m\}$ be a $B_H$-basis for $M$ and $\{f_1,\cdots,f_n\}$ be an $B_H$-basis for $N$. Then, they generate $M$ and $N$ as $A_H$-modules respectively. Thus, we can consider the surjective $B_H$-linear map $$\pi: M\otimes_{B_H}N\rightarrow M\otimes_{A_H}N,$$ which is $\tau$-equivariant. We use the notation $m\otimes_B n$ for elements in the left hand side.
	
	By the assumption, for any $i$ and $j$, we have $$(a-\gamma(a))^k(e_i\otimes_B f_j) = \tau\left(\sum_{i,j}b_{ij}e_i\otimes_B f_j\right) = \sum_{i,j}b_{ij}^{(1)}\tau_M(e_i)\otimes_B\tau_N(f_j)$$ for some $b_{ij}\in B_H$. Applying the map $\pi$, we see that $$(a-\gamma(a))^k(e_i\otimes f_j) = \sum_{ij}b_{ij}^{(1)}\tau\left(e_i\otimes f_j\right) = \tau\left(\sum_{ij}b_{ij}e_i\otimes f_j\right)\in \tau(\mcM).$$
\end{proof}
\begin{thm}\label{existsAnderson}
	Let $\varphi/E$ be a Drinfeld $\bA$-module of rank $r$ and $\rho: G\to \GL(W)$ an $\FF_{q^d}$-representation of dimension $n$. Then, there exists an abelian $\bA$-module $\mcE$ of dimension $N = nd$ and rank $rN$ over $E$ such that $$\MM(\varphi,\rho)\cong_E \MM(\mcE).$$
\end{thm}
\begin{proof}
    Let $\vec{\mathbf{u}}\in \operatorname{Sol}_E(\rho,\vec{\boldsymbol{\alpha}},\bm{w})$ be a fundamental solution and $N = nd$. Then, we have an isomorphism of $\bA$-motives $$\MM(\rho)\stackrel{\sim}{\to} \bigoplus_{i=0}^{N-1}\bA_E \vec{\mathbf{u}}^{(i)}$$ over $E$. The matrix of $\tau$-action with respect to  $\vec{\mathbf{u}}, \vec{\mathbf{u}}^{(1)}, \cdots, \vec{\mathbf{u}}^{(N-1)}$ is $$\Phi = \Phi({\vec{\mathbf{u}}}) = \begin{pmatrix}
    0 &  \cdots & 0 & f_0\\
    1 &  \cdots & 0 & f_1\\
    \vdots &  \ddots & \vdots & \vdots\\
    0 &  \cdots & 1 & f_{N-1}
\end{pmatrix},$$ where $(
    f_0,
    f_1,
    \cdots,
    f_{N-1}
)^\top = M(\vec{\mathbf{u}})^{-1} \vec{\mathbf{u}}^{(N)}\in \Mat_{N\times 1}(E)$ and $\vec{\mathbf{u}}$ is viewed as a column vector.
	
	Let $t\in \bA-\FF_q$ and let $\theta = \gamma(t)$ be the image of $t$ under the structure map $\gamma: \bA\to K$. Replace $r$ by $r\deg t$ and according to Lemma \ref{reduce_to_t-module}, we may assume that $\bA = \FF_q[t]$. Thus, $\MM(\varphi,\rho)$ is a free $E[t]$-module of rank $rN$. Let $e_1,\cdots,e_N$ be the standard basis of $E[t]^N$, i.e. $$e_1 = \begin{pmatrix}
		1,
		0,
		\cdots,
		0
	\end{pmatrix}^\top, e_2 = \begin{pmatrix}
	0,
	1,
	\cdots,
	0
\end{pmatrix}^\top,\cdots,e_N = \begin{pmatrix}
0,
\cdots,
0,
1
\end{pmatrix}^\top.$$
Suppose $$\varphi_t = \theta + a_1\tau+\cdots+a_r\tau^r.$$
For each $1\leq i\leq r$ and $1\leq j\leq d$, let $$e_{ij} = \begin{pmatrix}
	0,\cdots,e_j,\cdots,0
\end{pmatrix}^\top\in E[t]^{rN}$$ with $e_j$ in the $i$-th block and other blocks 0. Let $$T = \begin{pmatrix}
0 & \cdots & 0 & \frac{t-\theta}{a_r}\Phi\\
\Phi & \cdots & 0 & -\frac{a_1}{a_r}\Phi\\
\vdots & \ddots & \vdots & \vdots\\
0 & \cdots & \Phi & -\frac{a_{r-1}}{a_r}\Phi
\end{pmatrix}\in \Mat_{rN\times rN}(E[t]).$$
For each $1\leq i\leq r$ and $1\leq j\leq N$, we define $$\tau e_{ij} = Te_{ij}.$$ 
First, we observe that \begin{equation}\label{Phi_ej_expression}
    \Phi e_j = \begin{cases}
    e_{j+1}, & \textnormal{ if }1\leq j\leq N-1\\
    \sum\limits_{k=0}^{N-1} f_ke_{k+1}, & \textnormal{ if }j=N.
\end{cases}
\end{equation}
Then, one sees that if $1\leq i\leq r-1$, we have $$Te_{ij} = 
\begin{pmatrix}
    0 & \cdots & 0 & \frac{t-\theta}{a_r}\Phi\\
\Phi & \cdots & 0 & -\frac{a_1}{a_r}\Phi\\
\vdots & \ddots & \vdots & \vdots\\
0 & \cdots & \Phi & -\frac{a_{r-1}}{a_r}\Phi
\end{pmatrix}\begin{pmatrix}
    \vdots\\
    \tikzmarknode{A}{e_j}\\
    0\\
    \vdots\\
    0 
\end{pmatrix}\ \ =\ \ \begin{pmatrix}
    \vdots\\
    0\\
   \tikzmarknode{B}{\Phi e_j} \\
    \vdots\\
    0 
\end{pmatrix}.\begin{tikzpicture}[overlay, remember picture]
    \draw[<-] (A.east) -- ++(0.35,0.35) node[right] {\scriptsize $i$-th};
    \draw[<-] (B.east) -- ++(0.35,0.35) node[right] {\scriptsize $(i+1)$-th};
\end{tikzpicture}$$
\noindent Thus, together with (\ref{Phi_ej_expression}), we have \begin{equation}
    \tau e_{ij} = \begin{cases}
	e_{i+1,j+1}, & 1\leq j\leq N-1\\
	\sum\limits_{k=0}^{N-1}f_ke_{i+1,k+1}, & j = N.
\end{cases}
\end{equation}
If $i=r$, then $$Te_{rj} = \begin{pmatrix}
    0 & \cdots & 0 & \frac{t-\theta}{a_r}\Phi\\
\Phi & \cdots & 0 & -\frac{a_1}{a_r}\Phi\\
\vdots & \ddots & \vdots & \vdots\\
0 & \cdots & \Phi & -\frac{a_{r-1}}{a_r}\Phi
\end{pmatrix}\begin{pmatrix}
    0\\
    \vdots\\
   e_j\\
\end{pmatrix} =\begin{pmatrix}
\frac{t-\theta}{a_r}\Phi e_j\\
-\frac{a_1}{a_r}\Phi e_j\\
 \vdots\\
 -\frac{a_{r-1}}{a_r}\Phi e_j
\end{pmatrix}.$$
Thus, $$\tau e_{rj} = \begin{cases}
\frac{t-\theta}{a_r} e_{1, j+1} - \sum\limits_{m=1}^{r-1} \frac{a_m}{a_r} e_{m+1, j+1}, & 1\leq j\leq N-1\\
\sum\limits_{k=0}^{N-1} f_k \left( \frac{t-\theta}{a_r} e_{1, k+1} - \sum\limits_{m=1}^{r-1} \frac{a_m}{a_r} e_{m+1, k+1} \right), & j= d.
\end{cases}$$
Thus, one can check that for every $k\geq 1$, $$\left\{e_{11},\cdots,e_{1N},\tau e_{11},\cdots,\tau e_{1N},\cdots,\tau^k e_{11},\cdots,\tau^k e_{1N}\right\}$$ are linearly independent over $E$. Thus, $\{e_{11},\cdots,e_{1N}\}$ are linearly independent over $E[\tau]$. So, $$\MM(\varphi,\rho)\cong \Span_{E[\tau]}\{e_{11},\cdots,e_{1N}\}.$$

Furthermore, from the above expression of $\tau e_{ij}$, we see that $(t-\theta)e_{1j}\in \tau \MM(\varphi,\rho)$, hence $(t-\theta)\MM(\varphi,\rho)\subseteq \tau \MM(\varphi,\rho)$. Thus, by \cite[Theorem 4.2.9]{vanderHeiden2003}, there exists an Abelian $\bA$-module $\mcE$ over $E$ such that $$\MM(\varphi,\rho)\cong \MM(\mcE).$$
Note that this isomorphism is defined over $E$.
\end{proof}
\begin{Def}
    Let $\varphi:\bA\to E[\tau]$ be a Drinfeld $\bA$-module and $\rho: G\to \GL(\conj{\FF}_q)$ an Artin representation of $G$. We call the abelian $\bA$-module $\mcE = \mcE(\varphi,\rho)$ in Theorem \ref{existsAnderson} an \emph{Artin twist} of $\varphi$ by $\rho$.
\end{Def}

\subsubsection{Defining equation of $\mcE$ in the case $\bA = \FF_q[t]$}
In the case $\bA = \FF_q[t]$, we can pin down the expression of $\mcE$ under a fundamental basis $\vec{\mathbf{u}},\vec{\mathbf{u}}^{(1)},\cdots, \vec{\mathbf{u}}^{(N-1)}$. In this subsection, we assume that $\bA = \FF_q[t]$. 

	We first summarize the $\tau$-action on the $\bA_E$-basis $\{e_{ij}\}_{1\le i\le r,\;1\le j\leq N}$, 
	\[
	\tau e_{ij}=\begin{cases}
		e_{i+1,j+1}, & 1\le i\le r-1,\ 1\le j\le N-1,\\[4pt]
		\dis \sum\limits_{k=0}^{N-1}f_ke_{i+1,k+1}, & 1\le i\le r-1,\ j=N,\\[4pt]
		\displaystyle\frac{t-\theta}{a_r}e_{1,j+1}-\sum_{m=1}^{r-1}\frac{a_m}{a_r}e_{m+1,j+1},
		& i=r,\ 1\le j\le N-1,\\[8pt]
		\displaystyle\sum\limits_{k=0}^{N-1} f_k \left( \frac{t-\theta}{a_r} e_{1, k+1} - \sum\limits_{m=1}^{r-1} \frac{a_m}{a_r} e_{m+1, k+1} \right),
		& i=r,\ j=N.
	\end{cases}
	\]
    We have the following result.
\begin{prop}\label{prop:te1j}
	Let $E_i = (e_{i1}, e_{i2}, \cdots, e_{iN})^\top$ for each $1\leq i\leq N$.
    
	Then for each $1\le j\le N$ we have the identity
	\begin{equation}\label{eq:te1j-general}
		\tau E_i = \begin{cases}
		    \Phi^\top E_{i+1}, & 1\leq i\leq r-1\\
            \dis \Phi^\top\left(\frac{t-\theta}{a_r}E_1 - \sum_{m=1}^{r-1}\frac{a_m}{a_r}E_{m+1}\right), & i = r
		\end{cases}
	\end{equation}

\end{prop}

\begin{proof}
The result is immediate from the $\tau$-action.
\end{proof}
\begin{thm}\label{thmB}
    Let $\varphi: \bA\to E[\tau]$ be a Drinfeld module over $E$ with $$\varphi_t = \theta + a_1\tau +\cdots + a_r\tau^r$$ and $\rho:G_E\to \GL_n(\conj{\FF}_q)$ be an Artin representation. Let $N = nd$ with $d = d_\rho = [\FF_q(\rho): \FF_q]$. Then, $\mcE$ has a model over $E$ given by
    $$\mcE_t = \theta I_N + a_1\Psi \tau +a_2 \Psi\Psi^{(1)}\tau^2+\cdots+ a_r\Psi\Psi^{(1)}\cdots \Psi^{(r-1)}\tau^r,$$ where $\Psi = (\Phi(\vec{\mathbf{u}})^\top)^{-1}$ and $\vec{\mathbf{u}}\in \operatorname{Sol}_E(\rho,\vec{\boldsymbol{\alpha}},\bm{w})$ is a fundamental solution.
\end{thm}
\begin{proof}
    By Proposition \ref{prop:te1j}, we see that for each $2\leq i\leq r$, we have $$E_i = \Psi\Psi^{(1)}\cdots \Psi^{(i-2)}\tau^{i-1}E_1$$ and $$\Psi\tau E_r = \left(\frac{t-\theta}{a_r}E_1 - \sum_{m=1}^{r-1}\frac{a_m}{a_r}E_{m+1}\right).$$

    On the other hand, $$\Psi\tau E_r = \Psi\Psi^{(1)}\cdots \Psi^{(r-1)}\tau^r E_1.$$ Comparing the above two equalities, we obtain $$\begin{aligned}
        (t-\theta)E_1 &= \sum_{m=1}^{r-1} a_mE_{m+1} + a_r \Psi\Psi^{(1)}\cdots \Psi^{(r-1)}\tau^r E_1\\
        & = \sum_{m=1}^{r-1} a_m\Psi\Psi^{(1)}\cdots \Psi^{(m-1)}\tau^{m}E_1 + a_r \Psi\Psi^{(1)}\cdots \Psi^{(r-1)}\tau^r E_1\\
        & = \sum_{m=1}^{r} a_m\Psi\Psi^{(1)}\cdots \Psi^{(m-1)}\tau^{m}E_1
    \end{aligned}$$
    Thus, $$tE_1 =\left(\theta I_N+ \sum_{m=1}^{r} a_m\Psi\Psi^{(1)}\cdots \Psi^{(m-1)}\tau^{m}\right)E_1.$$ The result now follows.
\end{proof}
\begin{cor}\label{twsit_is_uniformizable}
    The Anderson $\bA$-module $\mcE$ over $E$ is uniformizable.
\end{cor}
\begin{proof}
    Without loss of generality, we may assume $\bA = \FF_q[t]$. Write $P = P(\vec{\mathbf{u}}) = (M(\vec{\mathbf{u}})^\top)^{-1}\in \Mat_{N\times N}(E^\sep)$. By Proposition \ref{Phi_in_terms_of_Mu}, we have $$P^{-1}P^{(1)} = \Psi.$$ Then, for each $i\geq 1$, we have $$\Psi\Psi^{(1)}\cdots \Psi^{(i-1)} = (P^{-1}P^{(1)})((P^{(1)})^{-1}P^{(2)})\cdots ((P^{(i-1)})^{-1}P^{(i)}) = P^{-1}P^{(i)}.$$

    Thus, $$\mcE_t = \theta I_N + \sum_{i=1}^ra_iP^{-1}P^{(i)}\tau^i = P^{-1}\varphi_t^{\oplus N}P.$$ Thus, $\mcE$ is isomorphic to $\varphi^{\oplus N}$ over $E^\sep$ (but not isomorphic over $E$). Since $\varphi$ is a Drinfeld module, $\varphi^{\oplus N}$ is uniformizable. Hence, $\mcE$ is also uniformizable.
\end{proof}
\begin{lem}
    Let $\varphi:\bA\to E[\tau]$ be a Drinfeld module over $E$ and $N$ a positive integer. Let $\varphi^{\oplus N}$ be the abelian $\bA$-module given by a direct sum of $N$ copies of $\varphi$. Then, we have $$\exp_{\varphi^{\oplus N}}: \Lie_{\varphi^{\oplus N}}(E)\to \varphi^{\oplus N}(E),$$ $$\begin{pmatrix}
        x_1\\
        \vdots\\
        x_N
    \end{pmatrix}\mapsto \begin{pmatrix}
        \exp_\varphi(x_1)\\
        \vdots\\
        \exp_{\varphi}(x_N)
    \end{pmatrix}$$ and $$\log_{\varphi^{\oplus N}}:  \varphi^{\oplus N}(E)\dashrightarrow \Lie_{\varphi^{\oplus N}}(E),$$ $$\begin{pmatrix}
        x_1\\
        \vdots\\
        x_N
    \end{pmatrix}\mapsto \begin{pmatrix}
        \log_\varphi(x_1)\\
        \vdots\\
        \log_{\varphi}(x_N)
    \end{pmatrix},$$ where by $\dashrightarrow$ we mean $\log$ is only defined on a neighborhood of 0.
\end{lem}
\begin{proof}
    It is straightforward to check that the given expressions have identity matrix as their constant terms and satisfy the functional equations $\exp_{\varphi^{\oplus N}}\circ \partial \varphi^{\oplus N} = \varphi^{\oplus N}\circ \exp_{\varphi^{\oplus N}}$ and $\log_{\varphi^{\oplus N}}\circ \varphi^{\oplus N} = \partial \varphi^{\oplus N}\circ \log_{\varphi^{\oplus N}}$ respectively.
\end{proof}
A similar argument yields the following result.
\begin{cor}\label{P_exp_log}
    Let $P = P(\vec{\mathbf{u}}) = (M(\vec{\mathbf{u}})^\top)^{-1}\in \Mat_{N\times N}(E^\sep)$. Then the exponential map $\exp_{\mcE}:\Lie_\mcE(E)\to \mcE(E)$ is given by $$\bm{x} = \begin{pmatrix}
        x_1\\
        \vdots\\
        x_N
    \end{pmatrix}\mapsto P^{-1}\begin{pmatrix}
        \exp_\varphi((P\bm{x})_1)\\
        \vdots\\
       \exp_\varphi((P\bm{x})_N)
    \end{pmatrix},$$ where $(P\bm{x})_i$ denotes the $i$-th coordinate of the column vector $P\bm{x}$.

    Similarly, the logarithm map $\log_\mcE: \mcE(E)\dashrightarrow \Lie_\mcE(E)$ is given by $$\bm{x} = \begin{pmatrix}
        x_1\\
        \vdots\\
        x_N
    \end{pmatrix}\mapsto P^{-1}\begin{pmatrix}
        \log_\varphi((P\bm{x})_1)\\
        \vdots\\
       \log_\varphi((P\bm{x})_N)
    \end{pmatrix}.$$
\end{cor}
\subsection{Goss $L$-series of $\mcE(\varphi,\rho)$}
In this section, we will compute the Goss $L$-series of $\mcE=\mcE(\varphi,\rho)$. This can be done by a computation on the motivic side. The idea is motivated by \cite{FontaineOuyang2022}.

First, we observe the following result.

\begin{prop}\label{coh_rho}
    We have an isomorphism of $K_\mfl[\Gal(E^\sep/E)]$-modules.
    $$H^1_{\mfl}(\MM(\rho),K_\mfl) \cong K_\mfl\otimes_{\FF_q}(\Res_{\FF_q(\rho)/\FF_q}V_\rho).$$
\end{prop}
\begin{proof}
    Let $M = \MM(\rho)$ and $V = \Res_{\FF_q(\rho)/\FF_q}V_\rho$. Let $R = \bA_\mfl(E^\sep) = \varprojlim (\bA\otimes E^\sep)/\mfl^n(\bA\otimes E^\sep)$. Thus, we have a natural inclusion $E^\sep\hookrightarrow R$, $x\mapsto 1\otimes x$. Hence, $R$ is an $E^\sep$-algebra.
		
		From Lemma \ref{G-inv_bc_to_Esep}, we have an isomorphism over $E^\sep$:
		$$ \alpha_\rho: E^\sep \otimes_E \bD(V) \to E^\sep \otimes_{\FF_q} V,$$ which is $G_E$-equivariant and $\tau$-equivariant.
        
		Apply the exact functor $R\otimes_{E^\sep}-$ to both sides, we get an isomorphism of $\bA_\mfl(E^\sep)$-modules
    $$\beta_\rho: \bA_\mfl(E^\sep) \otimes_E\bD(\Res_{\FF_{q}(\rho)/\FF_q}V_\rho)  \stackrel{\sim}{\to}\bA_\mfl(E^\sep)\otimes_{\FF_q}\Res_{\FF_q(\rho)/\FF_q}V_\rho.$$
    
    The $G_E$-equivariant and $\tau$-equivariant follows from the definition.
    
    Then, we may identify $M_\mfl(E^\sep) = M\otimes_{\bA_E}\bA_\mfl(E^\sep) = \bD(V)\otimes_E\bA_\mfl(E^\sep).$ Thus, we have a $G_E$-equivariant isomorphism $$\gamma_\rho: M_\mfl(E^\sep)\stackrel{\sim}{\to} \bA_\mfl(E^\sep)\otimes_{\FF_q}V,$$ which is also $\tau$-equivariant.

    Choose an $\FF_q$-basis $\{v_1,\cdots,v_\ell\}$ for $V$. Then, for every element $m$ of $M_\mfl(E^\sep)$ can be written as $$\gamma_\rho(m) = \sum_{i=1}^\ell x_i\otimes v_i\quad \textnormal{with unique }x_i\in \bA_\mfl(E^\sep).$$ Thus, $\gamma_\rho(\tau(m)) = \tau(\gamma_\rho(m)) = \sum_{i=1}^\ell x_i^{(1)}\otimes v_i$. We see that $\tau(m) = m$ if and only if $\gamma_\rho(\tau(m)) = \gamma_\rho(m)$ if and only if $x_i^{(1)} = x_i$ for all $1\leq i\leq \ell$. Hence, by letting $\gamma_\rho$ restrict to $M_\mfl^\tau$, we get an isomorphism
    $$M_\mfl^\tau \stackrel{\sim}{\to} \bA_\mfl\otimes_{\FF_q}V,$$ which is $G_E$-equivariant.

Now, tensor with $K_\mfl\otimes_{\bA_\mfl}-$ and by the definition of $M = \MM(\rho)$, we have a desired isomorphism of $K_\mfl[\Gal(E^\sep/E)]$-modules.
    $$H^1_{\mfl}(\MM(\rho),K_\mfl) \cong K_\mfl\otimes_{\FF_q}(\Res_{\FF_q(\rho)/\FF_q}V_\rho) .$$
\end{proof}

We also need the following lemmas. For a linear operator $T: V\to V$ on a vector space $V$, we shall use $\Char(T,V,X)$ to denote the characteristic polynomial of $T$ on $X$ in a variable $X$.
\begin{lem}\label{char_polys}
    Let $\wp$ be a prime in $E$, then $$\Char(\Frob_\wp, \Res_{\FF_q(\rho)/\FF_q}V_\rho, X) = \prod_{i=0}^{d-1} \Char(\Frob_\wp, V_\rho^{(i)},X)$$
\end{lem}
\begin{proof}
    Indeed, note that using a similar argument as in Lemma \ref{split_base_change}, we can show that there exists an isomorphism of $\FF_q(\rho)$-representations of $G$: $$(\Res_{\FF_q(\rho)/\FF_q}V_\rho)\otimes_{\FF_q} \FF_q(\rho)\cong \bigoplus_{i=0}^{d-1} V_\rho^{(i)}.$$ 

    Thus, $$\begin{aligned}
        \Char(\Frob_\wp, \Res_{\FF_q(\rho)/\FF_q}V_\rho, X) & = \det_{\FF_q} (X - \rho(\Frob_\wp)| \Res_{\FF_q(\rho)/\FF_q}V_\rho)\\
    & = \det_{\FF_q(\rho)} (X - \rho(\Frob_\wp)| (\Res_{\FF_q(\rho)/\FF_q}V_\rho)\otimes_{\FF_q}\FF_q(\rho))\\
    & = \det_{\FF_q(\rho)} \left(X - \rho(\Frob_\wp)| \bigoplus_{i=0}^{d-1} V_\rho^{(i)}\right)\\
    & = \prod_{i=0}^{d-1} \det_{\FF_q(\rho)} \left(X - \rho(\Frob_\wp)|V_\rho^{(i)}\right)\\
    & = \prod_{i=0}^{d-1} \Char(\Frob_\wp, V_\rho^{(i)},X).
    \end{aligned}
    $$
\end{proof}
\begin{lem}\label{H1_tensorproduct_motives}
    Let $M_1$ and $M_2$ be two $\bA$-motives over $E$ of ranks $r_1$ and $r_2$ respectively. Let $H^1_{\mfl}(M, \bA_\mfl)$ denote the $\mfl$-adic realization. Then there exists a canonical isomorphism of $\bA_\mfl[\Gal(E^\sep/E)]$-modules
		\[
		\widetilde{\Pi}: H^1_{\mfl}(M_1, \bA_\mfl) \otimes_{\bA_\mfl} H^1_{\mfl}(M_2, \bA_\mfl) \xrightarrow{\sim} H^1_{\mfl}(M_1 \otimes M_2, \bA_\mfl).
		\]
\end{lem}
\begin{proof}
    Let $M = M_1\otimes M_2$. Recall that $$M_\mfl(E^\sep) = M\otimes_{\bA_E}\bA_\mfl(E^\sep).$$ Thus, we have a natural isomorphism of $\bA_\mfl(E^\sep)$-modules
		$$\Pi: M_{1,\mfl}(E^\sep)\otimes_{\bA_\mfl(E^\sep)}M_{2,\mfl}(E^\sep)\to M_\mfl(E^\sep),$$ given by $$(m_1\otimes \alpha)\otimes (m_2\otimes \beta)\mapsto (m_1\otimes m_2)\otimes (\alpha\beta).
		$$ Moreover, if we equip left hand side with diagonal $\tau$-action, then this isomorphism is $\tau$-equivariant by the definition of tensor product of $\bA$-motives.

    By restriction to submodules $M_1^\tau\otimes_{\bA_\mfl} M_2^\tau$, it induces an $\bA_\mfl$-linear map $$\widetilde{\Pi}: M_1^\tau\otimes_{\bA_\mfl}M_2^\tau\to M^\tau$$ on the tensor product of the $\mfl$-adic realizations. Indeed, let $x \in M_1^\tau$ and $y \in M_2^\tau$, then by definition, $\tau(x) = x$ and $\tau(y) = y$. Thus, $$\tau(\widetilde{\Pi}(x\otimes y)) = \widetilde{\Pi}(\tau(x\otimes y)) = \widetilde{\Pi}(\tau(x)\otimes \tau(y)) = \widetilde{\Pi}(x\otimes y).$$ Moreover, this map is injective as $\Pi$ is. Since $\rk_{\bA_\mfl}(M_1^\tau) = r_1$, $\rk_{\bA_\mfl}(M_2^\tau) = r_2$, $\rk_{\bA_\mfl}(M) = r_1r_2$. Since they are all free modules over the discrete valuation ring $\bA_\mfl$, $\widetilde{\Pi}$ is thus an isomorphism of $\bA_\mfl$-modules.

    It remains to check that $\widetilde{\Pi}$ commutes with the action of $G_E= \text{Gal}(E^{\sep}/E)$.
		Let $\sigma \in G_E$. The Galois action on $M_1^\tau\otimes_{\bA_\mfl}M_2^\tau$ is diagonal, i.e. for $x \in M_1^\tau$ and $y \in M_2^\tau$:
			\[
			\sigma \cdot (x \otimes y) := \sigma(x) \otimes \sigma(y).
			\]
	
   On the right hand side, the action of $G_E$ on $M_{\mfl}(E^\sep)$ is acting on the coefficients $\bA_\mfl(E^\sep)$. We may write $x$ and $y$ as $$x = \sum_i u_i \otimes \alpha_i\in M_1^\tau, \quad \textnormal{and}\quad y = \sum_j v_j \otimes \beta_j\in M_2^\tau,$$ where $u_i \in M_1, v_j \in M_2$, and $\alpha_i, \beta_j \in \bA_\mfl(E^\sep)$. Then $\widetilde{\Pi}(x \otimes y) = \sum_{i,j} (u_i \otimes v_j) \otimes (\alpha_i \beta_j)$.

    Thus, we have $$\begin{aligned}
\widetilde{\Pi}(\sigma(x\otimes y)) & = \widetilde{\Pi}\left(\sigma(x)\otimes \sigma(y)\right)\\
 &= \widetilde{\Pi}\left(\left(\sum_i u_i \otimes \sigma(\alpha_i)\right) \otimes \left(\sum_j v_j \otimes \sigma(\beta_j)\right)\right)\\
 &=\sum_{i,j} (u_i \otimes v_j) \otimes \sigma(\alpha_i)\sigma(\beta_j)\\
&= \sum_{i,j} (u_i \otimes v_j) \otimes \sigma(\alpha_i \beta_j) \\
   & = \sigma \cdot \widetilde{\Pi}(x \otimes y).
    \end{aligned}$$
	
This shows that $\widetilde{\Pi}$ is $G_E$-equivariant. Therefore, $\widetilde{\Pi}$ is an isomorphism of $\bA_\mfl[\Gal(E^\sep/E)]$-modules.
\end{proof}
As a consequence of the above results, we have the following proposition.
\begin{prop}\label{motivic_L_series}
    Let $\varphi:\bA\to E[\tau]$ be a Drinfeld $\bA$-module over $E$ and $\rho: G_E\to \GL(V_\rho)$ an Artin representation. Then, there exists a finite set $S$ of places such that $$L_S(\MM(\varphi,\rho),s) = \prod_{i=0}^{d-1} L_S(\varphi^\vee, \rho^{(i)}, s),$$ where $d = [\FF_q(\rho):\FF_q]$ and $\rho^{(i)}$ denotes the $i$-th Frobenius twist of the Artin representation $\rho$. 
\end{prop}
\begin{proof}
Let $S$ contain all bad primes of $\MM(\varphi,\rho)$, $\MM(\varphi)$, and ramified primes of $\rho$ together with $\infty$. This is a finite set by Proposition $\ref{finite_bad}$.

    Let $\wp\notin S$, then by \cite[Proposition 4.49]{Gazda21motcoh}, $H^1_\mfl(\MM(\varphi,\rho)),H^1_\mfl(\MM(\varphi))$ and $H^1_\mfl(\rho)$ are unramified at $\mfp$.
    By the definition of $\MM(\varphi,\rho)$, Proposition \ref{coh_rho}, Lemma \ref{char_polys} and \ref{H1_tensorproduct_motives}, we have $$\begin{aligned}
        P_\wp(\MM(\varphi,\rho),X) & = \Char(\Frob_\wp, H^1_\mfl(\MM(\varphi,\rho),K_\mfl), X)\\
        & = \Char(\Frob_\wp, H^1_\mfl(\MM(\varphi),K_\mfl)\otimes_{K_\mfl} H^1_{\mfl}(\MM(\rho),K_\mfl), X)\\
        & = \Char(\Frob_\wp, H^1_\mfl(\MM(\varphi),K_\mfl), X)\otimes \Char(\Frob_\wp, \Res_{\FF_q(\rho)/\FF_q}V_\rho, X)\\
        & =  \prod_{i=0}^{d-1}\Char(\Frob_\wp, V_\mfl^\vee(\varphi), X)\otimes \Char(\Frob_\wp, V_\rho^{(i)},X).
    \end{aligned} $$

    As a result, by Proposition \ref{prop:good_primes}, for each integer $s$, we have $$\begin{aligned}
        L_S(\MM(\varphi,\rho),s) & = \prod_{i=0}^{d-1} L_S(\varphi^\vee, \rho^{(i)}, s).
    \end{aligned}$$
\end{proof}
\begin{thm}\label{thmC}
    Let $\varphi:\bA\to E[\tau]$ be a Drinfeld $\bA$-module over $E$ and $\rho: G_E\to \GL_n(\conj{\FF}_q)$ an Artin representation. Let $\mcE = \mcE(\varphi,\rho)$ be an Artin twist of $\varphi$ by $\rho$. Then, there exists a finite set $S$ of places such that $$L_S(\mcE^\vee,s) = \prod_{i=0}^{d-1}L_S(\varphi^\vee,\rho^{(i)},s).$$ 

    In the case $\bA = \FF_q[t]$, we have $$L_S(\mcE^\vee,s) = N_{K_\infty(\rho)/K_\infty}(L_S(\varphi^\vee,\rho,s))$$
\end{thm}
\begin{proof}
    This follows from Proposition \ref{motivic_L_series} and (\ref{L-series_motivic_module}).

    In the case $\bA = \FF_q[t]$, $L_S(\mcE^\vee,s)\in K_\infty$ and $L_S(\varphi^\vee,\rho^{(i)},s)\in K_\infty(\rho)$. Thus, for each integer $s$, we have \[L_S(\mcE^\vee,s) 
    = \prod_{\gamma\in \Gal(K_\infty(\rho)/K_\infty)} (L_S(\varphi^\vee,\rho,s))^\gamma
         = N_{K_\infty(\rho)/K_\infty} (L_S(\varphi^\vee,\rho,s)).\]
\end{proof}
    
By Taelman's class number formula, we obtain the following corollary. Note that although $\mcE$ in the theorem is defined over the field $E$, we may replace it by an integral model over $\mcO_E$ via clearing denominators. This modification only affects finitely many primes and introduces a rational factor in $K^*$.

For the rest of this section, we assume $\bA = \FF_q[t]$.
\begin{lem}\label{L-value_at0}
    Let $S$ denote the set of bad primes, then $$L_S(\mcE^\vee,0) = \prod_{\wp\notin S}\frac{[\Lie_{\mcE_\wp}(\FF_\wp)]_{A}}{[\mcE_\wp(\FF_\wp)]_A}.$$
\end{lem}
\begin{proof}
    We give a sketch of proof. First, we identify $\bA\stackrel{\sim}{\to} A$ via $t\mapsto \theta$. For any finitely generated $\bA$-module $M$, set $[M]_A: = [M]_{\bA}|_{t=\theta}\in A_+$. Fix a prime $\wp\notin S$ and $m = [\FF_\wp:\FF_q]$. Suppose $\mcE$ has dimension $d$ and let $\Lie_{\mcE_\wp^*}(\FF_\wp) = \FF_{\wp}^{\oplus d}$ be equipped with the $\bA$-module structure via $a\cdot x = (\partial \mcE_a)^\top(x)$ for all $a\in \bA$. Then, $\Lie_{\mcE_\wp^*}(\FF_\wp)$ is naturally an $\bA_{\FF_\wp}$-module via scalar multiplication.

    Denote by $L$ the $\bA_{\FF_\wp}$-structure of $\Lie_{\mcE_\wp^*}(\FF_\wp)$ and $\Res_{\FF_\wp/\FF_q} L$ the $\bA$-module structure of $\Lie_{\mcE_\wp^*}(\FF_\wp)$. By \cite[Proposition 3.2.6]{huang2022convolutions}, we may identify $L$ with $M/\tau M$ as $\bA_{\FF_{\wp}}$-modules. Write $L^{(i)} = \tau^i(M/\tau M)$. Again, using a similar argument as in Lemma \ref{split_base_change}, we can show that there is an $\bA_{\FF_\wp}$-module isomorphism $$(\Res_{\FF_\wp/\FF_q}L)\otimes_{\FF_q}\FF_\wp\cong \bigoplus_{i=0}^{m-1}L^{(i)}.$$ Thus, $$\begin{aligned}
        [\Res_{\FF_\wp/\FF_q}L]_A &= [(\Res_{\FF_\wp/\FF_q}L)\otimes_{\FF_q}\FF_\wp]_{A_{\FF_\wp}}
        = \left[\bigoplus_{i=0}^{m-1}L^{(i)}\right]_{A_{\FF_\wp}} \\
       & = \prod_{i=0}^{m-1}[\tau^iM/\tau^{i+1}M)]_{A_{\FF_\wp}}
        = [M/\tau^d M]_{A_{\FF_\wp}}\\
       & = \gamma \det(\tau^d|M) 
        = \gamma P_\wp(\mcE_\wp, 0)
    \end{aligned}$$ for some $\gamma\in \FF_q^*$ such that $\gamma P_\wp(\mcE_\wp,0)$ is monic in $t$. In other words, $[\Lie_{\mcE_\wp}(\FF_\wp)]_A = [\Lie_{\mcE^*_\wp}(\FF_\wp)]_A = \gamma P_\wp(\mcE_\wp,0)$. By \cite[Corollary 3.7.8]{huang2022convolutions}, $[\mcE_\wp(\FF_\wp)]_A = \gamma' P_\wp(\mcE_\wp,1)$ for some $\gamma'\in \FF_q^*$. 
    
    Pick a prime $\mfl\in \Spec(\bA)$ such that $\wp$ is not lying above $\mfl$. Let $P_\wp(\varphi,X)$ be the characteristic polynomial of $\Frob_\wp$ acting on the $\mfl$-adic Tate module of $\varphi$. Let $P_\wp(\rho^{(i)},X)$ be the characteristic polynomial of $\rho^{(i)}$. Then, using a similar argument as in the proof of Proposition \ref{motivic_L_series}, we see that $P_\wp(\mcE_\wp,X) = P_\wp(\mcE,X)  = \prod_{i=0}^{d-1} P_\wp(\varphi,X)\otimes P^\vee_\wp(\rho^{(i)},X)$. By \cite[Theorem 5.1(iii)]{gekeler91}, the roots $\{x_k: 1\leq k\leq r\}$ of $P_\wp(\varphi,X)$ satisfies $|x_k|\leq q^{\deg \wp/r}$, where $r$ is the rank of the Drinfeld module $\varphi$. Since $\rho^{(i)}$ are representations defined over $\FF_q(\rho)$, the roots of its characteristic polynomial must have absolute value 1. Using a similar argument as in \cite[\S 3]{chang2018logalgebraic}, one can show that $Q_\wp^\vee(\mcE_\wp,X) = 1 + a_1X + \cdots + a_mX^m$ for some integer $m$ and $a_i\in \conj{K}$ satisfying $|a_i|<1$. Thus, $Q_\wp^\vee(\mcE_\wp,1)$ must have sign 1. Then, by a similar argument as in \cite[\S 3]{chang2018logalgebraic}, $\dis Q^\vee_\wp(\mcE_\wp,1) = \frac{P_\wp(\mcE_\wp,1)}{P_\wp(\mcE_\wp,0)}$. Hence $\gamma' = \gamma$ and $$L_\wp(\mcE_\wp^\vee,0) = Q^\vee_\wp(\mcE_\wp,1)^{-1} = \frac{P_\wp(\mcE_\wp,0)}{P_\wp(\mcE_\wp,1)} = \frac{[\Lie_{\mcE_\wp}(\FF_\wp)]_A}{[\mcE_\wp(\FF_\wp)]_A}.$$
    Now the statement follows from Corollary \ref{reduction_at_good}.
\end{proof}
\begin{cor}\label{corD}
	Let $\varphi$ be a Drinfeld module over $E$ of rank $r$ and $\rho: G\to \GL_n(\conj{\FF}_q)$ be an Artin representation with $d = d_\rho\geq 1$. Let $\mcE$ be a model over $\mcO_E$ of the Artin twist $\mcE(\varphi,\rho)$. Then, there exists a constant $C = C(\varphi,\rho)\in K^*$ such that $$N_{K_\infty(\rho)/K_\infty}(L_S(\varphi^\vee, \rho, 0)) = C\cdot \Reg(\mcE/\mcO_E)\cdot h(\mcE/\mcO_E).$$
\end{cor}
\begin{proof}
Let $S$ be the finite set of places as in Theorem \ref{thmC}. Set $$C_1 = \prod_{\wp\in S} \frac{[\Lie_{\mcE_\wp}(\FF_\wp)]_{\bA}}{[\mcE_\wp(\FF_\wp)]_\bA}\in K^*$$ and $$C_2 = \prod_{\wp\in S}L_\wp(\mcE^\vee,0).$$ These values depend only on $\varphi$ and $\rho$.
    By Theorem \ref{thmC} and Lemma \ref{L-value_at0}, we see that $$\begin{aligned}
        C_1\cdot N_{K_\infty(\rho)/K_\infty}(L_S(\varphi^\vee, \rho, 0)) &= C_1L_S(\mcE^\vee,0) = C_1 \prod_{\wp\notin S}\frac{[\Lie_{\mcE_\wp}(\FF_\wp)]_{\bA}}{[\mcE_\wp(\FF_\wp)]_\bA}\\
        &= \prod_{\wp\in \Spec \mcO_E\backslash\{0\}}\frac{[\Lie_{\mcE_\wp}(\FF_\wp)]_{\bA}}{[\mcE_\wp(\FF_\wp)]_\bA} = \Reg(\mcE/\mcO_E)\cdot h(\mcE/\mcO_E).
    \end{aligned}$$
    The result now follows.
\end{proof}
Let $\varphi: \bA\to \mcO_E[\tau]$ be a Drinfeld module over $\mcO_E$. Let $L/E$ be a finite Galois extension such that $p\nmid [L:E]$. Then, $\varphi$ is defined over $\mcO_L$.

Set $G = \Gal(L/E)$ and denote by $\operatorname{irr}(G)$ the set of all irreducible representations of $G$. Recall that $$L_S(\varphi^\vee/\mcO_L,0) = \prod_{\rho\in \operatorname{irr}(G)}L_S(\varphi^\vee/\mcO_E,\rho,0)$$ for any finite set $S$ of places containing infinity.

Let $[\rho]$ denote the orbit of $\rho$ under Frobenius twisting, i.e. $\{\rho,\rho^{(1)},\cdots \}$. Then, we have the following result.

\begin{cor}\label{corE}
Let $\varphi: \bA\to \mcO_E[\tau]$ be a Drinfeld module over $\mcO_E$. Let $L/E$ be a Galois extension of degree $m$ with $p\nmid m$. Then there exists an element $c\in K^*$ such that
    $$\Reg(\varphi/\mcO_L) = c\cdot \prod_{[\rho]}\Reg(\mcE(\varphi,\rho)/\mcO_E)^{\dim\rho},$$ where $\mcE(\varphi,\rho)$ denotes an integral model over $\mcO_E$ obtained by clearing denominators and the product runs through all Frobenius orbits $[\rho]$ of irreducible representations of $\Gal(L/E)$.
\end{cor}
\begin{proof}
    Let $S$ be as before. Since $\dis L_S(\varphi^\vee/\mcO_L,0) = \prod_{\rho\in \operatorname{irr}(G)}L_S(\varphi^\vee,\rho,0)^{\dim \rho} $, we see that $$L_S(\varphi^\vee/\mcO_L,0) = \prod_{[\rho]}\prod_{i}L_S(\varphi^\vee,\rho^{(i)},0)^{\dim \rho} = \prod_{[\rho]}N_{K_\infty(\rho)/K_\infty}(L_S(\varphi^\vee,\rho,0))^{\dim \rho}.$$

    Now, for each $\rho$, we can form the $t$-module $\mcE(\varphi,\rho)$ and by the construction of $\operatorname{Sol}_E(\rho)$, it actually depends only on the class $[\rho]$. We can further clear the denominators and obtain an $\mcO_E$-integral model, denoted by $\mcE(\varphi,\rho)/\mcO_E$. Notice that $h(\mcE(\varphi,\rho)/\mcO_E)\in A$. By Corollary \ref{corD}, we have $$N_{K_\infty(\rho)/K_\infty}(L_S(\varphi^\vee,\rho,0)) = C_{[\rho]} \cdot \Reg(\mcE(\varphi,\rho)/\mcO_E),$$ where $C_{[\rho]} \in K^*$ is a constant that depends only on $\varphi$ and $[\rho]$. 

    Note also that $$L_S(\varphi^\vee/\mcO_L,0) = c'\Reg(\varphi/\mcO_L)$$ for some $c'\in K^*$.

    Now, comparing both sides and notice that $G$ is finite and there are only finitely many irreducible representations of $G$, the result follows.
\end{proof}

\section{An application in transcendence of special $L$-values}\label{sec:application}
Let $\bA=\FF_q[t]$ and let $K=\FF_q(\theta)$ throughout this section. Let $\varphi$ be a Drinfeld $\bA$-module over $K$ and let $\rho$ be an Artin representation of $G_K$. As an application of the theory of Artin twists, we prove that the special $L$-value $L(\varphi^\vee,\rho,0)$ is transcendental over $K$.

The transcendence of such special $L$-values was previously established by Chang, El-Guindy, and Papanikolas in the case of Dirichlet characters. More precisely, they proved in \cite[Corollary~4.6]{chang2018logalgebraic} that if $\varphi$ is a Drinfeld module defined over $A$, $\mfp\in A$ is a monic irreducible polynomial, and $\chi$ is a Dirichlet character modulo $\mfp$, then $L(\varphi^\vee,\chi,0)$ is transcendental over $K$. Since Dirichlet characters correspond to one-dimensional Artin representations factoring through the Galois groups of Carlitz cyclotomic extensions, our main theorem in this section can be understood as an extension of their result to arbitrary Artin representations.

By Corollary~\ref{corD}, $L(\varphi^\vee,\rho,0)\notin\overline K\iff\Reg(\mcE/A)\notin\overline K$. Therefore, it is enough to prove that the regulator $\Reg(\mcE/A)$ is transcendental over $\overline K$.  Recall that $\Reg(\mcE/A)$ is the determinant of the regulator matrix. Via the isomorphism $\mcE\cong\varphi^{\oplus N}$ over $K^{\sep}$, the logarithm map of $\mcE$ is identified with that of $\varphi^{\oplus N}$, allowing us to express the entries of the regulator matrix in terms of logarithms of $\varphi$. The desired transcendence then follows from the algebraic independence theorem of Chang and Papanikolas \cite[Theorem~1.1.1]{chang2011algebraica}.

Let $$B_\varphi:=\End_E(\varphi):=\left\{u\in\CC_\infty:u\circ\varphi_a=\varphi_a\circ u
\textnormal{ for all }a\in\bA\right\}$$ be the endomorphism ring of $\varphi$. Since $\varphi$ has generic characteristic, $B_\varphi$ is an integral domain. Furthermore,
$1\le \rk_A(B_\varphi)\le \rk(\varphi)$. Let $ H_\varphi:=\Frac(B_\varphi)$.

\begin{thm}[Chang--Papanikolas, {\cite[Theorem~1.1.1]{chang2011algebraica}}]\label{CPthm}
	Let $\varphi$ be a Drinfeld $\bA$-module defined over $K$. Let $u_1,\ldots,u_n\in\CC_\infty$ satisfy $\exp_\varphi(u_i)\in\overline K$ for all $i$.
	If $u_1,\ldots,u_n$ are linearly independent over $H_\varphi$, then they are algebraically independent over $\overline K$.
\end{thm}

\begin{thm}\label{thm:transcendence_criterion}
	Let $\varphi$ be a Drinfeld $\bA$-module over $K$ and let $	\rho:G_K\longrightarrow\GL_n(\overline{\FF}_q)$ be an Artin representation. Then the special $L$-value $L(\varphi^\vee,\rho,0)$ is transcendental over $K$.
\end{thm}

\begin{proof}
Choose an element $b\in A$ such that $\hat{\mcE}:=b^{-1}\mcE b$
is defined over $A$.  Then $\exp_{\hat{\mcE}}=b^{-1}\exp_{\mcE}b$. 
Let $e_1,\cdots,e_N$ be the standard $A$-basis of $\Lie_{\hat{\mcE}}(A)=A^{\oplus N}$. Let $f_1,\cdots,f_N$ be an $A$-basis of $\exp_{\hat{\mcE}}^{-1}(\hat{\mcE}(A))$.
Both bases extend to $K_\infty$-bases of $\Lie_{\hat{\mcE}}(K_\infty)$.  Let $$h:\Lie_{\hat{\mcE}}(K_\infty)\longrightarrow \Lie_{\hat{\mcE}}(K_\infty)$$ be the $K_\infty$-linear map defined by $h(e_j)=f_j$.  Then, $$\Reg(\hat{\mcE}/A)=\frac{\det(h)}{\sgn(\det(h))}.$$
Write $$f_j= \sum_{i=1}^Nf_{ij}e_i = \begin{pmatrix}
f_{1j}\\
\vdots\\
f_{Nj}
\end{pmatrix},
$$ for $f_{ij}\in K_\infty$. Let $$F:=[f_1,\cdots,f_N]\in \Mat_N(K_\infty).$$
Thus the matrix of $h$ with respect to these two bases is $F$ and $\det F\neq 0$ as $F$ is invertible.

Let $P=P(\mathbf{\vec{u}})$ be as in Corollary \ref{P_exp_log}, and put $Q:=Pb\in \GL_N(\conj K)$. Since $\mcE_t=P^{-1}\varphi_t^{\oplus N}P$, we have $$\hat{\mcE}_t=Q^{-1}\varphi_t^{\oplus N}Q.$$
Hence the exponential maps satisfy
$$\exp_{\hat{\mcE}}(\bx)=Q^{-1}\exp_{\varphi^{\oplus N}}(Q\bx)$$
for all $\bx\in \Lie_{\hat{\mcE}}(\CC_\infty)$.  Define $z_j:=Qf_j$ for each $1\leq j\leq N$ and put $$Z:=[z_1,\cdots,z_N]=QF.$$
Since $f_j\in \exp_{\hat{\mcE}}^{-1}(\hat{\mcE}(A))$, we have
$$\exp_{\varphi^{\oplus N}}(z_j)=Q\exp_{\hat{\mcE}}(f_j)\in \varphi^{\oplus N}(\conj K).$$
Writing $$z_j=\begin{pmatrix}
z_{1j}\\
\vdots\\
z_{Nj}
\end{pmatrix},$$ then we have $$\exp_{\varphi^{\oplus N}}(z_j) =\exp_{\varphi^{\oplus N}}(Qf_j) =   Q\exp_{\hat{\mcE}}(f_j)\in \varphi^{\oplus N}(\conj{K})$$ as $\exp_{\hat{\mcE}}(f_j)\in \hat{\mcE}(A)$ by definition. Thus, $$\begin{pmatrix}
\exp_\varphi(z_{1j})\\
\vdots\\
\exp_\varphi(z_{Nj})
\end{pmatrix} = \exp_{\varphi^{\oplus N}}(z_j)\in \varphi^{\oplus N}(\conj K)$$ and we get$$\exp_\varphi(z_{ij})\in \conj K$$
for all $1\leq i,j\leq N$.

Let $$\mcZ_0:=\{z_{ij}:1\leq i,j\leq N\}$$ and let $$V:=\Span_{H_\varphi}(\mcZ_0).$$
Choose an $H_\varphi$-basis $\{u_1,\cdots,u_m\}\subseteq \mcZ_0$ of $V$.  Since each $u_k$, $1\leq k\leq m$, is one of the $z_{ij}$, we have $$\exp_\varphi(u_k)\in \conj K$$
for all $1\leq k\leq m$.  By Theorem \ref{CPthm}, $u_1,\cdots,u_m$ are algebraically independent over $\conj K$.

For every $1\leq i,j\leq N$, we can find $c_{ij,k}\in H_\varphi\subseteq \conj K$ such that $$z_{ij}=\sum_{k=1}^m c_{ij,k}u_k.$$
Put $$L_{ij}(X_1,\cdots,X_m):=\sum_{k=1}^m c_{ij,k}X_k\in \conj K[X_1,\cdots,X_m]$$ and $$g(X_1,\cdots,X_m):=\det(L_{ij}(X_1,\cdots,X_m))_{1\leq i,j\leq N}.$$
Then, we see that $$\det Z=g(u_1,\cdots,u_m).$$
Moreover, $$\det Z=\det(Q)\det(F)\neq 0.$$
Hence $g$ is not the zero polynomial.  Since all $L_{ij}$ are homogeneous linear forms, $g$ is homogeneous of degree $N>0$.  Thus, $g$ is non-constant.

We claim that $\det Z$ is transcendental over $K$. Argue by contradiction, we assume that $\det Z$ is algebraic over $K$, say $\det Z=\alpha\in \conj K$, then $$g(X_1,\cdots,X_m)-\alpha$$ is a non-zero polynomial over $\conj K$ vanishing at $(u_1,\cdots,u_m)$, contradicting the algebraic independence of $u_1,\cdots,u_m$.  Thus $\det Z$ is transcendental over $K$.  Since $\det(Q)\in \conj K^\times$, $\det F=\det(Q)^{-1}\det Z$ is also transcendental over $K$.  Consequently $\Reg(\hat{\mcE}/A)$ is transcendental over $K$.

Finally, Corollary \ref{corD} gives $$N_{K_\infty(\rho)/K_\infty}(L_S(\varphi^\vee,\rho,0))
=C\cdot \Reg(\hat{\mcE}/A)\cdot h(\hat{\mcE}/A)$$
with $C\in K^\times$ and $h(\hat{\mcE}/A)\in A-\{0\}$.  Hence the norm on the left hand side is transcendental over $K$.  If $L(\varphi^\vee,\rho,0)$ were algebraic over $K$, then this norm would be algebraic over $K$, a contradiction.  Therefore, $L(\varphi^\vee,\rho,0)$ must be transcendental over $K$.
\end{proof}

\section{Examples}\label{sec:example}
\subsection{The $n$-th power extension}
In this section, let $n$ be a positive integer that divides $q-1$. We assume $\bA = \FF_q[t]$ and $A = \FF_q[\theta]$. Let $f\in A$ be an $n$-th power free polynomial of degree $d$. Then the $n$-power residue symbol $\chi_f = \left(\frac{f}{\cdot}\right)_n$ induces a character on $\Gal(K(\sqrt[n]{f})/K)$, still denoted by $\chi_f$ and called the $n$-th power character of $\Gal(K(\sqrt[n]{f})/K)$.

Consider the Carlitz module $$\begin{aligned}
	C: \bA&\to A[\tau]\\
	t&\mapsto \theta + \tau.
\end{aligned}$$ We can form the $\bA$-motive $\MM(C,\chi_f)$ and it has a model $\mcE = \mcE(C,\chi_f)$ defined over $A$, which satisfies $\mcE_t = \sqrt[n]{f}^{-1} C_t \sqrt[n]{f}$, i.e. $$\begin{aligned}
	\mcE: \sfA &\to A[\tau],\\
	t&\mapsto \theta + f^\frac{q-1}{n}\tau.
\end{aligned}$$
Recall that $\dis \mcL(\mcE/A) = \prod_{\mfp} \frac{[\Lie_{\mcE_\mfp}(\FF_\mfp)]_\bA}{[\mcE_\mfp(\FF_\mfp)]_\bA}$, where $\mfp$ runs through all monic irreducible polynomials of $A$ and $\mcE_\mfp$ denotes the reduction of $\mcE$ modulo $\mfp$. Then, we have the following result.
\begin{prop}
	$\mcL(\mcE/A) = L(\chi_f,1)$.
\end{prop}

\begin{proof}
	For an element $a\in \bA$ of degree $d$, if we write $C_a = a(\theta) + \alpha_1\tau +\cdots +\alpha_{d-1}\tau^{d-1} + \tau^d$, then 
	\[ \mcE_a = a(\theta) + f^{\frac{q-1}{n}}\alpha_1\tau +\cdots + f^{\frac{q^{d-1}-1}{n}}\alpha_{d-1}\tau^{d-1} + f^{\frac{q^d-1}{n}}\tau^d. \]
	Thus, for any $\mfp \nmid f$, the reduction $\mcE_\mfp$ of $\mcE$ modulo $\mfp$ is a Drinfeld module over $\FF_\mfp$ of rank 1, and it induces an $\bA$-module structure on $A/\mfp = \FF_\mfp$, which we denote by $\mcE_\mfp(\FF_\mfp)$.
	
	Since $C_\mfp \equiv \tau^{\deg \mfp} \pmod \mfp$, we see that 
	\[ \mcE_\mfp \equiv \left(\frac{f}{\mfp}\right)_n \tau^{\deg \mfp} \pmod \mfp = \chi_f(\mfp) \tau^{\deg \mfp} \pmod \mfp. \]
	So, for any $x\in \FF_\mfp$, we have 
	\[ (\mfp - \chi_f(\mfp))(x) = \chi_f(\mfp)(x^{q^{\deg \mfp}} - x) = 0. \]
	
	Let $P(X) = P_\mfp(\mcE_\mfp, X)$ be the characteristic polynomial of $\tau^{\deg \mfp}$ acting on the Tate module of $\mcE_\mfp$. By \cite[Theorem 4.2.7]{papikian2023drinfeld}, $P(X) = X + a$, where $a = - N_{\FF_\mfp/\FF_q}(\conj{f}^{\frac{q-1}{n}})^{-1} \cdot \mfp = -\left(\frac{f}{\mfp}\right)_n^{-1} \mfp = -\chi_f(\mfp)^{-1} \mfp$. 
	
	Thus, $P(1) = 1 - \chi_f(\mfp)^{-1} \mfp$, and the Fitting ideal is given by 
	\[ \Fitt_\bA(\mcE_\mfp(\FF_\mfp)) = (P(1)) = (\mfp - \chi_f(\mfp)). \]
	It follows that 
	\[ \mcE_\mfp(\FF_\mfp) \cong \frac{\bA}{\mfp - \chi_f(\mfp)} \]
	as $\bA$-modules. Equivalently, taking the monic generator of the ideal, we have 
	\[ [\mcE_\mfp(\FF_\mfp)]_\bA = \mfp - \chi_f(\mfp). \]
	
	Therefore, the special $L$-value is given by the Euler product:
	\[ \mcL(\mcE/A) = \prod_{\mfp} \frac{[\Lie_{\mcE_\mfp}(\FF_\mfp)]_\bA}{[\mcE_\mfp(\FF_\mfp)]_\bA} = \prod_{\mfp} \frac{\mfp}{\mfp - \chi_f(\mfp)} = \prod_{\mfp} \left(1 - \frac{\chi_f(\mfp)}{\mfp}\right)^{-1} = L(\chi_f, 1). \]
\end{proof}
\begin{rmk}
	We have $L(\chi_f,1) = L(\mcE^\vee,0)$.
\end{rmk}  

Recall that for $k\in \ZZ_+$, we define $$[k] = \theta^{q^k} - \theta,$$ $$D_0 = 1, D_k = [k][k-1]^q\cdots [1]^{q^{k-1}},$$ $$L_0 = 1, L_k = [k][k-1]\cdots [1].$$
\begin{prop}
	Let $$\exp_\mcE: = \sum_{k=0}^{\infty} e_k \tau^k\quad\textnormal{ and }\quad  \log_{\mcE} = \sum_{k=0}^{\infty}\ell_k\tau^k$$ be the exponentiation and logarithm of $\mcE$ respectively. Then for all $k\geq 0$, $$e_k = \frac{f^{\frac{q^k-1}{n}}}{D_k}\quad \textnormal{ and } \ell_k = (-1)^k\frac{f^{\frac{q^k-1}{n}}}{L_k}.$$
	
	In particular, the radius of convergence of $\log_\mcE$ is given by $$\rho(\log_\mcE) = q^{1+\frac{1}{q-1}-\frac{d}{n}},$$ where $d = \deg f$.
\end{prop}
\begin{proof}
	The first statement follows from Corollary \ref{P_exp_log}. To compute the radius of convergence, first notice that $|[k]|_\infty = |\theta^{q^k}| = q^{q^k}$ and hence $|L_k|_\infty = q^{q^k+q^{k-1}+\cdots+q} = q^{\frac{q(q^k-1)}{q-1}}$. Also, $|f|_\infty = q^d$. Thus, $$\rho(\log_\mcE)^{-1} = \limsup\limits_{k\to\infty}\left|\frac{f^{\frac{q^k-1}{n}}}{L_k}\right|_\infty^{1/q^k} = \frac{q^{d/n}}{q^{q/q-1}}.$$ Thus, the result follows.
\end{proof}

\begin{cor}
	Let $f\in A$ be an $n$-th power free element of degree $d$.
	
	$(1)$ If $d < n+ \frac{n}{q-1}$, then $\rho(\log_\mcE) > 1$ and if $\deg f > n+ \frac{n}{q-1}$, then $\rho(\log_\mcE)<1$.
	
	$(2)$ Suppose $\deg f \leq n$. Then, the unit module $U(\mcE/A) = A\log_{\mcE}(1)$ and the class module $H(\mcE/A) = 0$. In particular, $$L(\chi_f,1) = \log_{\mcE}(1)=1+ \sum_{k=1}^{\infty}(-1)^k\frac{f^{\frac{q^k-1}{n}}}{L_k}.$$ 
\end{cor}
\begin{proof}
	$(1)$ is straightforward.
	
	$(2)$ Let $\mcO_\infty = \{a\in K_\infty: |a|_\infty\leq 1\}$ be the ring of integers at $\infty$. 
	So, for any $a\in \mcO_\infty$, $\log_\mcE(a)$ converges by $(1)$ and we have $$\exp_\mcE(\log_\mcE(a)) = a$$ and hence $\mcO_\infty\subseteq \exp_\mcE(K_\infty)$. Since $A+\mcO_\infty = K_\infty$, we have $A+\exp_\mcE(K_\infty) = K_\infty$. Thus, the $A$-module $H(\mcE/A) = \frac{\mcE(K_\infty)}{\mcE(A)+\exp_\mcE(K_\infty)} = 0$ and $h(\mcE/A) = [H(\mcE/A)]_\bA = 1$.
	
	We claim that $U(\mcE/A) = \{\alpha\in K_\infty: \exp_\mcE(\alpha)\in A\}$ is generated by $\log_\mcE(1)$. Clearly, $\exp_\mcE(\log_\mcE(1)) = 1\in A$ and hence $\log_\mcE(1)\in U(\mcE/A)$. We know that $U(\mcE/A) = Au$ for some $u\in K_\infty$ as it is an $A$-lattice in $K_\infty$ of rank 1. Then, $\log_\mcE(1) = au$ for some $a\in A$. Taking exponentiation, we get $1 = \exp_\mcE(au) = \mcE_a(\exp_\mcE(u)) = \mcE_a(b)$, where $b = \exp_\mcE(u)\in A$. By looking at the degree of both sides, we must have $\deg a = \deg b = 0$, i.e., $\log_\mcE(1)$ is an $\FF_q^*$-multiple of $u$ and hence $U(\mcE/A) = A\log_\mcE(1)$. Further, $$\log_\mcE(1) = 1+ \sum_{k=1}^{\infty}(-1)^k\frac{f^{\frac{q^k-1}{n}}}{L_k}$$ and $\left|(-1)^k\frac{f^{\frac{q^k-1}{n}}}{L_k}\right|_\infty < 1$ whenever $\deg f \leq n$. Thus, $\log_\mcE(1)$ has sign 1 and by definition $[\Lie_\mcE(A):U(\mcE/A)]_\bA = \log_\mcE(1)$. By Taelman's class number formula, we see that $$L(\chi_f,1) = \mcL(\mcE/A) = [\Lie_\mcE(A):U(\mcE/A)]_\bA\cdot h(\mcE/A) = \log_\mcE(1).$$
\end{proof}

\subsection{$S_3$-extension}
In this section we assume $d = d_\rho = [\FF_q(\rho):\FF_q] = 1$. Then, we have $$\operatorname{Sol}_E(\rho) = \{\vec{\mathbf{x}}\in \Mat_{n\times 1}(L): g(\vec{\mathbf{x}}) = \rho(g)^{-1}\vec{\mathbf{x}}: \forall g\in G\}.$$
We need the following Lemma to construct our example.

\begin{lem}\label{d=1_nonzero_is_fundamental}
	Suppose $\rho$ is irreducible of dimension $n$. Let $\vec{\mathbf{u}}\in \operatorname{Sol}_E(\rho)$ be a non-zero element, then $\vec{\mathbf{u}}$ is a fundamental solution.
\end{lem}
\begin{proof}
	Write $\vec{\mathbf{u}} = (u_1,u_2,\cdots,u_n)^\top$. Assume the contrary. Then there exist a nonzero vector $\vec{\mathbf{c}} = (c_1,\cdots,c_n)^\top\in \FF_q^n$ such that $\vec{\mathbf{c}}^\top\vec{\mathbf{u}} = c_1u_1+\cdots+c_nu_n = 0$. Let $$H_{\vec{\mathbf{u}}} =\{\vec{\mathbf{c}}=(c_1,\cdots,c_n)^\top\in \FF_q^n:\vec{\mathbf{c}}^\top\vec{\mathbf{u}}=0 \}\subseteq V_\rho$$ be the $\FF_q$-subspace defined by $\vec{\mathbf{u}}$.
	
	Let $\rho^\vee$ be the dual representation of $\rho$, i.e. $\rho^\vee(g) = (\rho(g)^{-1})^\top$. Then, for any $\vec{\mathbf{c}}\in H_{\vec{\mathbf{u}}}$, we have $$ (\rho^\vee(g)\vec{\mathbf{c}})^\top\vec{\mathbf{u}} = \vec{\mathbf{c}}^\top\rho(g)^{-1}\vec{\mathbf{u}} = \vec{\mathbf{c}}^\top g(\vec{\mathbf{u}}) = g(\vec{\mathbf{c}}^\top \vec{\mathbf{u}}) = 0.$$ Thus, $\rho^\vee(g)\vec{\mathbf{c}}\in H_{\vec{\mathbf{u}}}$. It follows that $H_{\vec{\mathbf{u}}}$ is a $G$-invariant proper subspace of $V_{\rho^\vee}$. Since $\rho$ is irreducible, its dual representation $\rho^\vee$ is also irreducible. This contradicts the irreducibility of $\rho$. Hence, $u_1,u_2,\cdots,u_n$ must be linearly independent over $\FF_q$.
\end{proof}
\begin{cor}
	Suppose $\rho$ is irreducible of dimension $n$ and factors through a finite quotient group $G$. Choose a vector $\vec{\mathbf{y}}\in \Mat_{n\times 1}(E^\sep)$ such that $$\vec{\mathbf{u}}:= \sum_{g\in G}\rho(g)g(\vec{\mathbf{y}})\in \Mat_{n\times 1}(E^\sep)$$ is non-zero. Then, $\vec{\mathbf{u}}$ is a fundamental solution.
\end{cor}
\begin{proof}
	For any $h\in G$, we have $$\begin{aligned}
		h(\vec{\mathbf{u}}) & = h\left(\sum_{g\in G}\rho(g)g(\vec{\mathbf{y}})\right)  = \sum_{g\in G}\rho(g)(hg)(\vec{\mathbf{y}})  = \sum_{\sigma\in G}\rho(h^{-1}\sigma)\sigma(\vec{\mathbf{y}})\\
		& = \sum_{\sigma\in G}\rho(h)^{-1}\rho(\sigma)\sigma(\vec{\mathbf{y}}) = \rho(h)^{-1}\left(\sum_{\sigma\in G}\rho(\sigma)\sigma(\vec{\mathbf{y}})\right)\\
		& = \rho(h)^{-1}\vec{\mathbf{u}}.
	\end{aligned}$$
	Thus, $\vec{\mathbf{u}}\in \operatorname{Sol}_E(\rho)$ and by Lemma \ref{d=1_nonzero_is_fundamental}, the result follows.
\end{proof}

Assume $\text{char}(\FF_q) \neq 2, 3$. Let $A = \FF_q[\theta], K = \FF_q(\theta)$. Consider the irreducible cubic polynomial:
\[
f(x) = x^3 + \theta x + 1 \in A[x].
\]
It has discriminant $\Delta = -4\theta^3 - 27$. Let $L$ be the splitting field of $f$ over $K$. The Galois group is $G = \Gal(L/K) \cong S_3$. Let $\zeta_1, \zeta_2, \zeta_3\in \mcO_L$ be roots of $f(x)$. Then they satisfy
\[
\zeta_1 + \zeta_2 + \zeta_3 = 0.
\]

Let $\rho: G \to \GL_2(\conj{\FF}_q)$ be the standard representation defined on the generators $r = (123)$ and $s = (12)$ as follows:
\[
\rho(r) = \begin{pmatrix} 0 & -1 \\ 1 & -1 \end{pmatrix}, \quad 
\rho(s) = \begin{pmatrix} -1 & 1 \\ 0 & 1 \end{pmatrix}.
\]
We can extend $\rho$ to the absolute Galois group $\Gal(K^\sep/K)$ by sending elements outside $G$ to identity and obtain an Artin representation.

We record the following information:
$$\begin{aligned}
	& \rho(\id) = \begin{pmatrix}
		1 & 0\\
		0 & 1\\
	\end{pmatrix}\quad  &\rho(s) = \begin{pmatrix}
		-1 & 1\\
		0 & 1\\
	\end{pmatrix} \quad & \rho(r) = \begin{pmatrix} 0 & -1 \\ 1 & -1 \end{pmatrix}\\
	& \rho(r^2) = \begin{pmatrix} 
		-1 & 1 \\ -1 &0\end{pmatrix} \quad  & \rho(sr) = \begin{pmatrix} 
		1 & 0 \\ 1 & -1 \end{pmatrix} \quad & \rho(sr^2) = \begin{pmatrix} 0 & -1 \\ -1 & 0 \end{pmatrix}
\end{aligned}$$
Let $\vec{\mathbf{y}} = \begin{pmatrix}
	\zeta_1\\
	0 
\end{pmatrix}$. Then, a straightforward computation yields that $$\vec{\mathbf{u}}:= \frac{1}{3}\sum_{g\in G}\rho(g)g(\vec{\mathbf{y}}) = \begin{pmatrix}
	\zeta_1\\
	-\zeta_3
\end{pmatrix}.$$

Thus, $M(\vec{\mathbf{u}}) = [\vec{\mathbf{u}},\vec{\mathbf{u}}^{(1)}] = \begin{pmatrix}
	\zeta_1 & \zeta_1^q\\
	-\zeta_3 & -\zeta_3^q
\end{pmatrix}$ and $$\begin{pmatrix}
	f_0\\
	f_1
\end{pmatrix} = M(\vec{\mathbf{u}})^{-1}\vec{\mathbf{u}}^{(2)} = \frac{1}{\zeta_1^q\zeta_3-\zeta_1\zeta_3^q}\begin{pmatrix}
	\zeta_1^q\zeta_3^{q^2}-\zeta_1^{q^2}\zeta_3^q\\
	\zeta_1^{q^2}\zeta_3-\zeta_1\zeta_3^{q^2}
\end{pmatrix} $$

We also checked that $f_0,f_1\in K$ by brute force.

Now, we see that $$\dis \Phi = \begingroup 
\renewcommand{\arraystretch}{2.2} 
\begin{pmatrix}
	0 & \dfrac{\zeta_1^q\zeta_3^{q^2}-\zeta_1^{q^2}\zeta_3^q}{\zeta_1^q\zeta_3-\zeta_1\zeta_3^q} \\
	1 & \dfrac{\zeta_1^{q^2}\zeta_3-\zeta_1\zeta_3^{q^2}}{\zeta_1^q\zeta_3-\zeta_1\zeta_3^q}
\end{pmatrix}
\endgroup\quad\textnormal{ and }\quad \Psi = (\Phi^\top)^{-1}= \begin{pmatrix}
	\displaystyle \frac{\zeta_1 \zeta_3^{q^2} - \zeta_1^{q^2} \zeta_3}{\zeta_1^q \zeta_3^{q^2} - \zeta_1^{q^2} \zeta_3^q} & \displaystyle \frac{\zeta_1^q \zeta_3 - \zeta_1 \zeta_3^q}{\zeta_1^q \zeta_3^{q^2} - \zeta_1^{q^2} \zeta_3^q} \\[15pt]
	1 & 0
\end{pmatrix}.$$
Set $\mcE' = \mcE'(C,\rho)$, where $C$ is the Carlitz module. By Theorem \ref{thmB}, it has a model over $K$: $$\mcE'_t = \begin{pmatrix}
	\theta & 0\\
	0 & \theta
\end{pmatrix} + \begin{pmatrix}
	\displaystyle \frac{\zeta_1 \zeta_3^{q^2} - \zeta_1^{q^2} \zeta_3}{\zeta_1^q \zeta_3^{q^2} - \zeta_1^{q^2} \zeta_3^q} & \displaystyle \frac{\zeta_1^q \zeta_3 - \zeta_1 \zeta_3^q}{\zeta_1^q \zeta_3^{q^2} - \zeta_1^{q^2} \zeta_3^q} \\[15pt]
	1 & 0
\end{pmatrix}\tau\in \Mat_{2\times 2}(K)[\tau].$$

Now, let $q=5$. Then, using MAGMA, we get $$\begin{pmatrix}
	f_0\\
	f_1
\end{pmatrix} = \begin{pmatrix}
	4\theta^6 + 4\theta^3 + 1\\
	\theta^{10} + \theta^4 + 2\theta
\end{pmatrix}\in \Mat_{2\times 1}(A).$$
Moreover, 
$$\mcE'_t = \begin{pmatrix}
	\theta & 0\\
	0 & \theta
\end{pmatrix} + \begin{pmatrix}
	\displaystyle -\frac{f_1}{f_0} & \displaystyle \frac{1}{f_0} \\[15pt]
	1 & 0
\end{pmatrix}\tau\in \Mat_{2\times 2}(K)[\tau].$$

By clearing denominators, we get an integral model $\mcE/A$: $$\mcE_t = \begin{pmatrix}
	\theta & 0\\
	0 & \theta
\end{pmatrix} + \begin{pmatrix}
	\displaystyle -f_1f_0^3 & \displaystyle f_0^3 \\[15pt]
	f_0^4 & 0
\end{pmatrix}\tau\in \Mat_{2\times 2}(A)[\tau].$$


	By Theorem \ref{thmC}, there is a finite set $S$ of places such that $$L_S(\rho,s+1) = L_S(C^\vee,\rho,s) = L_S(\mcE^\vee,s).$$ Evaluating at $s = 0$, we get $$L(\rho,1) = c\cdot L(\mcE^\vee,0) = c'\cdot \Reg(\mcE/A) \quad \textnormal{ for some }c, c'\in K^*.$$
	\subsection{Carlitz cyclotomic extension}
	Let $\bA = \FF_q[t], A = \FF_q[\theta]$, $K = \FF_q(\theta)$. Let $f\in A$ be an irreducible polynomial of degree $d$ and $L = K(C[f]) = K(\lambda_f)$, where $\lambda_f = \exp_C(\widetilde{\pi}/f)$. Then $G = \Gal(L/K) \stackrel{\sim}{\to}(A/f)^*$ of order $\ell = q^d-1$. We see that $[L:K] = \ell = q^d-1\nmid q-1$ if and only if $d>1$.
	
	Fix an element $a\in \bA$ such that $\gcd(a,f) = 1$, and set $\sigma = \sigma_a\in G$ such that $$\sigma_a(\lambda_f) = C_a(\lambda_f),$$ where we identify $a\in \bA$ with its image in $A$ via $t\mapsto \theta$. Then, $G = \langle\sigma\rangle$. Let $\xi\in \conj{\FF}_q$ be a root of $f$ and consider the evaluation map $$\ev: A/f \to \conj{\FF}_q,$$ $$\conj{a}\mapsto a(\xi).$$
	We then have a character $$\chi: \Gal(L/K)\cong (A/f)^*\to \conj{\FF}_q^*,$$ $$\sigma_a\mapsto a(\xi).$$ Since the map $\ev$ is injective, we see that $\FF_q(\chi)$ contains a primitive root of unity $\zeta_\ell$ and hence $d_\chi = [\FF_q(\chi):\FF_q] = [\FF_q(\zeta_\ell):\FF_q]$ is the order of $q$ mod $\ell$, hence equals to $d$.
	\begin{eg}
		To ease the computation, we assume that $q = 2$ and take $f = \theta^2+1$. Denote by $\xi = \sqrt{-1}\in\conj{\FF}_2$ a root of $f$. Then, $\sigma_\theta$ is a generator of the cyclic group $\Gal(L/K)$. 
		
		Choose $u_1 = \lambda_f$ and $u_2 = \theta\lambda_f+\lambda_f^2 = C_t(\lambda_f) = \sigma_\theta(\lambda_f)$. Thus, we see that $\sigma_\theta(u_1) = u_2$ and $\sigma_\theta(u_2) = \sigma_{\theta^2}(u_1) = C_{t^2}(\lambda_f) = u_1$ as $C_f(\lambda_f) = 0$. Consider the $\FF_q$-basis $\{1,\sqrt{-1}\}$ for $\FF_q(\chi) = \FF_q(\sqrt{-1})$ and take $\vec{\boldsymbol{\alpha}} = (1,\sqrt{-1})^\top$. Then, one immediately sees that $\vec{\mathbf{u}}: = (u_1,u_2)$ is a solution of the system of equations $$\begin{cases}
			(g(u_1),g(u_2))\begin{pmatrix}
				1\\
				\sqrt{-1}
			\end{pmatrix} = \chi^{-1}(g) (u_1, u_2)\begin{pmatrix}
				1\\
				\sqrt{-1}
			\end{pmatrix}\\
			(g(u_1),g(u_2))\begin{pmatrix}
				1\\
				\sqrt{-1}^q
			\end{pmatrix} = \chi^{-q}(g) (u_1, u_2)\begin{pmatrix}
				1\\
				\sqrt{-1}^q
			\end{pmatrix}\\
		\end{cases}.$$ In other words, $\vec{\mathbf{u}}\in \operatorname{Sol}_E(\chi,\vec{\boldsymbol{\alpha}})$. Furthermore, $u_1,u_2$ are obviously linearly independent over $\FF_q$, hence a fundamental solution.
		
		Use our notation in previous sections, $N = 2$ and consider $\vec{\mathbf{u}}$ as a column vector, we have $$M(\vec{\mathbf{u}}) = [\vec{\mathbf{u}},\vec{\mathbf{u}}^{(1)}] = \begin{pmatrix}
			\lambda_f & \lambda_f^2\\
			\theta\lambda_f+\lambda_f^2 & \theta^2\lambda_f^2+\lambda_f^4
		\end{pmatrix}.$$ Use $C_f(\lambda_f) = 0$ or equivalently $$\lambda_f^4+(\theta^2+\theta)\lambda_f^2+ (\theta^2+1)\lambda_f = 0,$$
		we get 
		$$\begin{pmatrix}
			f_0\\
			f_1
		\end{pmatrix} = M(\vec{\mathbf{u}})^{-1}\vec{\mathbf{u}}^{(2)} = \begin{pmatrix}
			\theta^2+1\\
			\theta^2+\theta
		\end{pmatrix}\in \Mat_{2\times 1}(A).$$
		
		Moreover, set $\mcE' = \mcE'(C,\chi)$, by Theorem \ref{thmB}
		$$\mcE'_t = \begin{pmatrix}
			\theta & 0\\
			0 & \theta
		\end{pmatrix} + \begin{pmatrix}
			\displaystyle \frac{f_1}{f_0} & \displaystyle \frac{1}{f_0} \\[15pt]
			1 & 0
		\end{pmatrix}\tau\in \Mat_{2\times 2}(K)[\tau].$$
		
		By clearing denominators, we get an integral model $\mcE/A$: $$\mcE_t = \begin{pmatrix}
			\theta & 0\\
			0 & \theta
		\end{pmatrix} + \begin{pmatrix}
			\displaystyle f_1 & \displaystyle 1 \\[15pt]
			f_0 & 0
		\end{pmatrix}\tau\in \Mat_{2\times 2}(A)[\tau].$$
		
		By Theorem \ref{thmC}, there is a finite set $S$ of places such that $$L_S(\chi,s+1) = L_S(C^\vee,\chi,s) = L_S(\mcE^\vee,s).$$ Evaluating at $s = 0$, we get $$L(\chi,1) = c\cdot L(\mcE^\vee,0) = c'\cdot \Reg(\mcE/A) \quad \textnormal{ for some }c, c'\in K^*.$$
	\end{eg}
\bibliographystyle{alpha}
\bibliography{Bibliography/FField}



\end{document}